\documentclass[12pt,reqno]{amsart}
\usepackage{tikz,graphicx}
\usetikzlibrary{arrows.meta}

\usepackage[a4paper, 
 top=1in, bottom=1in, left=0.8in, right=0.8in
]{geometry}

\usepackage{pinlabel}
\usepackage{amssymb,amsfonts,amsmath}
\usepackage{multirow}
\usepackage{colortbl}
\usepackage{enumerate}
\usepackage{soul}
\usepackage{float}
\usepackage[caption = false]{subfig}
\usepackage{mathrsfs,mathtools}
\usepackage{todonotes,booktabs}
\usepackage{stmaryrd}
\usepackage{marvosym}
\usepackage{graphicx}
\usepackage{pdflscape}

\usepackage{microtype}
\usepackage[all]{xy}
\usepackage{tikz-cd}
\usepackage[capitalise]{cleveref}
\usepackage{marginnote}
\usepackage{todonotes}
\usepackage{longtable}
\usepackage{lscape}
\newtheorem{thm}{Theorem}[section]
\newtheorem*{thm*}{Theorem}

\newtheorem*{cor*}{Corollary}
\newtheorem*{prop*}{Proposition}
\newtheorem{cor}[thm]{Corollary}
\newtheorem*{notation*}{Notation}
\newtheorem{example}[thm]{Example}
\newtheorem{defn}[thm]{Definition}
\newtheorem*{defn*}{Definition}
\newtheorem{prop}[thm]{Proposition}
\newtheorem{lem}[thm]{Lemma}

\newtheorem*{conj*}{Conjecture}
\newtheorem*{quest*}{Question}

\theoremstyle{definition}
\newtheorem{rem}[thm]{Remark}

\newtheorem{conv}[thm]{Convention}

\newcommand{\bZ}{\mathbf z}

\newcommand{\cvec}{\mathbf c}
\definecolor{darkgreen}{RGB}{0,150,0}

\newcommand{\gamgamunco}[5]{
    \draw (-.75,-.75)--(.75,.75);
    \draw (-.75,.75)--(.75,-.75);
    \draw[fill=white] (-.75,-.75) circle (.28);
    \draw[fill=white] (-.75,.75) circle (.28);
    \draw[fill=white] (.75,.75) circle (.28);
    \draw[fill=white] (.75,.-.75) circle (.28);
    \node at (-.75,-.75) {$#1$};
    \node at (-.75,.75) {$#2$};
    \node at (.75,.75) {$#3$};
    \node at (.75,.-.75) {$#4$};
    \path[fill=white] (0,0) circle (.4);
    \node at (0,0) {$#5$};
    \node at (0,1) { };}

\newcommand{\spokea}[2]{
  \draw[line width=0.5mm, #1] (-.75,-.75)--(-.25,-.25);
  \draw[line width=0.5mm, #1,fill=white] (-.75,-.75) circle (.28);
  \node at (-.75,-.75) {$#2$};
}
\newcommand{\spokeb}[2]{
  \draw[line width=0.5mm, #1] (-.75,.75)--(-.25,.25);
  \draw[line width=0.5mm, #1,fill=white] (-.75,.75) circle (.28);
  \node at (-.75,.75) {$#2$};
}
\newcommand{\spokec}[2]{
  \draw[line width=0.5mm, #1] (.75,.75)--(.25,.25);
  \draw[line width=0.5mm, #1,fill=white] (.75,.75) circle (.28);
  \node at (.75,.75) {$#2$};
}
\newcommand{\spoked}[2]{
  \draw[line width=0.5mm, #1] (.75,-.75)--(.25,-.25);
  \draw[line width=0.5mm, #1,fill=white] (.75,.-.75) circle (.28);
  \node at (.75,-.75) {$#2$};
}
\newcommand{\spokepa}{\node at (-.75,-.75) {$+$};}
\newcommand{\spokepb}{\node at (-.75,.75) {$+$};}
\newcommand{\spokepc}{\node at (.75,.75) {$+$};}
\newcommand{\spokepd}{\node at (.75,-.75) {$+$};}
\newcommand{\cv}[1]{
  \path[fill=white] (0,0) circle (.25);
  \node at (0,0) {$z_i,z_j,#1$};
}
\newcommand{\rcent}{
  \path[fill=white] (0,0) circle (.25);
  \node at (0,0) {$z_i,z_j$};
}

\newcommand{\auxybelhs}[4]{\begin{tikzpicture}[baseline=(current bounding box.center),scale=.86,every node/.style={scale=.78}]
  \draw (0,1) to [out=0,in=180] (2,3) to (4,3);
  \draw (0,3) to [out=0,in=180] (2,1) to (4,1);
  \draw (3,0) to (3,4);
  \draw[fill=white] (0,1) circle (.3);
  \draw[fill=white] (0,3) circle (.3);
  \draw[fill=white] (3,4) circle (.3);
  \draw[fill=white] (4,3) circle (.3);
  \draw[fill=white] (4,1) circle (.3);
  \draw[fill=white] (3,0) circle (.3);
  \node at (0,1) {$a$};
  \node at (0,3) {$b$};
  \node at (3,4) {$c$};
  \node at (4,3) {$d$};
  \node at (4,1) {$e$};
  \node at (3,0) {$f$};
  \path[fill=white] (1,2) circle (.55);
  \node at (1,2) {$#1$};
  \path[fill=white] (3,3) circle (#4);
  \node at (3,3) {$#2$};
  \path[fill=white] (3,1) circle (#4);
  \node at (3,1) {$#3$};
\end{tikzpicture}}

\newcommand{\auxyberhs}[4]{\begin{tikzpicture}[baseline=(current bounding box.center),scale=.86,every node/.style={scale=.78}]
  \draw (0,1) to (2,1) to [out=0,in=180] (4,3);
  \draw (0,3) to (2,3) to [out=0,in=180] (4,1);
  \draw (1,0) to (1,4);
  \draw[fill=white] (0,1) circle (.3);
  \draw[fill=white] (0,3) circle (.3);
  \draw[fill=white] (1,4) circle (.3);
  \draw[fill=white] (4,3) circle (.3);
  \draw[fill=white] (4,1) circle (.3);
  \draw[fill=white] (1,0) circle (.3);
  \node at (0,1) {$a$};
  \node at (0,3) {$b$};
  \node at (1,4) {$c$};
  \node at (4,3) {$d$};
  \node at (4,1) {$e$};
  \node at (1,0) {$f$};
  \path[fill=white] (3,2) circle (.6);
  \node at (3,2) {$#1$};
  \path[fill=white] (1,1) circle (#4);
  \node at (1,1) {$#2$};
  \path[fill=white] (1,3) circle (#4);
  \node at (1,3) {$#3$};
\end{tikzpicture}}

\newcommand{\exampleboundarynode}[3]{\draw[fill=white] (#1,#2) circle (.3); \node at (#1,#2) {$#3$};}
\newcommand{\examplevertex}[3]{\path[fill=white] (#1,#2) circle (.42); \node at (#1,#2) {$#3$};}
\newcommand{\examplelabel}[3]{\node[fill=white,inner sep=1pt] at (#1,#2) {$#3$};}

\newcommand{\ExampleLHSState}{
\begin{tikzpicture}[baseline=(current bounding box.center),scale=.72,every node/.style={scale=.72}]
  \draw (0,1) to [out=0,in=180] (2,3) to (4,3);
  \draw (0,3) to [out=0,in=180] (2,1) to (4,1);
  \draw (3,0) to (3,4);
  \exampleboundarynode{0}{1}{+}
  \exampleboundarynode{0}{3}{+}
  \exampleboundarynode{3}{4}{2}
  \exampleboundarynode{4}{3}{+}
  \exampleboundarynode{4}{1}{1}
  \exampleboundarynode{3}{0}{1}
  \examplevertex{1}{2}{R}
  \examplevertex{3}{3}{A(2)}
  \examplevertex{3}{1}{D(1)}
  \examplelabel{2.35}{3.25}{+}
  \examplelabel{2.35}{.75}{+}
  \examplelabel{3.28}{2}{2}
\end{tikzpicture}}

\newcommand{\ExampleRHSStateOne}{
\begin{tikzpicture}[baseline=(current bounding box.center),scale=.72,every node/.style={scale=.72}]
  \draw (0,1) to (2,1) to [out=0,in=180] (4,3);
  \draw (0,3) to (2,3) to [out=0,in=180] (4,1);
  \draw (1,0) to (1,4);
  \exampleboundarynode{0}{1}{+}
  \exampleboundarynode{0}{3}{+}
  \exampleboundarynode{1}{4}{2}
  \exampleboundarynode{4}{3}{+}
  \exampleboundarynode{4}{1}{1}
  \exampleboundarynode{1}{0}{1}
  \examplevertex{1}{1}{D(1)}
  \examplevertex{1}{3}{A(2)}
  \examplevertex{3}{2}{R}
  \examplelabel{2.08}{1.25}{1}
  \examplelabel{2.08}{3.25}{+}
  \examplelabel{1.28}{2}{2}
\end{tikzpicture}}

\newcommand{\ExampleRHSStateTwo}{
\begin{tikzpicture}[baseline=(current bounding box.center),scale=.72,every node/.style={scale=.72}]
  \draw (0,1) to (2,1) to [out=0,in=180] (4,3);
  \draw (0,3) to (2,3) to [out=0,in=180] (4,1);
  \draw (1,0) to (1,4);
  \exampleboundarynode{0}{1}{+}
  \exampleboundarynode{0}{3}{+}
  \exampleboundarynode{1}{4}{2}
  \exampleboundarynode{4}{3}{+}
  \exampleboundarynode{4}{1}{1}
  \exampleboundarynode{1}{0}{1}
  \examplevertex{1}{1}{A(1)}
  \examplevertex{1}{3}{D(1)}
  \examplevertex{3}{2}{R}
  \examplelabel{2.08}{1.25}{+}
  \examplelabel{2.08}{3.25}{1}
  \examplelabel{1.28}{2}{1}
\end{tikzpicture}}

\author{Yingzi Yang}\address{Department of Mathematics, Stanford University}\email{yyingzi@stanford.edu}

\title{Closed colored models and Demazure crystals}{}

\begin{document}

\maketitle
\begin{abstract}
We construct and study solvable lattice models whose partition functions are Demazure characters. Two types of colored generalizations of the five-vertex model can be defined that are called the open and closed models. States of both models are associated with Gelfand-Tsetlin patterns, or equivalently, semistandard Young tableaux. We characterize the tableaux arising from states of the closed model and prove that they form a Demazure crystal. This is stronger than identifying the partition function as a Demazure character: it requires constructing Kashiwara crystal operators on the states through lattice moves and recoloring. For the open model, Brubaker, Bump, Buciumas, and Gustafsson showed that the states realize (crystal) Demazure atoms. We also give a new, independent proof of their result.

Bump and Naprienko considered bosonic models involving a parameter $t$. In these models, the edges and vertices are constructed by fusion of simpler edge and vertex types. Taking $t=0$ produces bosonic analogs of the open model. By reversing the fusion order, we construct closed bosonic models and establish their Yang-Baxter equations, yielding a Demazure–Lusztig recursion for the partition functions. At $t=0$, these models specialize to bosonic analogues of the closed five-vertex model, and we prove that their states form Demazure crystals.
\end{abstract}
\tableofcontents

\section{Introduction}

\subsection{Motivation and main results}
Lattice models have their origins in the study of physical systems, such as crystal ice \cite{onsager1944crystal}, and have found applications across various areas of mathematics; for example, \cite{andrews1984eight,jones1989knot,korepin1997quantum}, to name a few. Notable examples include Kuperberg’s proof of the enumeration for alternating sign matrices \cite{Kuperberg}. One particularly significant model is the eight-vertex model, which was solved by Baxter \cite{baxter} using the Yang–Baxter equation and a train argument.

The six-vertex and five-vertex models are well-known examples of solvable lattice models. Each model admits a colored variant, in which the partition function depends on a permutation that records the ordering of colors in the boundary conditions. In this paper, we study the closed colored five-vertex models, introduced by Buciumas and Scrimshaw \cite{buciumas2022quasi}. Their open counterparts, the open colored five-vertex models, were introduced and studied in \cite{BBBG}. Constructions of both models will be recalled in \cref{sec:model}.

The Tokuyama model is a solvable lattice model $\mathfrak{S}_{\lambda}({\bf z};q)$ depending on a partition $\lambda=(\lambda_1,\lambda_2,...,\lambda_r)$, parameters ${\bf z}=(z_1,...,z_r)$, and a deformation parameter $q$ \cite{tokuyama,hamel2007bijective, brubaker2011schur}. When $q=0$, the five-vertex variant has a partition function equal to ${\bf z}^{\rho}s_\lambda({\bf z})$, where $\rho=(r-1,r-2,...,1,0),$ and $s_\lambda$ is the Schur polynomial. The closed colored five-vertex models may also be viewed as colored variants of the models $\mathfrak{S}_{\lambda}({\bf z};0)$. As mentioned above, coloring introduces a permutation $w$ into the boundary conditions.  We will denote these colored models by $\mathfrak{S}_{\lambda,w}^\bullet({\bf z})$ or simply $\mathfrak{S}_{\lambda,w}^\bullet$ for the closed case, and by $\mathfrak{S}_{\lambda,w}^\circ({\bf z})$ or $\mathfrak{S}_{\lambda,w}^\circ$  for the open case.

\vspace{2mm}

Type A Demazure atoms and Demazure characters can be viewed as ``pieces" of Schur polynomials. Let $\partial_w$ and $\partial^\circ_w$ denote the Demazure operators (to be reviewed in \cref{sec:crystal}). The Demazure character is given by $\partial_w{\bf z}^\lambda$, and the Demazure atom by $\partial^\circ_w{\bf z}^\lambda$. The Schur polynomial $s_\lambda({\bf z})$ is a sum of monomial terms, and both $\partial_w({\bf z}^\lambda)$ and $\partial^\circ_w({\bf z}^\lambda)$ correspond to sums over specific subsets of these monomials.

Demazure characters arise as characters of key submodules of highest weight representations of semisimple Lie algebras. Their structure encodes deep phenomena, including the singularities of Schubert varieties and aspects of equivariant $K$-theory. Demazure atoms, in turn, provide a finer, “atomic” decomposition of Demazure characters, isolating the smallest combinatorial pieces underlying highest-weight submodules. 

Therefore, finding solvable lattice models whose partition functions compute Demazure characters or Demazure atoms provides not only an exactly solvable statistical interpretation of these algebraic objects, but also a powerful new tool for exploring representation theory via solvability.

The open colored five-vertex models $\mathfrak{S}_{\lambda,w}^\circ({\bf z})$ in the work of Brubaker, Bump, Buciumas, and Gustafsson \cite{BBBG} form one such family of models. Using the Yang–Baxter equation, they established a recursive formula for the partition functions and proved that the partition function $Z(\mathfrak{S}_{\lambda,w}^\circ({\bf z}))$ is equal to the Demazure atom:
\begin{equation}\label{eq:1}
 Z(\mathfrak{S}_{\lambda,w}^\circ({\bf z}))=\partial^\circ_w({\bf z}^\lambda).
\end{equation}

Analogously, the partition function of the closed colored five-vertex model $\mathfrak{S}_{\lambda,w}^\bullet({\bf z})$ is equal to the Demazure character:
\begin{equation}\label{eq:2}
 Z(\mathfrak{S}_{\lambda,w}^\bullet({\bf z}))=\partial_w({\bf z}^\lambda).  
\end{equation}
Meanwhile, Demazure atoms and Demazure characters are related by the following relation (see Theorem 2.1 of \cite{BBBG}):
\begin{equation}\label{eq:3}
\partial_w({\bf z}^\lambda)=\sum_{y\leq w} \partial^\circ_y({\bf z}^\lambda).
\end{equation} 
This leads to the following relation between partition functions:
\begin{equation}\label{eq:4}
Z(\mathfrak{S}_{\lambda,w}^\bullet({\bf z}))=\sum_{y\leq w} Z(\mathfrak{S}_{\lambda,y}^\circ({\bf z})).
\end{equation}
Equations (2) and (4) may be proved, similarly to (3), using Yang-Baxter equations in \cite{buciumas2022quasi}, who also consider Demazure characters for symplectic and orthogonal types.

\vspace{2mm}

Other than as partition functions, Demazure atoms and Demazure characters can also be realized as crystal characters of specific subsets of the crystal of tableaux $\mathcal{B}_\lambda$, known as {\em (crystal) Demazure atoms} and {\em Demazure crystals}, respectively. Given that the states of both models can be embedded into the crystal $\mathcal{B}_\lambda$ (by taking the underlying Gelfand-Tsetlin patterns and the corresponding Young tabelaux of the states), the identifications \eqref{eq:1} and \eqref{eq:2} suggest a deeper algebraic structure underlying the models. In particular, they naturally lead to the prediction that the sets of states, over which the partition functions sum, carry the structure of Demazure atoms and Demazure crystals, respectively. Recent work \cite{xiao2026hybrid} constructs hybrid pipe-dream models, combining west- and east-type rows, whose weight-generating functions are key polynomials. Their explicit weight-preserving bijections change the bottom row type or interchange adjacent rows of opposite types; whereas our adjustments preserve the underlying Gelfand–Tsetlin pattern while converting open states to closed states or changing boundary flags along Bruhat order. \cite{xiao2026hybrid} does not construct Kashiwara operators on these states or identify their state sets with Demazure crystals; here we realize these operators through lattice moves and recoloring and identify the closed-state sets with the corresponding Demazure crystals.

Verifying this prediction at the level of states is considerably more subtle than at the level of partition functions. In the open case, this was established in \cite{BBBG} (see also \cref{thm1:BBBG3}); for the closed case, this forms the main result of the present paper (\cref{mainthm}). We now elaborate on this aspect in more detail.

Let $\mathcal{B}_\lambda$ be the Kashiwara-Nakashima crystal of semistandard Young tableaux of a fixed shape $\lambda$ in the alphabet $\{1,2,...,r\}$, where $r$ is the length of $\lambda$. The Demazure crystals, denoted $\mathcal{B}_\lambda(w)$, are certain subsets of $\mathcal{B}_\lambda$ indexed by elements $w\in W$, where $W$ is the Weyl group. Kashiwara \cite{kashiwara} and Littelmann \cite{littelmann} showed that the Demazure character can be recovered by summing the weights of the elements in the corresponding Demazure crystal:
\begin{equation}\label{chfor}
   \text{ch}(\mathcal{B}_\lambda(w))=\partial_w({\bf z}^\lambda). 
\end{equation}

The crystal Demazure atoms, denoted by $\mathcal{B}^\circ_\lambda(w)$, are disjoint subsets of $\mathcal{B}_\lambda$ such that 
\begin{equation}\label{eq:6} 
\mathcal{B}_{\lambda}(w)=\bigcup_{y\leq w}\mathcal{B}^\circ_\lambda(y).
\end{equation} 
Note that if $y'\leq y$ then $\mathcal{B}_\lambda(y')\subseteq \mathcal{B}_\lambda(y)$. The crystal characters can be identified with polynomial Demazure atoms: 
\begin{equation}
   \text{ch}(\mathcal{B}^\circ_\lambda(w))=\partial^\circ_w({\bf z}^\lambda). 
\end{equation}

Given a state $\mathfrak{s}$ of the model $\mathfrak{S}^\bullet_{\lambda,w}$ or $\mathfrak{S}^\circ_{\lambda,w}$, there is an associated semistandard Young tableau $\mathfrak{T}(\mathfrak{s})$ in $\mathcal{B}_{\lambda}$. The relation between the open states for the colored five-vertex models $\mathfrak{S}^\circ_{\lambda,w}$ and crystal Demazure atoms was studied in \cite{BBBG}. They prove the following refinement of (\ref{eq:1}):

\begin{thm}\label{thm1:BBBG3}
For $y\in W$, let $\mathfrak{S}^\circ_{{\lambda}}(y)$ be the set of states for the model $\mathfrak{S}^\circ_{\lambda,y}$. Let $v\mapsto v'$ be the Lascoux-Schützenberger involution on $
\mathcal{B}_\lambda$. If $y\in W$ is the longest element of the coset in $W/W_\lambda$, then the map $\mathfrak{s}\mapsto \mathfrak{T}(\mathfrak{s})'$ is a bijective map $\mathfrak{S}^\circ_{{\lambda}}(y)\rightarrow \mathcal{B}^\circ_\lambda(y)$.
\end{thm}

The main result of this paper is to independently establish an analogue for closed states, and then use the result to reprove \cref{thm1:BBBG3}. Specifically, we show that there is a bijective correspondence between Demazure crystals and the sets of closed states in the colored five-vertex models:

\begin{thm}\label{mainthm}
For $y\in W$, let $\mathfrak{S}^\bullet_{{\lambda}}(y)$ be the set of closed states for the model $\mathfrak{S}^\bullet_{\lambda,y}$. If $y\in W$ is the longest element of the coset in $W/W_\lambda$ then the map $\mathfrak{s}\mapsto \mathfrak{T}(\mathfrak{s})'$ is a bijective map $\mathfrak{S}^\bullet_{{\lambda}}(y)\rightarrow \mathcal{B}_\lambda(y)$.
\end{thm}

The strategy of the proof of \cref{mainthm} will be explained in more detail in the next subsection. Roughly speaking, we work directly with colored paths and the finite tableau crystal $\mathcal{B}_\lambda$. The proof replaces the passage to $\mathcal{B}_\infty$ and the use of Kashiwara’s (*)-involution in \cite{BBBG} with explicit recoloring arguments and the characterization of crystal operators. 

With the deterministic coloring property of the open model and the known partition-function identities, we show that \cref{eq:4} can be lifted to an equivalence at the level of states, namely,
\[\mathfrak{S}^\bullet_{\lambda}(w)\cong\bigcup_{y\leq w} \mathfrak{S}_{\lambda}^\circ(y).\]
Combining with \cref{mainthm}, we obtain a new proof of \cref{thm1:BBBG3}.

\vspace{2mm}

We also apply \cref{mainthm} to colored bosonic models. Uncolored bosonic models were studied by Korff \cite{korff2013cylindric}, and a colored version was developed by Bump and Naprienko \cite{bump2022colored}. Unlike the five-vertex models, in which each edge carries at most one color, a bosonic model allows a vertical edge to carry a multiset of colors, while each horizontal edge still carries a single color or the vacuum spin $+$. Its local Boltzmann weights are obtained by fusing monochrome vertices. We consider two opposite orders of fusion: fusion from small colors to large colors (increasing) gives the open bosonic model $\mathfrak{S}^{b,\circ}_{\lambda,w}$, whereas fusion from large colors to small colors (decreasing) gives the closed bosonic model $\mathfrak{S}^{b,\bullet}_{\lambda,w}$. The former is studied in \cite{bump2022colored}, while the latter is new.

For the decreasing fusion order, we construct compatible fused and auxiliary $R$-matrices and establish the Yang--Baxter equation for both the $P$- and $R$-weights. A train argument then yields a Demazure--Lusztig recursion for the partition functions. At $t=0$, this becomes the usual Demazure recursion and gives
\[
  Z\bigl(\mathfrak{S}^{b,\bullet}_{\lambda,w}\bigr)
  =\partial_w({\bf z}^{\lambda}).
\]
Thus the closed bosonic model provides another lattice-model realization of Demazure characters.

The analogy between the bosonic and fermionic models also extends to individual states. At $t=0$, a shift between Gelfand--Tsetlin patterns with top row $\lambda$ and left-strict patterns with top row $\lambda+\rho$ leads to boundary-flag-preserving bijections between bosonic states and their five-vertex counterparts. Combining these bijections with the results above, and using the coset-representative convention of Theorems~\ref{thm1:BBBG3} and~\ref{mainthm}, we obtain

\begin{thm}\label{mainthm2}
For $y\in W$, let $\mathfrak{S}^{b,\bullet}_{{\lambda}}(y)$ be the set of closed states for the closed bosonic model $\mathfrak{S}^{b,\bullet}_{\lambda,y}$. Then there is a natural bijective map $\mathfrak{S}^{b,\bullet}_{{\lambda}}(y)\rightarrow \mathcal{B}_\lambda(y)$. Similar results hold for open bosonic models, there is a natural bijective map $\mathfrak{S}^{b,\circ}_{{\lambda}}(y)\rightarrow \mathcal{B}^\circ_\lambda(y)$.
\end{thm}

Moreover, the crystal operators and recoloring procedures along the proof of \cref{mainthm} also extend to the bosonic setting.

\subsection{Strategy of proofs.} The proof of \cref{mainthm} combines an analysis of the closed states of lattice models whose boundary conditions are parametrized by elements of the Weyl group with an analysis of the effect of crystal operations. We will inspect two successive “adjustments” that convert between open and closed states. In particular, whenever $y\leq w$ in the Bruhat order on $S_r$, there is a composite map
\begin{equation}\mathfrak{S}_{\lambda}^\circ(y)\xrightarrow{\text{adjustment 1}}\mathfrak{S}^\bullet_{\lambda}(y)\xrightarrow{\text{adjustment 2}}\mathfrak{S}_{\lambda}^\bullet(w),\end{equation} 
where ``adjustment 1" (to be called {\em closed adjustment} in \cref{sec:state}) transforms an open state into a closed state, and ``adjustment 2" (see \cref{prop:y>w}) alters the boundary conditions to produce closed states of $\mathfrak{S}_{\lambda,w}^\bullet$ from closed states of $\mathfrak{S}^\bullet_{\lambda,y}$.

For example, let $y=(123)$ and $w=(13)$. The following illustrations show how the adjustments are applied to transform an open state of $\mathfrak{S}^\circ_{\lambda,y}$ to a closed state of $\mathfrak{S}^\bullet_{\lambda,w}$.

\begin{figure}[h!]
\centering
    \begin{tikzpicture}[scale=0.5,every node/.style={scale=.4}]
\foreach \i in {1,3,5}
	\draw[thick] (0,\i) -- (12,\i);
\foreach \j in {1,3,5,7,9,11}
	\draw[thick] (\j,0) to (\j,6);

\foreach \i in {0,2,4,6}
	\foreach \j in {1,3,5,7,9,11}
	  \draw[thick, fill=white] (\j,\i) circle (.25);
\foreach \i in {1,3,5}
	\foreach \j in {0,2,4,6,8,10,12}
	  \draw[thick, fill=white] (\j,\i) circle (.25);

      \draw[line width=.5mm,red] (1,6) to (1,5.5) to [out=-90,in=180] (1.5,5) to 
	(7.5,5) to [out=0,in=90] (9,4.5) to (9,3.5) to [out=-90,in=180] (9.5,3) to (12,3);

    \draw[line width=.5mm,blue] (5,6) to (5,3.5) to [out=-90,in=180] (6.5,3) to (8.5,3) to [out=0,in=90] (9,2.5) to (9,1.5) to [out=-90,in=180] (9.5,1) to (12,1);

    \draw[line width=.5mm,green!80!black] (11,6) to (11,5.5) to [out=-90,in=180] (11.5,5) to (12,5);

	\foreach \i/\j in {1/6,2/5,4/5,6/5,8/5,9/4,10/3,12/3}
		 \draw[line width=.5mm,red,fill=white] (\i,\j) circle (.25);

	\foreach \i/\j in {5/6,5/4,6/3,8/3,9/2,10/1,12/1}
		 \draw[line width=.5mm,blue,fill=white] (\i,\j) circle (.25);

         \foreach \i/\j in {11/6,12/5}
		 \draw[line width=.5mm,green,fill=white] (\i,\j) circle (.25);

\foreach \i/\j in {1/1,3/1,5/1,7/1,9/1,11/1}
\node at (\i,\j) [circle,fill,inner sep=2.5pt]{};
\foreach \i/\j in {1/1,3/1,5/1,7/1,9/1,11/1}
\node at (\i,3) [circle,fill,inner sep=2.5pt]{};
\foreach \i/\j in {1/1,3/1,5/1,7/1,9/1,11/1}
\node at (\i,5) [circle,fill,inner sep=2.5pt]{};
\foreach \i/\j in {1/1,3/1,5/1,7/1,9/1,11/1}
\node at (\i,0) {{\bf +}};
\foreach \i/\j in {3/1,7/1,9/1}
\node at (\i,6) {{\bf +}};
\foreach \i/\j in {0/1,0/3,0/5}
\node at (0,\j) {{\bf +}};
	\foreach \j/\c in {1/5,3/4,5/3,7/2,9/1,11/0}
	\node at (\j,6.5) {$\c$};
	\foreach \i/\c in {1/3,3/2,5/1}
	\node at (-.5,\i) {$\c$};

\end{tikzpicture}
\end{figure}
\[\llap{adjustment 1}{\Big\downarrow}\]

\newpage 

\begin{figure}[h!]
\centering
\begin{tikzpicture}[scale=.5,every node/.style={scale=.4}]

\foreach \i in {1,3,5}
	\draw[thick] (0,\i) -- (12,\i);
\foreach \j in {1,3,5,7,9,11}
	\draw[thick] (\j,0) to (\j,6);

\foreach \i in {0,2,4,6}
	\foreach \j in {1,3,5,7,9,11}
	  \draw[thick, fill=white] (\j,\i) circle (.25);
\foreach \i in {1,3,5}
	\foreach \j in {0,2,4,6,8,10,12}
	  \draw[thick, fill=white] (\j,\i) circle (.25);

      \draw[line width=.5mm,red] (1,6) to (1,5.5) to [out=-90,in=180] (1.5,5) to 
	(4.5,5) to [out=0,in=90] (5,4.5) to (5,3.5) to [out=-90,in=180] (5.5,3) to (12,3);

    \draw[line width=.5mm,blue] (5,6) to (5,5.5) to [out=-90,in=180] (5.5,5) to (8.5,5) to  [out=0,in=90] (9,4.5) to (9,1.5) to [out=-90,in=180] (9.5,1) to (12,1);

    \draw[line width=.5mm,green!80!black] (11,6) to (11,5.5) to [out=-90,in=180] (11.5,5) to (12,5);

	\foreach \i/\j in {1/6,2/5,4/5,5/4,6/3,8/3,10/3,12/3}
		 \draw[line width=.5mm,red,fill=white] (\i,\j) circle (.25);

	\foreach \i/\j in {5/6,6/5,8/5,9/4,9/2,10/1,12/1}
		 \draw[line width=.5mm,blue,fill=white] (\i,\j) circle (.25);

         \foreach \i/\j in {11/6,12/5}
		 \draw[line width=.5mm,green,fill=white] (\i,\j) circle (.25);

\foreach \i/\j in {1/1,3/1,5/1,7/1,9/1,11/1}
\node at (\i,\j) [circle,fill,inner sep=2.5pt]{};
\foreach \i/\j in {1/1,3/1,5/1,7/1,9/1,11/1}
\node at (\i,3) [circle,fill,inner sep=2.5pt]{};
\foreach \i/\j in {1/1,3/1,5/1,7/1,9/1,11/1}
\node at (\i,5) [circle,fill,inner sep=2.5pt]{};
\foreach \i/\j in {1/1,3/1,5/1,7/1,9/1,11/1}
\node at (\i,0) {{\bf +}};
\foreach \i/\j in {3/1,7/1,9/1}
\node at (\i,6) {{\bf +}};
\foreach \i/\j in {0/1,0/3,0/5}
\node at (0,\j) {{\bf +}};
	\foreach \j/\c in {1/5,3/4,5/3,7/2,9/1,11/0}
	\node at (\j,6.5) {$\c$};
	\foreach \i/\c in {1/3,3/2,5/1}
	\node at (-.5,\i) {$\c$};

\end{tikzpicture}
\end{figure} 

\[\llap{adjustment 2}{\Big\downarrow}\]
\begin{figure}[h!]
\centering
\begin{tikzpicture}[scale=.5,every node/.style={scale=.4}]

\foreach \i in {1,3,5}
	\draw[thick] (0,\i) -- (12,\i);
\foreach \j in {1,3,5,7,9,11}
	\draw[thick] (\j,0) to (\j,6);

\foreach \i in {0,2,4,6}
	\foreach \j in {1,3,5,7,9,11}
	  \draw[thick, fill=white] (\j,\i) circle (.25);
\foreach \i in {1,3,5}
	\foreach \j in {0,2,4,6,8,10,12}
	  \draw[thick, fill=white] (\j,\i) circle (.25);

      \draw[line width=.5mm,red] (1,6) to (1,5.5) to [out=-90,in=180] (1.5,5) to 
	(4.5,5) to [out=0,in=90] (5,4.5) to (5,3.5) to [out=-90,in=180] (5.5,3) to (8.5,3) to  [out=0,in=90] (9,2.5) to (9,1.5) to [out=-90,in=180] (9.5,1) to (12,1); 

    \draw[line width=.5mm,blue] (5,6) to (5,5.5) to [out=-90,in=180] (5.5,5) to (8.5,5) to  [out=0,in=90] (9,4.5) to  (9,3.5) to [out=-90,in=180] (9.5,3) to (12,3);

    \draw[line width=.5mm,green!80!black] (11,6) to (11,5.5) to [out=-90,in=180] (11.5,5) to (12,5);

	\foreach \i/\j in {1/6,2/5,4/5,5/4,6/3,8/3,9/2,10/1,12/1}
		 \draw[line width=.5mm,red,fill=white] (\i,\j) circle (.25);

	\foreach \i/\j in {5/6,6/5,8/5,9/4,10/3,12/3}
		 \draw[line width=.5mm,blue,fill=white] (\i,\j) circle (.25);

         \foreach \i/\j in {11/6,12/5}
		 \draw[line width=.5mm,green,fill=white] (\i,\j) circle (.25);

\foreach \i/\j in {1/1,3/1,5/1,7/1,9/1,11/1}
\node at (\i,\j) [circle,fill,inner sep=2.5pt]{};
\foreach \i/\j in {1/1,3/1,5/1,7/1,9/1,11/1}
\node at (\i,3) [circle,fill,inner sep=2.5pt]{};
\foreach \i/\j in {1/1,3/1,5/1,7/1,9/1,11/1}
\node at (\i,5) [circle,fill,inner sep=2.5pt]{};
\foreach \i/\j in {1/1,3/1,5/1,7/1,9/1,11/1}
\node at (\i,0) {{\bf +}};
\foreach \i/\j in {3/1,7/1,9/1}
\node at (\i,6) {{\bf +}};
\foreach \i/\j in {0/1,0/3,0/5}
\node at (0,\j) {{\bf +}};
	\foreach \j/\c in {1/5,3/4,5/3,7/2,9/1,11/0}
	\node at (\j,6.5) {$\c$};
	\foreach \i/\c in {1/3,3/2,5/1}
	\node at (-.5,\i) {$\c$};

\end{tikzpicture}

\end{figure}

Based on these adjustments, we then establish criteria (\cref{prop:wstate}, \cref{prop:y<w}, \cref{prop:y>w}) for the existence and uniqueness of closed states in terms of the Bruhat order on the Weyl group, leading to a bijection
\begin{equation}\label{eq:7}
   \mathfrak{S}^\bullet_{\lambda}(w)\cong\bigcup_{y\leq w} \mathfrak{S}_{\lambda}^\circ(y). 
\end{equation} Combining this bijection with the comparison maps 
\[\mathfrak{S}_{\lambda}^\circ(w) \rightarrow \mathcal{B}_\lambda\] and
\[\theta:\mathfrak{S}_{\lambda}^\bullet(w) \rightarrow \mathcal{B}_\lambda,\] which associate a semistandard Young tableau in $\mathcal{B}_\lambda$ to each state, we obtain a commutative diagram:
\[
\begin{tikzcd}[row sep=large, column sep=huge]
  \mathfrak{S}_{\lambda}^\circ(y) \arrow[r, "\text{adjustment}"] \arrow[dr] & \mathfrak{S}_{\lambda}^\bullet(w) \arrow[d] \\
  & \mathcal{B}_\lambda
\end{tikzcd}
\]

Next, we use an inductive description of the Demazure crystals: If $s_iw>w$ then 
    \[
    \mathcal{B}_\lambda(s_i w) = \{ x \in \mathcal{B}_\lambda \mid e_i^k x \in \mathcal{B}(w) \text{ for some } k \geq 0\}. 
    \]  We study how the crystal operators $e_i$ act on closed states, and inductively show the image $\theta(\mathfrak{S}_\lambda^\bullet(w))$  agrees with the Demazure crystal $\mathcal{B}_\lambda(w)$, as asserted in our main result, \cref{mainthm}. Specifically, if $s_iw>w$ and $\mathfrak{s}\in \mathfrak{S}_\lambda^\bullet(s_iw)$, we will prove the following two facts which establish the crystal structure of closed states:

\textbf{Case 1.} If $e_i(\theta(\mathfrak{s})) \neq 0$ then there exists $\mathfrak{t}\in \mathfrak{S}_{\lambda}^\bullet(s_iw)$ such that $\theta(\mathfrak{t})=e_i(\theta(\mathfrak{s})).$

\textbf{Case 2.} If $e_i(\theta(\mathfrak{s})) = 0$ then there exists a recoloring of the state $\mathfrak{s}$ into $\mathfrak{S}_\lambda^\bullet(w)$, i.e., there exists $\mathfrak{t}\in \mathfrak{S}_{\lambda}^\bullet(w)$ such that $\theta(\mathfrak{s})=\theta(\mathfrak{t}).$
    
    Finally, comparing the bijection (\ref{eq:7}) with the structure of Demazure crystals (\ref{eq:6}) reproves \cref{thm1:BBBG3}.

The organization of the paper is as follows: In Section 2, we define the colored five-vertex model and related admissible states. We compare open and closed states in Section 3,
proving a criterion for the existence and uniqueness of closed states. In Section 4, we collect
the background and results on crystals and Demazure operators. In Section 5, we analyze
the crystal operators and give the proof of \cref{mainthm}. Finally, in Section 6, we introduce
the two fusion orders for colored bosonic models, establish the Yang–Baxter equation and
partition-function recursion for the closed model, and apply the preceding results to identify
the bosonic state spaces at $t=0$ with Demazure atoms and Demazure crystals.

\subsection{Acknowledgements}
The author would like to thank Daniel Bump for his guidance throughout the project.
\section{Colored lattice models}\label{sec:model}

\subsection{Closed colored models\label{subsec:closedmodel}}

We begin by describing the {\em closed colored five-vertex models} $\mathfrak{S}^\bullet_{\lambda,w}$, which depend on a partition $\lambda$ of length $r$ and a permutation $w \in S_r$. Let $r \geq 1$ and consider a partition $\lambda = (\lambda_1, \lambda_2, \dots, \lambda_r)$. Define the partition $\rho = (r-1, r-2, \dots, 1, 0)$, so that $\lambda + \rho = (\lambda_1 + r - 1, \lambda_2 + r - 2, \dots, \lambda_{r-1} + 1, \lambda_r)$. The colored five-vertex model is constructed on a planar, rectangular $r \times N$ grid of vertices and edges, where $N \geq \lambda_1 + r$. In all cases considered below, we assume that $N = \lambda_1 + r$, as any additional columns beyond this bound do not affect the model.

The columns in the grid are labeled $0$ to $N-1$ from right to left, and the rows are labeled $1$ to $r$ from top to bottom. Thus, the $r\times N$ grid consists of vertices labeled by coordinates $(i,j)$ with $1\leq i \leq r$ and $0\leq j\leq N-1$. Each vertex is connected to four edges in the up/down and left/right directions. An edge that connects to only one vertex is called a {\em boundary edge}, while an edge that connects two vertices is called an {\em interior edge}.

Every edge $e$ is assigned a set $\Sigma_e$ of possible {\em spins}. In our model, all edges are assigned the same spinset $\Sigma=\{+\}\cup \{c_1,...,c_r\}$ where the elements $c_i$ represent {\em colors} ordered by $c_1>c_{2}>...>c_r.$ If an edge is assigned the spin $c_i$, we say the edge has color $c_i$; if it is assigned the spin $+$, we say the edge is \emph{uncolored}.

Let ${\bf c}=(c_1,...,c_r)$ be a vector. Given $w\in S_r$, let $w{\bf c}$ be the permuted vector of colors, called the {\em boundary flag}, given by $w{\bf c}=(c_{w^{-1}(i)}).$ As part of the data of the model, the spins of the boundary edges are prescribed: Put colors $c_1,...,c_r$ on the top boundary at columns $\lambda_1+r-1,\lambda_2+r-2,\ldots,\lambda_r$. On the right boundary, put colors in the order of $w{\bf c}$ from top to bottom. In other words, the color $c_i$ is in the row $w(i).$ Then we put ``$+$" at all remaining boundary edges.

\begin{example}\label{modelex}
   {\em  Below is a visualization of the colored five-vertex $\mathfrak{S}_{\lambda,w}^\bullet$
model with $\lambda=(3,2,0)$ and $w=(1\,2\,3)$. Since $\lambda+\rho=(5,3,0)$, we have three rows and six columns. Note that every vertex is a black dot and at the center of every edge we draw a circle. 

For the boundary conditions, we assume the colors are red, blue, and green, and they are ordered as 
\[\text{red ($c_1$) $>$ blue ($c_2$) $>$ green ($c_3$)}.\]
So if $w=(1\,2\,3)$ then the colors on the right edge are $(c_3,c_1,c_2)$. We color the boundary edges accordingly, and put a $+$ on the other boundary edges. }
\end{example}
\begin{figure}[h!]
\begin{centering}
\begin{tikzpicture}[scale=.85,every node/.style={scale=.7}]

\foreach \i in {1,3,5}
	\draw[thick] (0,\i) -- (12,\i);
\foreach \j in {1,3,5,7,9,11}
	\draw[thick] (\j,0) to (\j,6);

\foreach \i in {0,2,4,6}
	\foreach \j in {1,3,5,7,9,11}
	  \draw[thick, fill=white] (\j,\i) circle (.25);
\foreach \i in {1,3,5}
	\foreach \j in {0,2,4,6,8,10,12}
	  \draw[thick, fill=white] (\j,\i) circle (.25);

	\foreach \i/\j in {1/6,12/3}
		 \draw[line width=.5mm,red,fill=white] (\i,\j) circle (.25);

	\foreach \i/\j in {5/6,12/1}
		 \draw[line width=.5mm,blue,fill=white] (\i,\j) circle (.25);

         \foreach \i/\j in {11/6,12/5}
		 \draw[line width=.5mm,green,fill=white] (\i,\j) circle (.25);

\foreach \i/\j in {1/1,3/1,5/1,7/1,9/1,11/1}
\node at (\i,\j) [circle,fill,inner sep=2.5pt]{};
\foreach \i/\j in {1/1,3/1,5/1,7/1,9/1,11/1}
\node at (\i,3) [circle,fill,inner sep=2.5pt]{};
\foreach \i/\j in {1/1,3/1,5/1,7/1,9/1,11/1}
\node at (\i,5) [circle,fill,inner sep=2.5pt]{};
\foreach \i/\j in {1/1,3/1,5/1,7/1,9/1,11/1}
\node at (\i,0) {{\bf +}};
\foreach \i/\j in {3/1,7/1,9/1}
\node at (\i,6) {{\bf +}};
\foreach \i/\j in {0/1,0/3,0/5}
\node at (0,\j) {{\bf +}};
	\foreach \j/\c in {1/5,3/4,5/3,7/2,9/1,11/0}
	\node at (\j,6.5) {$\c$};
	\foreach \i/\c in {1/3,3/2,5/1}
	\node at (-.5,\i) {$\c$};

\end{tikzpicture}
\caption{Boundary conditions for $\mathfrak{S}^\bullet_{(3,2,0),(1\,2\,3)}$.}
\end{centering}
\end{figure}
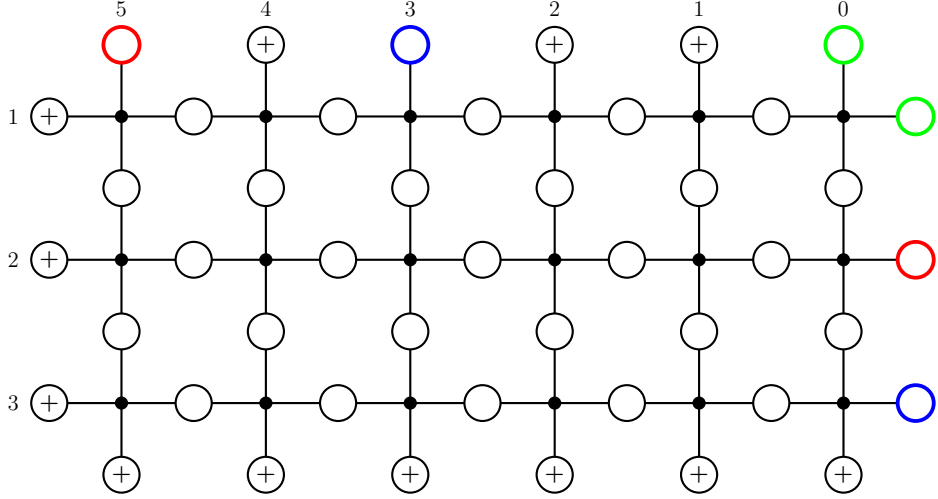

\vspace{2mm}

Given a lattice model, a {\em state} $\mathfrak{s}$ of the model is an assignment of an element of its spinset to every (interior) edge of the model. There are local constraints at each vertex on the possible configurations
of spins adjacent to it. A state in which these constraints are
satisfied at every vertex is called {\em admissible}. 

Given any admissible local configuration near the vertex, there is an assigned value $\beta_v$, called the {\em local Boltzmann weight} at each vertex $v$. The table below lists and labels all admissible configurations with the assumption that ``red" $>$  ``blue" in the order of colors.  The last row are the local Boltzmann weights, where the indeterminates $z_i$ corresponds to vertices in the $i$-th row of the model.  

\begin{table}[h]

\begin{center}

\begin{tabular}{ |c|c|c|c|c|c|c| } 
\hline
$a_1$ & $a_{2,1}$ & $a_{2,3}$ & $a_{2,5}$ & $b_2$ & $c_1$ & $c_2$ \\ 
 \hline
 \begin{tikzpicture}[scale=.7,every node/.style={scale=.5}]

\foreach \i in {1}
	\draw[thick] (0,\i) -- (2,\i);
\foreach \j in {1}
	\draw[thick] (\j,0) to (\j,2);
    \foreach \i/\j in {1/0,0/1,1/2,2/1}
  \draw[thick, fill=white] (\j,\i) circle (.25);
  \foreach \i/\j in {0/1,2/1}
\node at (\i,1) {{\bf +}};
  \foreach \i/\j in {1/2,1/0}
\node at (1,\j) {{\bf +}};
 \end{tikzpicture}  & \begin{tikzpicture}[scale=.7,every node/.style={scale=.5}]

\foreach \i in {1}
	\draw[thick] (0,\i) -- (2,\i);
\foreach \j in {1}
	\draw[thick] (\j,0) to (\j,2);
    
   \draw[line width=.5mm,red] (0,1) to (2,1);
   \draw[line width=.5mm,blue] (1,0) to (1,2);
   \foreach \i/\j in {1/0,0/1,1/2,2/1}
  \draw[thick, fill=white] (\j,\i) circle (.25);

 \end{tikzpicture}  & \begin{tikzpicture}[scale=.7,every node/.style={scale=.5}]

\foreach \i in {1}
	\draw[thick] (0,\i) -- (2,\i);
\foreach \j in {1}
	\draw[thick] (\j,0) to (\j,2);
    
   \draw[line width=.5mm,blue] (1,2) to (1,1.5) to [out=-90,in=180] (1.5,1) to (2,1);
 \draw[line width=.5mm,red] (0,1) to (0.5,1) to [out=0,in=90] (1,0.5) to (1,0);
   \foreach \i/\j in {1/0,0/1,1/2,2/1}
  \draw[thick, fill=white] (\j,\i) circle (.25);

 \end{tikzpicture} &  \begin{tikzpicture}[scale=.7,every node/.style={scale=.5}]

\foreach \i in {1}
	\draw[thick] (0,\i) -- (2,\i);
\foreach \j in {1}
	\draw[thick] (\j,0) to (\j,2);
    
   \draw[line width=.5mm,red] (0,1) to (2,1);
   \draw[line width=.5mm,red] (1,0) to (1,2);
   \foreach \i/\j in {1/0,0/1,1/2,2/1}
  \draw[thick, fill=white] (\j,\i) circle (.25);

 \end{tikzpicture}& \begin{tikzpicture}[scale=.7,every node/.style={scale=.5}]

\foreach \i in {1}
	\draw[thick] (0,\i) -- (2,\i);
\foreach \j in {1}
	\draw[thick] (\j,0) to (\j,2);
    
   \draw[line width=.5mm,red] (0,1) to (2,1);

   \foreach \i/\j in {1/0,0/1,1/2,2/1}
  \draw[thick, fill=white] (\j,\i) circle (.25);
  \foreach \i/\j in {1/2,1/0}
\node at (1,\j) {{\bf +}};
 \end{tikzpicture}  &  \begin{tikzpicture}[scale=.7,every node/.style={scale=.5}]

\foreach \i in {1}
	\draw[thick] (0,\i) -- (2,\i);
\foreach \j in {1}
	\draw[thick] (\j,0) to (\j,2);
    
   \draw[line width=.5mm,red] (0,1) to (0.5,1) to [out=0,in=90] (1,0.5) to (1,0);

   \foreach \i/\j in {1/0,0/1,1/2,2/1}
  \draw[thick, fill=white] (\j,\i) circle (.25);
  \foreach \i/\j in {1/2,2/1}
\node at (\i,\j) {{\bf +}};
 \end{tikzpicture} & \begin{tikzpicture}[scale=.7,every node/.style={scale=.5}]

\foreach \i in {1}
	\draw[thick] (0,\i) -- (2,\i);
\foreach \j in {1}
	\draw[thick] (\j,0) to (\j,2);
    
   \draw[line width=.5mm,red] (1,2) to (1,1.5) to [out=-90,in=180] (1.5,1) to (2,1);

   \foreach \i/\j in {1/0,0/1,1/2,2/1}
  \draw[thick, fill=white] (\j,\i) circle (.25);
  \foreach \i/\j in {0/1,1/0}
\node at (\i,\j) {{\bf +}};
 \end{tikzpicture}\\ 

 \hline
 $1$ & $z_i$ & $z_i$ & $z_i$ & $z_i$ & $z_i$ & $1$ \\ 
 \hline

\end{tabular}

\vspace{2mm}
\caption{Admissible local configurations and Boltzmann weights in closed models, with red $>$ blue.}
\label{table2}
\end{center}
\end{table}

   The {\em Boltzmann weight} of a state $\mathfrak{s}$ is the product of local Boltzmann weights over all vertices:
\[\beta(\mathfrak{s})=\prod_{v} \beta_v(\mathfrak{s}).\]

\begin{defn}\label{defmodel} For a partition $\lambda$ of length $r$ and $w\in S_r$, the {\bf closed colored five-vertex model} $\mathfrak{S}^\bullet_{\lambda,w}$ is the $r\times N$ rectangular model with the boundary conditions, local constraints and Boltzmann weights at each vertex described above.    

By abuse of notation, we will denote the set of admissible states of the model also by $\mathfrak{S}^\bullet_{\lambda}(w)$. We will often call the states of the closed model {\bf closed states}.
\end{defn}

The {\em partition function} $Z(\mathfrak{S}_{\lambda,w}^\bullet)$ is defined to be the sum of the Boltzmann weights over all admissible states. 
\[Z(\mathfrak{S}^\bullet_{\lambda,w})=\sum_{\mathfrak{s}\in \mathfrak{S}_{\lambda}^\bullet(w)}\beta(\mathfrak{s}).\]
We will see that the partition function encodes representation-theoretic information.

\begin{example}\label{closedstateex}
    {\em  Below in \cref{closedstatefig} is a closed state for the model in \cref{modelex}. At vertex $(1,3),$ it is an $a_{2,3}$ pattern. At vertex $(2,1)$, it is an $a_{2,1}$ pattern.}
\end{example}

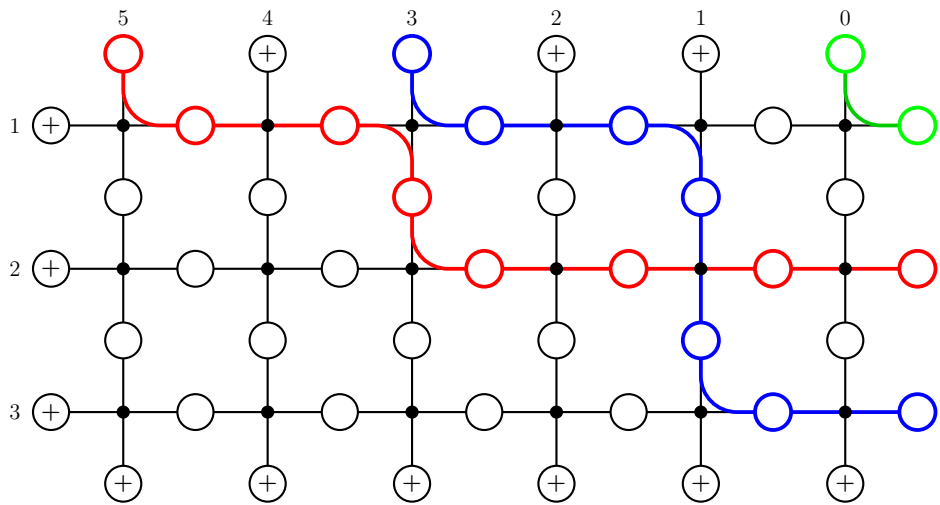
\begin{figure}[h!]
\centering
\begin{tikzpicture}[scale=.85,every node/.style={scale=.7}]

\foreach \i in {1,3,5}
	\draw[thick] (0,\i) -- (12,\i);
\foreach \j in {1,3,5,7,9,11}
	\draw[thick] (\j,0) to (\j,6);

\foreach \i in {0,2,4,6}
	\foreach \j in {1,3,5,7,9,11}
	  \draw[thick, fill=white] (\j,\i) circle (.25);
\foreach \i in {1,3,5}
	\foreach \j in {0,2,4,6,8,10,12}
	  \draw[thick, fill=white] (\j,\i) circle (.25);

      \draw[line width=.5mm,red] (1,6) to (1,5.5) to [out=-90,in=180] (1.5,5) to 
	(4.5,5) to [out=0,in=90] (5,4.5) to (5,3.5) to [out=-90,in=180] (5.5,3) to (12,3);

    \draw[line width=.5mm,blue] (5,6) to (5,5.5) to [out=-90,in=180] (5.5,5) to (8.5,5) to  [out=0,in=90] (9,4.5) to (9,1.5) to [out=-90,in=180] (9.5,1) to (12,1);

    \draw[line width=.5mm,green!80!black] (11,6) to (11,5.5) to [out=-90,in=180] (11.5,5) to (12,5);

	\foreach \i/\j in {1/6,2/5,4/5,5/4,6/3,8/3,10/3,12/3}
		 \draw[line width=.5mm,red,fill=white] (\i,\j) circle (.25);

	\foreach \i/\j in {5/6,6/5,8/5,9/4,9/2,10/1,12/1}
		 \draw[line width=.5mm,blue,fill=white] (\i,\j) circle (.25);

         \foreach \i/\j in {11/6,12/5}
		 \draw[line width=.5mm,green,fill=white] (\i,\j) circle (.25);

\foreach \i/\j in {1/1,3/1,5/1,7/1,9/1,11/1}
\node at (\i,\j) [circle,fill,inner sep=2.5pt]{};
\foreach \i/\j in {1/1,3/1,5/1,7/1,9/1,11/1}
\node at (\i,3) [circle,fill,inner sep=2.5pt]{};
\foreach \i/\j in {1/1,3/1,5/1,7/1,9/1,11/1}
\node at (\i,5) [circle,fill,inner sep=2.5pt]{};
\foreach \i/\j in {1/1,3/1,5/1,7/1,9/1,11/1}
\node at (\i,0) {{\bf +}};
\foreach \i/\j in {3/1,7/1,9/1}
\node at (\i,6) {{\bf +}};
\foreach \i/\j in {0/1,0/3,0/5}
\node at (0,\j) {{\bf +}};
	\foreach \j/\c in {1/5,3/4,5/3,7/2,9/1,11/0}
	\node at (\j,6.5) {$\c$};
	\foreach \i/\c in {1/3,3/2,5/1}
	\node at (-.5,\i) {$\c$};

\end{tikzpicture}
\vspace{4mm}
\caption{A closed state for $\mathfrak{S}^\bullet_{(3,2,0),(1\,2\,3)}$  with  $\text{red} > \text{blue} > \text{green}.$}
\label{closedstatefig}
\end{figure}

Now we turn to the definition of open models. The open colored five-vertex model, denoted by $\mathfrak{S}_{\lambda,w}^\circ$, has the same grid and boundary conditions, but the admissible states and Boltzmann weights are different. Below is the table listing the local admissible configurations and their local Boltzmann weights. 

\begin{table}[h]
\begin{center}
\begin{tabular}{ |c|c|c|c|c|c|c| } 
\hline
$a_1$ & $a_{2,1}$ & $a_{2,4}$ & $a_{2,5}$ & $b_2$ & $c_1$ & $c_2$ \\ 
 \hline
 \begin{tikzpicture}[scale=.7,every node/.style={scale=.5}]

\foreach \i in {1}
	\draw[thick] (0,\i) -- (2,\i);
\foreach \j in {1}
	\draw[thick] (\j,0) to (\j,2);
    \foreach \i/\j in {1/0,0/1,1/2,2/1}
  \draw[thick, fill=white] (\j,\i) circle (.25);
  \foreach \i/\j in {0/1,2/1}
\node at (\i,1) {{\bf +}};
  \foreach \i/\j in {1/2,1/0}
\node at (1,\j) {{\bf +}};
 \end{tikzpicture}  & \begin{tikzpicture}[scale=.7,every node/.style={scale=.5}]

\foreach \i in {1}
	\draw[thick] (0,\i) -- (2,\i);
\foreach \j in {1}
	\draw[thick] (\j,0) to (\j,2);
    
   \draw[line width=.5mm,red] (0,1) to (2,1);
   \draw[line width=.5mm,blue] (1,0) to (1,2);
   \foreach \i/\j in {1/0,0/1,1/2,2/1}
  \draw[thick, fill=white] (\j,\i) circle (.25);

 \end{tikzpicture}  & \begin{tikzpicture}[scale=.7,every node/.style={scale=.5}]

\foreach \i in {1}
	\draw[thick] (0,\i) -- (2,\i);
\foreach \j in {1}
	\draw[thick] (\j,0) to (\j,2);
    
   \draw[line width=.5mm,red] (1,2) to (1,1.5) to [out=-90,in=180] (1.5,1) to (2,1);
 \draw[line width=.5mm,blue] (0,1) to (0.5,1) to [out=0,in=90] (1,0.5) to (1,0);
   \foreach \i/\j in {1/0,0/1,1/2,2/1}
  \draw[thick, fill=white] (\j,\i) circle (.25);

 \end{tikzpicture} &  \begin{tikzpicture}[scale=.7,every node/.style={scale=.5}]

\foreach \i in {1}
	\draw[thick] (0,\i) -- (2,\i);
\foreach \j in {1}
	\draw[thick] (\j,0) to (\j,2);
    
   \draw[line width=.5mm,red] (0,1) to (2,1);
   \draw[line width=.5mm,red] (1,0) to (1,2);
   \foreach \i/\j in {1/0,0/1,1/2,2/1}
  \draw[thick, fill=white] (\j,\i) circle (.25);

 \end{tikzpicture} & \begin{tikzpicture}[scale=.7,every node/.style={scale=.5}]

\foreach \i in {1}
	\draw[thick] (0,\i) -- (2,\i);
\foreach \j in {1}
	\draw[thick] (\j,0) to (\j,2);
    
   \draw[line width=.5mm,red] (0,1) to (2,1);

   \foreach \i/\j in {1/0,0/1,1/2,2/1}
  \draw[thick, fill=white] (\j,\i) circle (.25);
  \foreach \i/\j in {1/2,1/0}
\node at (1,\j) {{\bf +}};
 \end{tikzpicture}  &  \begin{tikzpicture}[scale=.7,every node/.style={scale=.5}]

\foreach \i in {1}
	\draw[thick] (0,\i) -- (2,\i);
\foreach \j in {1}
	\draw[thick] (\j,0) to (\j,2);
    
   \draw[line width=.5mm,red] (0,1) to (0.5,1) to [out=0,in=90] (1,0.5) to (1,0);

   \foreach \i/\j in {1/0,0/1,1/2,2/1}
  \draw[thick, fill=white] (\j,\i) circle (.25);
  \foreach \i/\j in {1/2,2/1}
\node at (\i,\j) {{\bf +}};
 \end{tikzpicture} & \begin{tikzpicture}[scale=.7,every node/.style={scale=.5}]

\foreach \i in {1}
	\draw[thick] (0,\i) -- (2,\i);
\foreach \j in {1}
	\draw[thick] (\j,0) to (\j,2);
    
   \draw[line width=.5mm,red] (1,2) to (1,1.5) to [out=-90,in=180] (1.5,1) to (2,1);

   \foreach \i/\j in {1/0,0/1,1/2,2/1}
  \draw[thick, fill=white] (\j,\i) circle (.25);
  \foreach \i/\j in {0/1,1/0}
\node at (\i,\j) {{\bf +}};
 \end{tikzpicture}\\ 

 \hline
 $1$ & $z_i$ & $z_i$ & $z_i$ & $z_i$ & $z_i$ & $1$ \\ 
 \hline

\end{tabular}
\vspace{2mm}
\caption{Admissible local configurations and Boltzmann weights in open models, with red $>$ blue.}
\label{table3}
\end{center}
\end{table}
\begin{defn}\label{defopenmodel} For a partition $\lambda$ of length $r$ and $w\in S_r$, the {\bf open colored five-vertex model} $\mathfrak{S}^\circ_{\lambda,w}$ is the $r\times N$ rectangular model with the boundary conditions, local constraints and Boltzmann weights at each vertex described above.    

By abuse of notation, we will denote the set of admissible states of the model also by $\mathfrak{S}^\circ_{\lambda}(w)$. We will often call the states of the open model {\bf open states}.
\end{defn}

\begin{example}\label{openclosed}
   {\em  Below in \cref{openstatefig} is an open state for the model $\mathfrak{S}^\circ_{\lambda,w}$ with the same $\lambda$ and $w$ as in \cref{modelex}. The configuration at vertex $(1,3),$ it is an $a_{2,1}$ pattern. The configuration at vertex $(2,1)$, it is an $a_{2,4}$ pattern.}
\end{example}

\begin{figure}[h!]
\centering
\begin{tikzpicture}[scale=.85,every node/.style={scale=.7}]
\foreach \i in {1,3,5}
	\draw[thick] (0,\i) -- (12,\i);
\foreach \j in {1,3,5,7,9,11}
	\draw[thick] (\j,0) to (\j,6);

\foreach \i in {0,2,4,6}
	\foreach \j in {1,3,5,7,9,11}
	  \draw[thick, fill=white] (\j,\i) circle (.25);
\foreach \i in {1,3,5}
	\foreach \j in {0,2,4,6,8,10,12}
	  \draw[thick, fill=white] (\j,\i) circle (.25);

      \draw[line width=.5mm,red] (1,6) to (1,5.5) to [out=-90,in=180] (1.5,5) to 
	(7.5,5) to [out=0,in=90] (9,4.5) to (9,3.5) to [out=-90,in=180] (9.5,3) to (12,3);

    \draw[line width=.5mm,blue] (5,6) to (5,3.5) to [out=-90,in=180] (6.5,3) to (8.5,3) to [out=0,in=90] (9,2.5) to (9,1.5) to [out=-90,in=180] (9.5,1) to (12,1);

    \draw[line width=.5mm,green!80!black] (11,6) to (11,5.5) to [out=-90,in=180] (11.5,5) to (12,5);

	\foreach \i/\j in {1/6,2/5,4/5,6/5,8/5,9/4,10/3,12/3}
		 \draw[line width=.5mm,red,fill=white] (\i,\j) circle (.25);

	\foreach \i/\j in {5/6,5/4,6/3,8/3,9/2,10/1,12/1}
		 \draw[line width=.5mm,blue,fill=white] (\i,\j) circle (.25);

         \foreach \i/\j in {11/6,12/5}
		 \draw[line width=.5mm,green,fill=white] (\i,\j) circle (.25);

\foreach \i/\j in {1/1,3/1,5/1,7/1,9/1,11/1}
\node at (\i,\j) [circle,fill,inner sep=2.5pt]{};
\foreach \i/\j in {1/1,3/1,5/1,7/1,9/1,11/1}
\node at (\i,3) [circle,fill,inner sep=2.5pt]{};
\foreach \i/\j in {1/1,3/1,5/1,7/1,9/1,11/1}
\node at (\i,5) [circle,fill,inner sep=2.5pt]{};
\foreach \i/\j in {1/1,3/1,5/1,7/1,9/1,11/1}
\node at (\i,0) {{\bf +}};
\foreach \i/\j in {3/1,7/1,9/1}
\node at (\i,6) {{\bf +}};
\foreach \i/\j in {0/1,0/3,0/5}
\node at (0,\j) {{\bf +}};
	\foreach \j/\c in {1/5,3/4,5/3,7/2,9/1,11/0}
	\node at (\j,6.5) {$\c$};
	\foreach \i/\c in {1/3,3/2,5/1}
	\node at (-.5,\i) {$\c$};

\end{tikzpicture}
\vspace{4mm}
\caption{An open state for $\mathfrak{S}^\circ_{(3,2,0),(1\,2\,3)}$ with  $\text{red} > \text{blue} > \text{green}.$}
\label{openstatefig}
\end{figure}
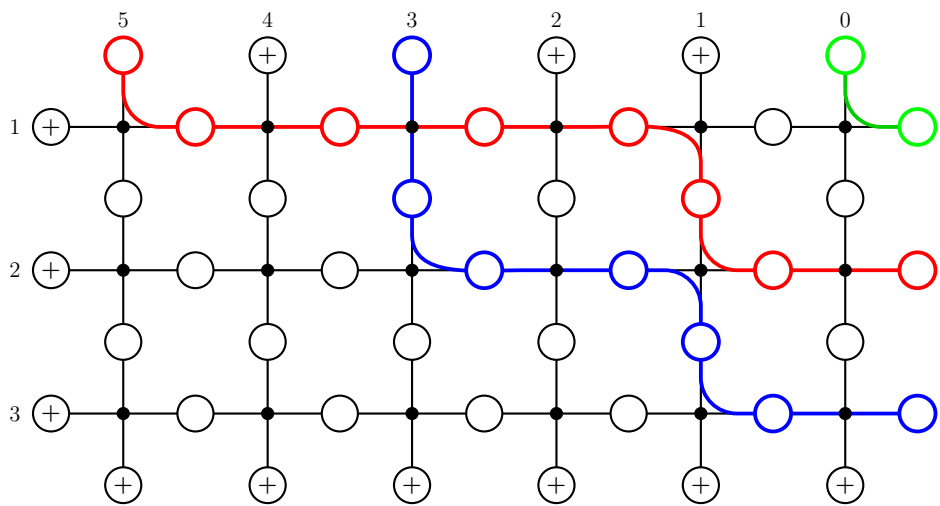

Based on the admissible configurations at a vertex given above, it turns out that we may describe admissible states in terms of {\em colored paths} through the vertices, as visualized in  \cref{closedstateex} and \cref{openclosed}, by thinking of all the edges with spin $c_i$ as connected together. For example in \cref{table2}, in the $a_{2,1}$ pattern above, the red path goes into the vertex from the left, and then goes out from the right, while the blue path goes in from above and goes out from below. 

More precisely, say $\mathfrak{s}$ is an admissible state of a model $\mathfrak{S}^\bullet_{\lambda,w}$. For $1\leq i\leq r$, the {\em   path in color $c_i$}, denoted by $p_i$, is the set of edges that are assigned to the spin $c_i$ in the state $\mathfrak{s}$. With the paths oriented downward and to the right, the edges with spin $c_i$ form a connected path in the grid, which starts at the top boundary edge and ends at the right boundary edge.

\vspace{2mm}

A key idea for the proof is to compare and perform adjustments between closed states and open states, instead of working only within one particular model. We will define a closed adjustment in \cref{sec:state}, which converts an open state to a closed state without changing the associated Gelfand-Tsetlin pattern. Hence there is a map 
\[\mathfrak{S}_{\lambda}^\circ(w)\xrightarrow{\text{closed adjustment}} \mathfrak{S}_{\lambda}^\bullet(w).\] To describe the intermediate stages of a series of closed adjustments, it is helpful to relax the local constraints and consider a larger set of (not necessarily admissible) states for both models. 

\begin{defn}
    Recall that a state (not necessarily admissible) is an assignment of spins to all the interior edges. We say such a state $\mathfrak{s}$ is a {\bf generalized state} if the local configuration at each vertex is one contained in the following \cref{table11}. 

 To distinguish from admissible states, we will henceforth decorate generalized states with a tilde, for example, $\tilde{\mathfrak{s}}$.
\end{defn}

\begin{rem}
       Since the local configurations in \cref{table11} include those in \cref{table2} and \cref{table3}, generalized states simultanueously generalize admissibile states for closed and open models. 
\end{rem}
\begin{table}[h]
\begin{center}

\begin{tabular}{ |c|c|c|c|c|c|c| } 
\hline
$a_1$ & $a_{2,1}$ & $a_{2,2}$ & $a_{2,3}$ & $a_{2,4}$ & $a_{2,5}$ & $b_1$ \\ 
 \hline
 \begin{tikzpicture}[scale=.65,every node/.style={scale=.5}]

\foreach \i in {1}
	\draw[thick] (0,\i) -- (2,\i);
\foreach \j in {1}
	\draw[thick] (\j,0) to (\j,2);
    \foreach \i/\j in {1/0,0/1,1/2,2/1}
  \draw[thick, fill=white] (\j,\i) circle (.25);
  \foreach \i/\j in {0/1,2/1}
\node at (\i,1) {{\bf +}};
  \foreach \i/\j in {1/2,1/0}
\node at (1,\j) {{\bf +}};
 \end{tikzpicture}  & \begin{tikzpicture}[scale=.65,every node/.style={scale=.5}]

\foreach \i in {1}
	\draw[thick] (0,\i) -- (2,\i);
\foreach \j in {1}
	\draw[thick] (\j,0) to (\j,2);
    
   \draw[line width=.5mm,red] (0,1) to (2,1);
   \draw[line width=.5mm,blue] (1,0) to (1,2);
   \foreach \i/\j in {1/0,0/1,1/2,2/1}
  \draw[thick, fill=white] (\j,\i) circle (.25);

 \end{tikzpicture}  & \begin{tikzpicture}[scale=.65,every node/.style={scale=.5}]

\foreach \i in {1}
	\draw[thick] (0,\i) -- (2,\i);
\foreach \j in {1}
	\draw[thick] (\j,0) to (\j,2);
    
   \draw[line width=.5mm,blue] (0,1) to (2,1);
   \draw[line width=.5mm,red] (1,0) to (1,2);
   \foreach \i/\j in {1/0,0/1,1/2,2/1}
  \draw[thick, fill=white] (\j,\i) circle (.25);

 \end{tikzpicture} & \begin{tikzpicture}[scale=.65,every node/.style={scale=.5}]

\foreach \i in {1}
	\draw[thick] (0,\i) -- (2,\i);
\foreach \j in {1}
	\draw[thick] (\j,0) to (\j,2);
    
   \draw[line width=.5mm,blue] (1,2) to (1,1.5) to [out=-90,in=180] (1.5,1) to (2,1);
 \draw[line width=.5mm,red] (0,1) to (0.5,1) to [out=0,in=90] (1,0.5) to (1,0);
   \foreach \i/\j in {1/0,0/1,1/2,2/1}
  \draw[thick, fill=white] (\j,\i) circle (.25);

 \end{tikzpicture} & \begin{tikzpicture}[scale=.65,every node/.style={scale=.5}]

\foreach \i in {1}
	\draw[thick] (0,\i) -- (2,\i);
\foreach \j in {1}
	\draw[thick] (\j,0) to (\j,2);
    
   \draw[line width=.5mm,red] (1,2) to (1,1.5) to [out=-90,in=180] (1.5,1) to (2,1);
 \draw[line width=.5mm,blue] (0,1) to (0.5,1) to [out=0,in=90] (1,0.5) to (1,0);
   \foreach \i/\j in {1/0,0/1,1/2,2/1}
  \draw[thick, fill=white] (\j,\i) circle (.25);

 \end{tikzpicture} & \begin{tikzpicture}[scale=.65,every node/.style={scale=.5}]

\foreach \i in {1}
	\draw[thick] (0,\i) -- (2,\i);
\foreach \j in {1}
	\draw[thick] (\j,0) to (\j,2);
    
   \draw[line width=.5mm,red] (0,1) to (2,1);
   \draw[line width=.5mm,red] (1,0) to (1,2);
   \foreach \i/\j in {1/0,0/1,1/2,2/1}
  \draw[thick, fill=white] (\j,\i) circle (.25);

 \end{tikzpicture}  & \begin{tikzpicture}[scale=.65,every node/.style={scale=.5}]

\foreach \i in {1}
	\draw[thick] (0,\i) -- (2,\i);
\foreach \j in {1}
	\draw[thick] (\j,0) to (\j,2);

   \draw[line width=.5mm,red] (1,0) to (1,2);
   \foreach \i/\j in {1/0,0/1,1/2,2/1}
  \draw[thick, fill=white] (\j,\i) circle (.25);

 \end{tikzpicture}  \\ 
 \hline
 $b_1$ & \multicolumn{2}{|c|}{$b_2$} & \multicolumn{2}{|c|}{$c_1$} & \multicolumn{2}{|c|}{$c_2$} \\ 
 \hline
 \begin{tikzpicture}[scale=.65,every node/.style={scale=.5}]

\foreach \i in {1}
	\draw[thick] (0,\i) -- (2,\i);
\foreach \j in {1}
	\draw[thick] (\j,0) to (\j,2);

   \draw[line width=.5mm,blue] (1,0) to (1,2);
   \foreach \i/\j in {1/0,0/1,1/2,2/1}
  \draw[thick, fill=white] (\j,\i) circle (.25);

 \end{tikzpicture} & \begin{tikzpicture}[scale=.65,every node/.style={scale=.5}]

\foreach \i in {1}
	\draw[thick] (0,\i) -- (2,\i);
\foreach \j in {1}
	\draw[thick] (\j,0) to (\j,2);
    
   \draw[line width=.5mm,red] (0,1) to (2,1);

   \foreach \i/\j in {1/0,0/1,1/2,2/1}
  \draw[thick, fill=white] (\j,\i) circle (.25);
  \foreach \i/\j in {1/2,1/0}
\node at (1,\j) {{\bf +}};
 \end{tikzpicture} &  \begin{tikzpicture}[scale=.65,every node/.style={scale=.5}]

\foreach \i in {1}
	\draw[thick] (0,\i) -- (2,\i);
\foreach \j in {1}
	\draw[thick] (\j,0) to (\j,2);
    
   \draw[line width=.5mm,blue] (0,1) to (2,1);

   \foreach \i/\j in {1/0,0/1,1/2,2/1}
  \draw[thick, fill=white] (\j,\i) circle (.25);
  \foreach \i/\j in {1/2,1/0}
\node at (1,\j) {{\bf +}};
 \end{tikzpicture} &  \begin{tikzpicture}[scale=.65,every node/.style={scale=.5}]

\foreach \i in {1}
	\draw[thick] (0,\i) -- (2,\i);
\foreach \j in {1}
	\draw[thick] (\j,0) to (\j,2);
    
   \draw[line width=.5mm,red] (0,1) to (0.5,1) to [out=0,in=90] (1,0.5) to (1,0);

   \foreach \i/\j in {1/0,0/1,1/2,2/1}
  \draw[thick, fill=white] (\j,\i) circle (.25);
  \foreach \i/\j in {1/2,2/1}
\node at (\i,\j) {{\bf +}};
 \end{tikzpicture} & \begin{tikzpicture}[scale=.65,every node/.style={scale=.5}]

\foreach \i in {1}
	\draw[thick] (0,\i) -- (2,\i);
\foreach \j in {1}
	\draw[thick] (\j,0) to (\j,2);
    
   \draw[line width=.5mm,blue] (0,1) to (0.5,1) to [out=0,in=90] (1,0.5) to (1,0);

   \foreach \i/\j in {1/0,0/1,1/2,2/1}
  \draw[thick, fill=white] (\j,\i) circle (.25);
  \foreach \i/\j in {1/2,2/1}
\node at (\i,\j) {{\bf +}};
 \end{tikzpicture}&  \begin{tikzpicture}[scale=.65,every node/.style={scale=.5}]

\foreach \i in {1}
	\draw[thick] (0,\i) -- (2,\i);
\foreach \j in {1}
	\draw[thick] (\j,0) to (\j,2);
    
   \draw[line width=.5mm,red] (1,2) to (1,1.5) to [out=-90,in=180] (1.5,1) to (2,1);

   \foreach \i/\j in {1/0,0/1,1/2,2/1}
  \draw[thick, fill=white] (\j,\i) circle (.25);
  \foreach \i/\j in {0/1,1/0}
\node at (\i,\j) {{\bf +}};
 \end{tikzpicture}&  \begin{tikzpicture}[scale=.65,every node/.style={scale=.5}]

\foreach \i in {1}
	\draw[thick] (0,\i) -- (2,\i);
\foreach \j in {1}
	\draw[thick] (\j,0) to (\j,2);
    
   \draw[line width=.5mm,blue] (1,2) to (1,1.5) to [out=-90,in=180] (1.5,1) to (2,1);

   \foreach \i/\j in {1/0,0/1,1/2,2/1}
  \draw[thick, fill=white] (\j,\i) circle (.25);
  \foreach \i/\j in {0/1,1/0}
\node at (\i,\j) {{\bf +}};
 \end{tikzpicture}  \\ 
 \hline

\end{tabular}

\vspace{3mm}
\caption{Admissible local configurations for a generalized state, with red $>$ blue.}
\label{table11}
\end{center}
\end{table}

Given a generalized state $\tilde{\mathfrak{s}}$, there are at most two different colors on the adjacent edges at each vertex. Suppose that at vertex $v$ there are two colors around, denoted by $c_i,c_j$ with $i<j$ (so $c_i>c_j$). We say the two paths $p_i$ and $p_j$ {\em cross} at the vertex $v$, if it is an
$a_{2,1}$ or $a_{2,2}$ pattern at the vertex. We say the two paths $p_i$ and $p_j$ {\em intersect} at the vertex $v$, if it is an $a_{2,1}, a_{2,2}, a_{2,3},$ or $a_{2,4}$ pattern at the vertex.

\begin{defn}\label{def:states}
Let $\tilde{\mathfrak{s}}$ be a generalized state. The state $\tilde{\mathfrak{s}}$ is called
 
(1). {\bf reduced} if no patterns $b_1$ are allowed, and any two paths cross at most once.

(2). {\bf  open} if no patterns $a_{2,2}$, $a_{2,3}$, or $b_1$ are allowed at any vertices.

(3). {\bf closed} if no patterns $a_{2,2}$, $a_{2,4}$, or $b_1$ are allowed at any vertices.

\noindent 
We will denote the set of all reduced states by $\mathfrak{S}^\text{red}_{\lambda}(w).$
\end{defn}

\begin{rem}
    The definition of an open state in \cref{def:states} is equivalent to the definition of an admissible state of $\mathfrak{S}^\circ_{\lambda,w}$.  The definition of a closed state in \cref{def:states} is equivalent to the definition of an admissible state of $\mathfrak{S}^\bullet_{\lambda,w}$. 
\end{rem}

With the definitions above, inspection of the possible local configurations gives the following properties of colored paths in open/closed states.

\begin{prop}\label{prop:opend}
    In an open state, if two paths intersect, they must cross at the first vertex at which they intersect. Furthermore, no two paths can cross more than once.
\end{prop}

\begin{proof}
Let $p_i$ and $p_j$ be two paths. Without loss of generality assume $i<j$. Now for any vertex that $p_i$ and $p_j$ intersect, we have the following observations by inspecting the admissible configurations in open states:

\begin{enumerate}
    \item If color $c_i$ goes in from the left, it must go out to the right.

    \item If color $c_i$ goes in from above, then it must go out from the right.
\end{enumerate}

By the two rules above, there is only one way to draw the picture as follows: the colored
paths cross the first time they meet but never again.
    
\end{proof}

\begin{prop}\label{prop:closedd}
    In a closed state, if two paths cross, they must cross at the last vertex at which they intersect. Furthermore, any two paths cannot cross more than once.
\end{prop}

\begin{proof}

Let $p_i$ and $p_j$ be two paths. Without loss of generality assume $i<j$. Now for any vertex at which $p_i$ and $p_j$ intersect, we have the following observations by inspecting the admissible configurations in closed states:

\begin{enumerate}
    \item If color $c_i$ goes in from the left, it could go out to the right or downward.

    \item The color $c_i$ cannot go in from above.
\end{enumerate}

    By the two rules above, except at the last intersection, the red path cannot go out from the
right, since then at the next intersection the red path will go in from above, which is not allowed.
At the last intersection, the red path can go to the right or downward, so there are two
possibilities.
\end{proof}

Here is an immediate corollary:

\begin{cor}
    An open state is a reduced state. A closed state is a reduced state. \hfill \qed
\end{cor}

\subsection{Gelfand-Tsetlin patterns and semistandard Young tableaux}\label{subsec:gtp}  \hfill

Let $\lambda=(\lambda_1,...,\lambda_r)$ be a dominant weight. In this section we will relate Gelfand-Tsetlin patterns and semistandard Young tableaux to states of colored models.

Gelfand–Tsetlin patterns are triangular arrays of integers satisfying certain inequalities. More precisely, a {\em Gelfand-Tsetlin pattern of size $r$} is a collection of integers $A_{i,j}$'s with $1 \leq i\leq r$ and $1\leq j\leq r-i+1$. Each row is weakly decreasing and adjacent rows are interleaving in the sense that if $(\lambda_1,...,\lambda_k)$ and $(\mu_1,...,\mu_{k-1})$ are numbers in two adjacent rows, then 
\[\lambda_1\geq \mu_1\geq \lambda_2\geq \mu_2\geq \dots \geq \lambda_{k-1}\geq \mu_{k-1}\geq \lambda_k.\]

A Gelfand-Tsetlin pattern is {\em left-strict} if its entries satisfy $A_{i,j}>A_{i+1,j}\geq A_{i,j+1}$.

A {\em Young diagram} in shape $\lambda$ is a finite collection of boxes, or cells, arranged in left-justified rows, with $\lambda_i$ boxes in the $i$-th row. A {\em semistandard Young tableau $T$ of shape $\lambda$} is a filling of boxes in the Young diagram by $\{1,2,3,...,r\}$ such that the rows are weakly increasing, and the columns are strictly increasing. The weight of $T$ is $(\mu_1,...,\mu_r)$ where $\mu_i$ is the number of $i$’s in $T$. We denote the set of all semistandard Young tableaux of shape $\lambda$ by $\mathcal{B}_\lambda$.

The Gelfand-Tsetlin patterns, semistandard Young tableau, and states of colored models are closely related: Given any state $\mathfrak{s}$ of the model $\mathfrak{S}^\bullet_{\lambda,w}$ or $\mathfrak{S}^\circ_{\lambda,w}$, we can construct an associated Gelfand-Tsetlin pattern $\text{GTP}(\mathfrak{s})=(A_{ij})$ of size $r$ by the following rule:  The entry $A_{ij}$ will be the label of the $j$-th colored vertical edge from left to right above the $i$-th row of vertices. For example, both states in \cref{openclosed} have the following Gelfand-Tsetlin pattern

\[
\begin{Bmatrix}
5 & & 3 & & 0 \\
& 3 & & 1 & \\
&& 1 &&
\end{Bmatrix}.
\]

\begin{prop}
    If $\tilde{\mathfrak{s}}$ is a generalized state with no $b_1$ pattern, then $\text{GTP}(\tilde{\mathfrak{s}})$ is a left-strict Gelfand-Tsetlin pattern of size $r$ with first row $\lambda+\rho$.
\end{prop}

\begin{proof}
    By definition, above the first row, the edges are colored at the columns indexed by $\lambda+\rho$. Inductively, when we move from row $i$ to row $i+1$, the color at the right boundary of row $i$ will not show up on the vertical edges above row $i+1$, hence there will be one fewer entry in the pattern. By examining the admissible local configurations in \cref{table11}, a colored path cannot go straight down without crossing other paths at a vertex (no $b_1$ patterns), thus every colored path must appear at least one column to the right moving down a row. Inductively this implies that the Gelfand-Tsetlin pattern is left-strict.
\end{proof}

Given a semistandard Young tableau $T$ in $\mathcal{B}_\lambda$, there is associated a Gelfand-Tsetlin pattern with top row $\lambda$ as follows. The top row is the shape $\lambda$ of the diagram $T$. The $k$-th row is the shape of the diagram after deleting all the boxes with entries $\geq r-k+2$. Clearly it is a bijection, and the semistandard Young tableau for the Gelfand-Tsetlin pattern above is:
\[
T =
\begin{array}{|c|c|c|c|c|}
\hline
1 & 2 & 2 & 3 & 3 \\
\hline
2 & 3 & 3 \\
\cline{1-3}

\end{array}
\]
For more detail, we refer the readers to \cite{kirillov1995groups}.

We will need to consider all possible states with the same underlying Gelfand-Tsetlin pattern for different models. Hence we take the Gelfand-Tsetlin patterns into account and make the following definition.

\begin{defn}
    Let $A$ be a left-strict Gelfand-Tsetlin pattern with top row $\lambda+\rho$.

(1). Define $\mathfrak{S}^\bullet_{\lambda}(w;A) \subseteq \mathfrak{S}^\bullet_{\lambda}(w)$ to be the set of closed states $\mathfrak{s}$ such that $\text{GTP}(\mathfrak{s})=A$.

(2). Define $\mathfrak{S}^\circ_{\lambda}(w;A) \subseteq \mathfrak{S}^\circ_{\lambda}(w)$ to be the set of open states $\mathfrak{s}$ such that $\text{GTP}(\mathfrak{s})=A$.

(3). Define $\mathfrak{S}^\text{red}_{\lambda}(w;A) \subseteq \mathfrak{S}^\text{red}_{\lambda}(w)$ to be the set of reduced states $\mathfrak{s}$ such that $\text{GTP}(\mathfrak{s})=A$.
\end{defn}

With this definition, a result in \cite{BBBG} can be rephrased as follows:

\noindent
\begin{prop}{\rm(}{\cite[Proposition~4.1]{BBBG}}{\rm)}\label{thm:BBBG1}
    Let $A$ be a left-strict Gelfand-Tsetlin pattern with top row $\lambda+\rho$. There exists a unique element $w=w_A\in S_r$ such that $\mathfrak{S}^\circ_{\lambda}(y;A)$ is empty if $y\neq w_A$, and $\mathfrak{S}^\circ_{\lambda}(y;A)$ has a unique element $\mathfrak{s}_A$ if $y=w_A$. Therefore the Gelfand-Tsetlin pattern $A$ uniquely determines a boundary flag $w_A$ and an open state $\mathfrak{s}_A$.
\end{prop} 

\begin{proof}
    By inspecting the local admissible configurations, we see the states and the boundary conditions on the right edge (i.e., the flag) are uniquely determined by the boundary conditions on the top row. 
\end{proof}

\begin{rem}
    In fact, a model is called {\em deterministic} if for any vertex $v$, assignments of the spins for two adjacent edges determine the spins for the remaining two edges at the vertex. Open models are deterministic. However, closed models are not deterministic, since both $a_{2,1}$ and $a_{2,3}$ patterns are allowed. This fact also leads to different results in \cref{prop:opend} and \cref{prop:closedd}.

    Since open models are deterministic, every element $T\in\mathcal{B}_\lambda$ corresponds to a unique state of a model, thereby also determining the flag.
\end{rem}

In \cref{sec:state}, we will prove an analogue of \cref{thm:BBBG1} for closed states.

\section{{Existence and uniqueness of closed states}{}}
\label{sec:state}

With the setup in \cref{sec:model}, the main goal of this section is to prove the following theorem:

\begin{thm}\label{thm:state}
    Let $w_0$ be the longest element in $S_r$. Let $A$ be a left-strict Gelfand-Tsetlin pattern with top row $\lambda+\rho$ and $w=w_A$ be the associated element in the Weyl group. If $w_0y\leq w_0w_A$ (equivalently, $y\geq w_A$) in the Bruhat order of $S_r$, then $\mathfrak{S}^\bullet_{\lambda}(y;A)$ has a unique element. If $w_0y\nleq w_0w$ (equivalently, $y\ngeq w_A$) then $\mathfrak{S}^\bullet_{\lambda}(y;A)$ is empty.
\end{thm} 

\begin{proof}
    The proof will be done in four steps. We first prove that $\mathfrak{S}^\bullet_{\lambda}(w_A;A)$ is nonempty in \cref{prop:wstate}. Next, we show $\mathfrak{S}^\bullet_{\lambda}(y;A)$ is nonempty if $y\geq w_A$ in the Bruhat order of $S_r$. This is \cref{prop:y>w}. We then show the converse: $\mathfrak{S}^\bullet_{\lambda}(y;A)$ is empty if $y\ngeq w_A$. This is \cref{prop:y<w}. Finally, we show the uniqueness also  in \cref{prop:y<w}.
\end{proof}

Suppose that $y\in S_r$ and $1\leq i<j\leq r$ so the color $c_i>c_j$. We say $(i,j)$ is a {\em crossing pair}, if color $c_i$ is above color $c_j$ on the right boundary. By the definition of the model, the color $c_i$ is at row $y(i)$ on the right boundary. Hence $(i,j)$ is a crossing pair iff $i<j$ and $y(i)<y(j)$ in $y\in S_r$, i.e., $(i,j)$ is an inversion in $w_0y.$

\begin{prop}
    Let $y\in S_r$ and $y{\bf c}$ be the boundary flag. If $(i,j)$ is an inversion in $w_0y$, then in any reduced state $\mathfrak{s}$, the paths $p_i$ and $p_j$ will cross. Conversely, if $p_i$ and $p_j$ cross in a reduced state $\mathfrak{s}$ then $(i,j)$ is a crossing pair. 
\end{prop} 
\begin{proof}
    Since $i<j$ we have $c_i$ is on the left of $c_j$ on the top boundary. If $p_i$ and $p_j$ cross, they cross exactly once, which makes $c_i$ be above $c_j$ on the right boundary. On the other hand, if $(i,j)$ is an inversion, $p_i$ and $p_j$ must cross to make it happen. 
\end{proof}

\begin{prop}
    Let $y\in S_r$ and $y{\bf c}$ be the boundary flag. Then $(i,j)$ is an inversion in $w_0y$ if and only if in any generalized state $\mathfrak{s}$, the path $p_i$ and $p_j$ will cross an odd number of times.
\end{prop}

\begin{proof}
    Similar to the previous proof. 
\end{proof}

\begin{conv}
    From now on we will fix a left-strict Gelfand-Tsetlin pattern $A$ and its associated flag $w_A$, by the result of \cref{thm:state}. 
 All the states will be required to have the underlying Gelfand-Tsetlin pattern $A$.
\end{conv}

Let $\mathfrak{s} \in \mathfrak{S}^\text{red}_{\lambda}(y)$. Let $p_i,p_j$ be two paths in $\mathfrak{s}$. If they cross once and intersect at multiple vertices, we do a  {\bf closed adjustment} to move the crossing to the last (furthest down and to the right) vertex at which they intersect without changing the underlying Gelfand-Tsetlin pattern. This process does not change spins for the boundary edges. The picture below is an illustration of a closed adjustment assuming red $>$ blue.

\noindent
\begin{figure}[h!]
\begin{centering}
\includegraphics[width=15cm]{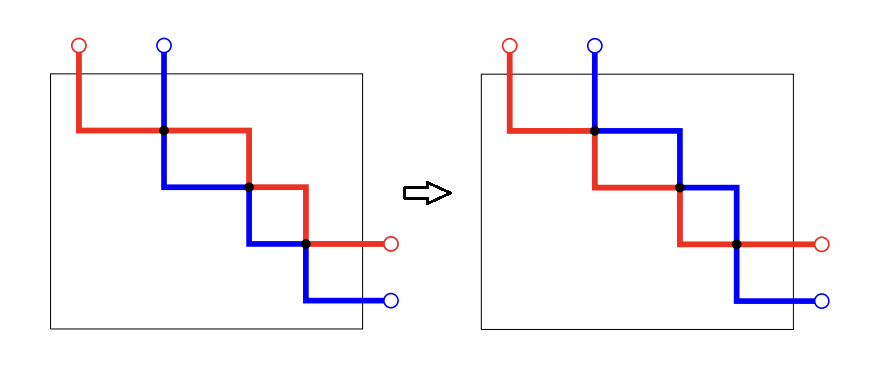}
\caption{Illustration of a closed adjustment, with red $>$ blue.}
\end{centering}
\end{figure}

In \cref{openclosed}, if we perform a closed adjustment to the open state, the result is the closed state.

\begin{lem}\label{lem:closedadj}
    Let $\mathfrak{s} \in \mathfrak{S}^\text{red}_{\lambda}(y)$. After performing a closed adjustment, the resulting state $\mathfrak{s}'$ is still in $\mathfrak{S}^\text{red}_{\lambda}(y)$. Moreover, any two paths cross after the adjustment iff they cross before.
\end{lem} 

\begin{proof}
From the definition, it is clear that a closed adjustment does not change the underlying Gelfand-Tsetlin pattern, hence $\mathfrak{s}'$ is a state for $\mathfrak{S}_{\lambda,y}$ with Gelfand-Tsetlin pattern $A$. We now show that $\mathfrak{s}'$ is reduced. First, if $p_i$ and $p_j$ don't cross, the adjustment does nothing.

Assume that $p_i$ and $p_j$ cross. Since $\mathfrak{s}$ is reduced, they cross once. By the definition of closed adjustments, $p_i$ and $p_j$ still cross once after the adjustment.

Now fix $k\neq i,j$. We investigate how many times $p_k$ crosses $p_i$ after the adjustment (the inspection for $p_j$ is the same). For clarity, we will assume that the three paths $p_i,p_j,p_k$ are of colors red, blue, and green respectively. 

The adjustment will change a part of the blue path to red. Since $\mathfrak{s}$ is reduced, there is at most one crossing between blue and green on this part. Hence the effect on changing blue to red on this part, is to increase the number of crossings between red and green by 0 or 1. Similarly, changing a part of the red paths to color blue decreases the number of crossings between red and green by 0 or 1. Therefore, the total change in the number of crossings is $-1,0$ or $+1$. However, in the state $\mathfrak{s}'$, the parity of the number of crossings cannot change since the boundary flag is unchanged, i.e., the number of crossings is unchanged modulo 2. Hence the change of the number of crossings between red and green has to be $0$.

Clearly the adjustment will not affect the number of crossings between $p_k$ and $p_{\ell}$ if $k,\ell \notin \{i,j\}$. Hence we have shown that an adjustment does not affect how many times any two paths cross. Thus the state $\mathfrak{s}'$ is reduced, as desired.
\end{proof}

\noindent

\begin{lem}
    Let $\mathfrak{s} \in \mathfrak{S}^\text{red}_{\lambda}(y)$ be a reduced state such that any crossing occurs at the last time the two paths intersect. Then $\mathfrak{s}$ is a closed state.
\end{lem}

\begin{proof}

Recall that a state is closed if there are no $a_{2,2}$ or $a_{2,4}$ patterns at any vertex. For any $i<j$, we will investigate all the intersections between $p_i$ and $p_j$. 

If $p_i$ and $p_j$ do not cross, then the color $c_j$ is above $c_i$ on the right boundary. The paths may still intersect at several vertices. However, the no-crossing condition ensures that the path $p_j$ is entirely in the up-right region bounded by the path $p_i$, and all the intersection must have an $a_{2,3}$ pattern with path $p_i$ entering from the left and exiting to the bottom.

If $p_i$ and $p_j$ cross once, the crossing must be the path $p_i$ entering from the left and exiting to the right, which is an $a_{2,1}$ pattern. As the path $p_i$ crosses at the last vertex at which it intersects $p_j$, all the previous intersections are in $a_{2,3}$ patterns.
\end{proof}

\noindent
\begin{prop}\label{prop:wstate}
    There exists a closed state with underlying Gelfand-Tsetlin pattern $A$ for the colored model $\mathfrak{S}^\bullet_{\lambda,w_A}$. In other words, $\mathfrak{S}^\bullet_\lambda(w_A;A)$ is nonempty.
\end{prop} 

\begin{proof}
    By the result of \cite{BBBG}, there is an open state $\mathfrak{s}$ for $\mathfrak{S}^\circ_{\lambda,w_A}$. If there are any two paths $p_i$ and $p_j$ that intersect multiple times, do a closed adjustment for $p_i$ and $p_j$. This process must terminate in finitely many steps, since every nontrivial adjustment moves a crossing down and to the right, and the set of crossing pairs does not change. When there are no adjustments to make, we arrive at a closed model state by the previous Lemma. The state has flag $w{\bf c}$ and underlying Gelfand-Tsetlin pattern $A$ since closed adjustments do not change them.
\end{proof}

In fact, the proof of \cref{prop:wstate} can be generalized to the following result:

\noindent
\begin{prop}\label{prop:adjust}
    Let $\mathfrak{s} \in \mathfrak{S}^\text{red}_{\lambda}(y)$. After doing all possible closed adjustments, we end with a closed state in $\mathfrak{S}^\bullet_{\lambda}(y)$.  \hfill \qed
\end{prop}

Fro the remaining proof, we record the properties of the Bruhat order on permutation groups. They are summarized in the following proposition:

\begin{prop}{\rm(}{\cite{BB}}{\rm)}\label{bruhat}
    Let $S_r$ be the permutation group on $r$ elements. Then,
    \begin{enumerate}
    \item There is a longest element $w_0$ such that $w\mapsto ww_0$ and $w\mapsto w_0w$ are both anti-automorphisms of the Bruhat order.

    \item If an element $\sigma$ of $S_n$ is denoted by $\sigma(1),...,\sigma(n)$, and $t=(i\,j)$ is a transposition, then $\sigma t$ is obtained by interchanging $\sigma(i)$ and $\sigma(j)$, $t\sigma$ is obtained by interchanging $i$ and $j$.

    \item The length $\ell(\sigma)$ equals the number of inversions of $\sigma$.

    \item If $x< y$ in the Bruhat order then there exists $x< u_1< u_2< ...< y$ such that at every step the length increase by $1$ and $u_i=u_{i-1}t_i$ for some transposition $t_i$.
\end{enumerate}
\end{prop}

Also, we record the following lemma for clarity:
\begin{lem}
    Let $y\in S_r$ and $t=(i,j)$ be a transposition. Then the flag $(yt){\bf c}$ is obtained by switching the colors $c_i$ and $c_j$ in $y{\bf c}$, while the flag $(ty){\bf c}$ is obtained by switching the colors in row $i$ and row $j$ in $y{\bf c}$.
\end{lem}

\begin{proof}
    By definition, the color $c_i$ is at row $y(i)$ on the right boundary.  In $S_r$, suppose that $t=(i\,j)$. Then $yt$ is changing $y(i)$ and $y(j)$ while $ty$ is changing $i$ and $j$ in the outputs. Hence the result follows. 
\end{proof}

\noindent
\begin{prop}\label{prop:y>w}
    Let $y\geq w_A$. Then there exists a unique closed state with underlying Gelfand-Tsetlin pattern $A$ for the colored model $\mathfrak{S}^\bullet_{\lambda,y}$. In other words, $\mathfrak{S}^\bullet_{\lambda}(y;A)$ has a single element.
\end{prop}

\begin{proof}
For existence, we will inductively construct a closed state for $y$ by increasing length from $w_A$ to $y$. For the base case: by \cref{prop:wstate}, $\mathfrak{S}^\bullet_{\lambda}(w_A;A)$ is nonempty.

Suppose that $y>w_A$. By the induction hypothesis and the Bruhat order on $S_r$, we can find a transposition $t=(i,j)\in S_r$ such that $y>yt \geq w_A$, $\ell(y)=\ell(yt)+1$, and there exists a closed state $\mathfrak{s} \in \mathfrak{S}^\bullet_{\lambda}(yt;A)$. Assume $i<j$ as usual. Since $y>yt$, in $yt$ we have $yt(i)<yt(j)$, i.e., the color $c_i$ is above $c_j$ and the paths $p_i$ and $p_j$ cross in the state $\mathfrak{s}.$ Moreover, switching $yt(i)$ and $yt(j)$ will increase the length of $yt$ by $1$.

We do an adjustment (this is ``adjustment 2" in the introduction) for the paths $p_i$ and $p_j$: Since $\mathfrak{s}$ is a closed state, the paths cross at the last vertex at which they intersect. We swap the paths after this vertex, which gives us a state $\mathfrak{s}'$ for the model $\mathfrak{S}^\bullet_{\lambda,y}$. In $\mathfrak{s}'$, the paths $p_i$ and $p_j$ don't cross any more.

Now we will use proof by contradiction to show that $\mathfrak{s}'$ is still a reduced state. The idea is that, with the assumption that $\ell(y)=\ell(yt)+1$, the colors $c_i$ and $c_j$ are ``adjacent", so changing colors in paths $p_i$ and $p_j$ won't affect the number of crossings with paths in other colors.

Assume that it is not. Then there exist two paths that cross more than once. As in the proof of the previous lemma, the adjustment will change the number of crossings between two paths by $+1,0,$ or $-1$, hence for the state $\mathfrak{s}'$ to be not reduced, it must be in the case that the number of crossings between two paths increases from $1$ to $2$ (since in the reduced $\mathfrak{s}$, the number of crossings between two paths is $0$ or $1$). However, this will increase the length once again, hence contradicting the assumption that $\ell(y)=\ell(yt)+1$. More precisely, there are four cases:

\begin{enumerate}
    \item Paths $p_j$ and $p_k$ cross once in $\mathfrak{s}$, but cross twice in $\mathfrak{s}'$, and $j<k$. Thus in $yt$, $yt(j)<yt(k)$; while in $y$, $y(j)>y(k)$, contradicting $\ell(y)=\ell(yt)+1$.

    \item Paths $p_j$ and $p_k$ cross once in $\mathfrak{s}$, but cross twice in $\mathfrak{s}'$, and $k<j$. Thus in $yt$, $yt(k)<yt(j)$; while in $y$, $y(k)>y(j)$, contradicting $\ell(y)=\ell(yt)+1$.

    \item Paths $p_i$ and $p_k$ cross once in $\mathfrak{s}$, but cross twice in $\mathfrak{s}'$, and $k>i\,$. Similar to the previous cases.

    \item Paths $p_i$ and $p_k$ cross once in $\mathfrak{s}$, but cross twice in $\mathfrak{s}'$, and $k<i$. Similar to the previous cases.

\end{enumerate}

In all cases, we have a contradiction. Hence $\mathfrak{s}' \in \mathfrak{S}^\text{red}_{\lambda}(y;A)$. By \cref{prop:adjust}, the reduced state $\mathfrak{s}'$ can be adjusted to a closed state for $\mathfrak{S}^\bullet_{\lambda,y}$.

For uniqueness, suppose that $\mathfrak{s}_1,\mathfrak{s}_2$ are two closed states for the colored model $\mathfrak{S}^\bullet_{\lambda,y}$ with underlying Gelfand-Tsetlin pattern $A$. We will show that $\mathfrak{s}_1=\mathfrak{s}_2$.

By the proof of existence, $y\geq w_A$. We prove the statement inductively by decreasing the length of the permutation $y$ from $\ell(w_0)$ to $\ell(w_A)$. If the flag is $w_0{\bf c}$, it is clear that there is a unique closed state: different paths can not cross.

Assume the claim holds for all permutations of length $\geq k$ and we now consider the permutation $ys_i$, where $ys_i\geq w_A$ is of length $k-1$ and $y$ is of length $\geq k$. Since $ys_i$ has smaller length, in the flag $(ys_i){\bf c}$, the color $c_i$ is above the color $c_{i+1}$.
    
Let $\mathfrak{s} \in \mathfrak{S}^\bullet_{\lambda}(ys_i;A)$. The paths $p_i$ and $p_{i+1}$ must cross at the last vertex where they intersect. Now, we can alter the two paths $p_i$ and $p_{i+1}$ to make them not cross, and this will form a state for $\mathfrak{S}^\bullet_{\lambda,y}$. This state is closed since we are only changing paths for $c_i$ and $c_{i+1}$. Each of the two paths might intersect with other paths, but switching between $c_i$ and $c_{i+1}$ in such an intersection won't affect the order relation with other colors (for example, if $c_i>c_k$ then $c_{i+1}>c_k$), so the state is still a closed state. 

By the induction hypothesis, the closed state for $\mathfrak{S}^\bullet_{\lambda,y}$ is unique, hence the closed state for $\mathfrak{S}^\bullet_{\lambda,ys_i}$ is unique for colors other than $c_i,c_{i+1}$. Meanwhile, the state for two colors $c_i,c_{i+1}$ is also unique. The proof is done.

\end{proof}

\noindent
\begin{prop}\label{prop:y<w}
Let $y\in S_r$ such that $y\ngeq w_A$. Then there does not exist a closed state with underlying Gelfand-Tsetlin pattern $A$ for the model $\mathfrak{S}^\bullet_{\lambda,y}$. In other words, $\mathfrak{S}^\bullet_{\lambda}(y;A)$ is empty.
\end{prop}

\begin{proof}
    Recall that $w_A$ depends only on $A$. Let $X$ be the subset of $S_r$ consisting of the $y$'s such that $\mathfrak{S}^\bullet_{\lambda}(y;A)$ is nonempty. We want to show that if $y\in X$ then $y\geq w_A$.
    
    Let $y\in X$ and $\mathfrak{s} \in \mathfrak{S}^\bullet_{\lambda}(y;A)$. For any $1\leq i<j\leq r$, let $p_i,p_j$ be the corresponding paths in $\mathfrak{s}$. If $p_i,p_j$ intersect but don't cross, then $y(i)>y(j)$. We do an adjustment to make $p_i$ and $p_j$ cross at the last time that they intersect. This adjustment will result in a new state, and it will still be a reduced state: It is impossible for the number of crossings between any two paths to jump from $+1$ to $+2$. Therefore we obtain a reduced state  $\mathfrak{s}'\in \mathfrak{S}^\text{red}_{\lambda}(y(i,j);A)$. 

    We replace $y$ by $y(i,j)$ and repeat this process until every pair of intersecting paths crosses. Now, by a similar argument we can show that one can adjust this state by moving the crossing to the first intersection. This process does not change the boundary flag and results in an open state. By the uniqueness of the open state, we end up at $w_A.$ Hence we get a sequence of inequalities 
    \[y\geq y(i,j)\geq...\geq w_A,\] as desired.
\end{proof}

\begin{example}
   {\em  Consider the model in \cref{modelex}. In the Bruhat order of $S_3$,  $(1\,3)\geq (1\,2\,3)$, and the following figure is a closed state for $\mathfrak{S}^\bullet_{(3,2,0),(1\,3)}$ with the same GTP.}
\end{example}

\begin{figure}[htb]
\centering
\begin{tikzpicture}[scale=.85,every node/.style={scale=.7}]

\foreach \i in {1,3,5}
	\draw[thick] (0,\i) -- (12,\i);
\foreach \j in {1,3,5,7,9,11}
	\draw[thick] (\j,0) to (\j,6);

\foreach \i in {0,2,4,6}
	\foreach \j in {1,3,5,7,9,11}
	  \draw[thick, fill=white] (\j,\i) circle (.25);
\foreach \i in {1,3,5}
	\foreach \j in {0,2,4,6,8,10,12}
	  \draw[thick, fill=white] (\j,\i) circle (.25);

      \draw[line width=.5mm,red] (1,6) to (1,5.5) to [out=-90,in=180] (1.5,5) to 
	(4.5,5) to [out=0,in=90] (5,4.5) to (5,3.5) to [out=-90,in=180] (5.5,3) to (8.5,3)  to [out=0,in=90] (9,2.5) to (9,1.5) to [out=-90,in=180] (9.5,1) to (12,1);

    \draw[line width=.5mm,blue] (5,6) to (5,5.5) to [out=-90,in=180] (5.5,5) to (8.5,5) to  [out=0,in=90] (9,4.5) to (9,3.5) to [out=-90,in=180] (9.5,3) to (12,3);

    \draw[line width=.5mm,green!80!black] (11,6) to (11,5.5) to [out=-90,in=180] (11.5,5) to (12,5);

	\foreach \i/\j in {1/6,2/5,4/5,5/4,6/3,8/3,9/2,10/1,12/1}
		 \draw[line width=.5mm,red,fill=white] (\i,\j) circle (.25);

	\foreach \i/\j in {5/6,6/5,8/5,9/4,10/3,12/3}
		 \draw[line width=.5mm,blue,fill=white] (\i,\j) circle (.25);

         \foreach \i/\j in {11/6,12/5}
		 \draw[line width=.5mm,green,fill=white] (\i,\j) circle (.25);

\foreach \i/\j in {1/1,3/1,5/1,7/1,9/1,11/1}
\node at (\i,\j) [circle,fill,inner sep=2.5pt]{};
\foreach \i/\j in {1/1,3/1,5/1,7/1,9/1,11/1}
\node at (\i,3) [circle,fill,inner sep=2.5pt]{};
\foreach \i/\j in {1/1,3/1,5/1,7/1,9/1,11/1}
\node at (\i,5) [circle,fill,inner sep=2.5pt]{};
\foreach \i/\j in {1/1,3/1,5/1,7/1,9/1,11/1}
\node at (\i,0) {{\bf +}};
\foreach \i/\j in {3/1,7/1,9/1}
\node at (\i,6) {{\bf +}};
\foreach \i/\j in {0/1,0/3,0/5}
\node at (0,\j) {{\bf +}};
	\foreach \j/\c in {1/5,3/4,5/3,7/2,9/1,11/0}
	\node at (\j,6.5) {$\c$};
	\foreach \i/\c in {1/3,3/2,5/1}
	\node at (-.5,\i) {$\c$};

\end{tikzpicture}
\vspace{4mm}
\caption{A closed state for $\mathfrak{S}^\bullet_{(3,2,0),(1\,2\,3)}.$}
\end{figure}

\begin{rem}
    The proof of \cref{prop:y>w} actually constructs an injective map $\mathfrak{S}^\bullet_{\lambda}(y)\rightarrow \mathfrak{S}^\bullet_{\lambda}(w)$ provided that $y\leq w$ in the Bruhat order. Furthermore, the map preserves the underlying Gelfand-Tsetlin pattern, as in the following commutative diagram:
    \[
\begin{tikzcd}[row sep=large, column sep=huge]
  \mathfrak{S}^\bullet_{\lambda}(y) \arrow[r, "\text{adjustment}"] \arrow[dr] & \mathfrak{S}_{\lambda}^\bullet(w) \arrow[d] \\
  & \mathcal{B}_\lambda.
\end{tikzcd}
\]
\end{rem}

\section{{Crystals and Demazure operators}{}}
\label{sec:crystal}

In this section, we first briefly review Kashiwara crystals. We will then focus on the crystal of semistandard Young tableau and the Demazure crystals.

Fix a root system $\Phi$ with index set $I$ and weight lattice $\Lambda$. We decompose $\Phi$ into positive and negative roots, and let $\alpha_i$ ($i \in I$) be the simple positive roots. Let $\alpha_i^\vee$ denote the corresponding simple coroots and $s_i$ the corresponding simple reflections generating the Weyl group $W$.

The root system $\Phi,\Lambda$ may be regarded as coming from a complex reductive Lie group G with maximal torus $T$. Then we may identify $\Lambda$ with the group $X^*(T)$ of rational
characters of $T$. If $\mathbf{z} \in T$ and $\lambda\in \Lambda$ we will denote by $\mathbf{z}^\lambda$
the application of $\lambda$ to $\mathbf{z}$. Let $\mathcal{O}(T)$ be the set of polynomial functions on $T$. For each simple reflection $s_i$ with $i \in I$, we define the Demazure operator acting on $f \in \mathcal{O}(T)$ by
\[
\partial_i f(\mathbf{z}) = \frac{f(\mathbf{z}) - \mathbf{z}^{-\alpha_i} f(s_i \mathbf{z})}{1 - \mathbf{z}^{-\alpha_i}}. 
\]

These operators satisfy $\partial_i^2 = \partial_i = s_i \partial_i$. Given any $\mu \in \Lambda$, set $k = \langle \mu, \alpha_i^\vee \rangle$ so $s_i(\mu) = \mu - k \alpha_i$. Then the action on the monomial $\mathbf{z}^\mu$ is given by
\[
\partial_i \mathbf{z}^\mu =
\begin{cases}
    \mathbf{z}^\mu + \mathbf{z}^{\mu - \alpha_i} + \cdots + \mathbf{z}^{s_i(\mu)} & \text{if } k \geq 0, \\
    0 & \text{if } k = -1, \\
    -(\mathbf{z}^{\mu + \alpha_i} + \mathbf{z}^{\mu + 2 \alpha_i} + \cdots + \mathbf{z}^{s_i(\mu + \alpha_i)}) & \text{if } k < -1.
\end{cases} 
\]

\begin{defn}
    A {\bf Kashiwara crystal} of type $\Phi$ is a nonempty set $\mathcal{B}$ together with maps
   \[e_i,f_i:\mathcal{B}\rightarrow \mathcal{B}\cup \{0\},\]
      \[\varepsilon_i,\varphi_i:\mathcal{B}\rightarrow \mathbb{Z}\cup \{-\infty\},\]
      \[\textrm{wt}: \mathcal{B}\rightarrow \Lambda,\]
    where $i\in I$ and $0\notin \mathcal{B}$, satisfying the following conditions:

    \begin{enumerate}
        \item If $x,y\in\mathcal{B}$ then $e_i(x)=y$ if and only if $f_i(y)=x$. In this case, it is assumed that 
        \[\text{wt}(y)=\text{wt}(x)+\alpha_i\quad, \varepsilon_i(y)=\varepsilon_i(x)-1,\quad \varphi_i(y)=\varphi_i(x)+1.\]

        \item We require that
        \[\varphi_i(x)=\langle \text{wt}(x),\alpha_i^\vee\rangle +\varepsilon_i(x)\]
        for all $x\in \mathcal{B}$ and $i\in I$. In particular, $\varphi_i(x)=-\infty$ if and only if $\varepsilon_i(x)=-\infty$.  If $\varphi_i(x)=-\infty$, then we require that $e_i(x)=f_i(x)=0$.
    \end{enumerate}
    The map $\text{wt}$ is called
the weight map. The operators $e_i$ and $f_i$ are called crystal operators.

\end{defn}

\begin{example}
    The set $\mathcal{B}_\lambda$ of semistandard Young tableaux of shape $\lambda$ in the alphabet $\{1,2,...,r\}$ forms a Kashiwara crystal of type $A_{r-1}$.  We call $\mathcal{B}_\lambda$ a crystal of tableaux. For more detail, we refer the readers to \cite{bump2017crystal}.
\end{example}

The crystal $\mathcal{B}_\lambda$ has a connection with representations theory. If $\lambda$ is a dominant weight, then there is a unique irreducible representation $\pi_\lambda$ with highest weight $\lambda$, and they have the same character.

Let $\mathbb{Z}\mathcal{B}_\lambda$ be the free abelian group generated by $\mathcal{B}_\lambda.$ The Demazure operator $D_i$ on $\mathbb{Z}\mathcal{B}_\lambda$ is defined by linearly extending the following formula
\[
D_i v =
\begin{cases}
    v + f_i v + \cdots + f_i^k v & \text{if } k \geq 0, \\
    0 & \text{if } k = -1, \\
    -(e_i v + \cdots + e_i^{-k-1} v) & \text{if } k < -1,
\end{cases}
\]
for $v\in \mathcal{B}_\lambda$ where $k=\langle \text{wt}(v),\alpha_i^\vee\rangle.$

The Demazure operators satisfy $D_i^2=D_i$ and $s_iD_i=D_i$. They also satisfy the same braiding relations as the Weyl group $S_r$ of $A_{r-1}$. In fact, they are lifts of the Demazure operators $\partial_i$'s to crystals.

Littelmann \cite{littelmann} and Kashiwara \cite{kashiwara} proved the following refined Demazure character formula:

\begin{thm}\label{likas}
    For the crystal $\mathcal{B}_\lambda$ and Weyl group $W$:

\begin{itemize}
    \item[(i)] \textit{There exist subsets} $\mathcal{B}_\lambda(w)$ \textit{of} $\mathcal{B}$ \textit{indexed by} $w \in W$ \textit{such that} $\mathcal{B}_\lambda(1) = \{ v_\lambda \}$, $\mathcal{B}_\lambda(w_0) = \mathcal{B}$ \textit{and if} $s_i w > w$ \textit{then}
    \[
    \mathcal{B}_\lambda(s_i w) = \{ x \in \mathcal{B}_\lambda \mid e_i^k x \in \mathcal{B}(w) \text{ for some } k \geq 0\}.
    \]
    
    \item[(ii)] \textit{If} $S$ \textit{is} an equivalence class of elements of $\mathcal{B}_\lambda$ under the equivalence relation that $x \equiv y$ if $x = e_i^r y$ or $x = f_i^r y$ for some $r$, then  $\mathcal{B}_\lambda(w) \cap S$ is one of the three possibilities: $\varnothing, S$ \textit{or} $\{ u_S \}$, where $u_S$ is the unique highest weight element in $S$ characterized by $e_i(u_S) = 0$. 
    
    \item[(iii)] We have
    \[
    \sum_{x \in \mathcal{B}_\lambda(w)} \mathbf{z}^{\text{wt}(x)} = \partial_w \mathbf{z}^\lambda.
    \]
\end{itemize}
\end{thm}

\begin{rem}
We can also define operators $\mathfrak{D}_i$ on $\mathcal{B}_\lambda$ by
\[
\mathfrak{D}_i X = \{ x \in \mathcal{B} \mid e_i^k(x) \in X \text{ for some } k \geq 0\}. 
\]

Let $\lambda$ be a dominant weight. Let $w = s_{i_r} \cdots s_{i_1}$ be a reduced word for the Weyl group element $w$. Then \[
\mathcal{B}_\lambda(w) = \mathfrak{D}_{i_r} \cdots \mathfrak{D}_{i_1} \{ u_\lambda \},
\]
where $u_\lambda$ is the highest weight element of the normal crystal $\mathcal{B}_\lambda$ with highest weight $\lambda$. 
\end{rem}

\begin{defn}
    The subsets $\mathcal{B}_\lambda(w)$ in \cref{likas} are called {\bf Demazure crystals}.
\end{defn}

For $y\in W$, let $\mathcal{B}^\circ(y)$ be a collection of disjoint subsets of $\mathcal{B}$. We call these a family of \textit{crystal Demazure atoms} if
\[
\mathcal{B}_\lambda(w) = \bigcup_{y \leq w} \mathcal{B}_\lambda^\circ(y). 
\]
Furthermore, if a family of disjoint subsets $\mathcal{B}_\lambda^\circ(y)$ satisfies the condition above, then it is unique.

There is a map $\sigma:\mathcal{B}_\lambda \rightarrow W$ such that for $v\in \mathcal{B}_\lambda$, if $w=w_0\sigma(v)$ then $v\in \mathcal{B}^\circ_\lambda(w)$. The map $\sigma$ is related to the theory of Lascoux-Schützenberger keys. In each atom $\mathcal{B}_\lambda^\circ(y)$, there exists a unique {\em key tableau}, which has the following property: from left to right every column contains the next column as a subset elementwise. For more details we refer the readers to \cite{BBBG}.

\section{{Proof of the main theorem}{}}
\label{sec:main}

Fix a partition $\lambda=(\lambda_1,...,\lambda_r)$, and let $\mathcal{B}_{\lambda}$ be the crystals of semistandard Young tableaux in shape $\lambda.$ Let $W=S_r$ be the Weyl group and $W_\lambda$ be the stabilizer of $\lambda$ in $W$. Let $\sigma:\mathcal{B}_\lambda\rightarrow \mathcal{B}_\lambda$ be the Schützenberger involution. Let $\theta':\mathfrak{S}_\lambda^\bullet(y)\rightarrow \mathcal{B}_\lambda$ be the map sending a state $\mathfrak{s}$ to the tableau $\theta'(\mathfrak{s})$ obtained as follows: Let $\text{GTP}(\mathfrak{s})$ be the Gelfand-Tsetlin pattern associated with $\mathfrak{s}$. It is left-strict
since there are no $b_1$ patterns in the closed model so we may define $\text{GTP}^\circ(\mathfrak{s})$. Then there is
a tableau $T'$ associated with $\text{GTP}^\circ(\mathfrak{s})$ and we define $\theta'(\mathfrak{s})=T'$. Now $\theta:\mathfrak{S}_\lambda^\bullet(y)\rightarrow \mathcal{B}_\lambda$ is the map $\sigma \circ \theta'$.

We will prove our main result, \cref{mainthm} in this section. We will first restate it in a more relevant way:

\begin{thm}\label{mainthm2}
For any $y$ in the Weyl group $W=S_r$, the map $\theta$ is a bijective map from $\mathfrak{S}^\bullet_{\lambda}(y)$ to the Demazure crystal $\mathcal{B}_\lambda(y)$. 
\end{thm}

\begin{proof}

The proof will be by induction
on the Bruhat order. If $y = 1$ this is easily checked. Let $s_i$ be a simple reflection. Assume that $\theta(\mathfrak{S}^\bullet_\lambda(y))=\mathcal{B}_\lambda(y)$ and that $s_iy>y$. Inductively, we want to show that \[\theta(\mathfrak{S}^\bullet_\lambda(s_iy))=\mathcal{B}_\lambda(s_iy).\] We note that since $\mathfrak{S}^\bullet_\lambda(s_iy)$ and $\mathcal{B}_\lambda(s_iy)$ have the same character (see \cref{chfor} and \cref{eq:3}), it is sufficient to show that \[\theta(\mathfrak{S}^\bullet_\lambda(s_iy))\subseteq\mathcal{B}_\lambda(s_iy).\]
Now let $\mathfrak{s} \in \mathfrak{S}_\lambda^\bullet(s_iy)$.  We will prove the following two facts:

\textbf{Case 1.} If $e_i(\theta(\mathfrak{s})) \neq 0$ then there exists $\mathfrak{t}\in \mathfrak{S}_{\lambda}^\bullet(s_iy)$ such that $\theta(\mathfrak{t})=e_i(\theta(\mathfrak{s})).$ This will be proved in \cref{prop:case1}.

\textbf{Case 2.} If $e_i(\theta(\mathfrak{s})) = 0$ then there exists a recoloring of the state $\mathfrak{s}$ into $\mathfrak{S}_\lambda^\bullet(y)$, i.e., there exists $\mathfrak{t}\in \mathfrak{S}_{\lambda}^\bullet(y)$ such that $\theta(\mathfrak{s})=\theta(\mathfrak{t}).$ This will be proved in \cref{prop:case2}.

Assume that Case 1 and Case 2 are established. Then we can prove the main theorem.  Let $k=\varepsilon_i(\theta(\mathfrak{s}))$. Applying the result of Case 1 $k$ times, we obtain $\mathfrak{t} \in \mathfrak{S}_\lambda^\bullet(s_iy)$ such that \[\theta(\mathfrak{t})=e_i^k\theta(\mathfrak{s}).\] Then Case 2 applies to $\mathfrak{t}$, so we can recolor $\mathfrak{t}$ into $\mathfrak{t}'\in \mathfrak{S}_\lambda^\bullet(y).$ By induction \[\theta(\mathfrak{t}') \in \mathcal{B}_\lambda(y).\] Thus \[\theta(\mathfrak{t})\in \mathcal{B}_\lambda(y)\] since $\theta(\mathfrak{t})=\theta(\mathfrak{t}').$ This implies that 
\[\theta(\mathfrak{s})\in \mathcal{B}_\lambda(s_iy),\] 
by the characterization of Demazure crystals in \cref{likas}.

\end{proof}

\begin{rem}\label{strongrmk}
    
Note that it is not hard to show the map $\theta$ is weight-preserving: i.e., the Boltzmann weight of $\mathfrak{s}$ equals the weight of the Young tableau $\theta(\mathfrak{s})$. However, \cref{chfor}, which follows from the Yang-Baxter equation for the model, together with the existence of a weight-preserving map does not guarantee the image of $\theta$ is precisely the Demazure crystal, as there are distinct Young tableaux with the same weight.  The point of \cref{mainthm2} is that $\theta$ is a natural weight-preserving embedding of states of $\mathfrak{S}_{\lambda,w}^\bullet$ into $\mathcal{B}_\lambda$, and the image is $\mathcal{B}_\lambda(w).$ Therefore, \cref{mainthm} is a stronger result than \cref{chfor} and the existence of a weight-preserving $\theta$.
\end{rem}

To establish the two cases, we need to first describe how the crystal operator $e_i$ works. The crystal structure of $\mathcal{B}_\lambda$ is constructed via a morphism of crystals from the tensor product of crystals of rows. Translating how $e_i$ works on $\theta(\mathfrak{s})$ in $\mathcal{B}_\lambda$, the effect of $e_i$ on $\mathfrak{s}$ is to move a vertical edge above row $i$ to the right. Equivalently, the effect of $e_i$ on $\text{GTP}^\circ(\mathfrak{s})$ is to decrease an entry in the $i+1$-th row by $1$.

\begin{example}
   {\em Here is an example for Case 1.  Consider $\lambda=(2,1,0)$, $y=s_2\in S_3$, and the closed state $\mathfrak{s}_1$ given below. Then $s_1s_2>y$. The operator $e_1$ will move colored edges on and above row $2$ one unit to the right. In this case we need to ``recolor" to obtain a closed state after moving the edges.

   }

\[e_1\left(
\vcenter{\hbox{
\begin{tikzpicture}[scale=.55,every node/.style={scale=.7}]

\foreach \i in {1,3,5}
	\draw[thick] (0,\i) -- (10,\i);
\foreach \j in {1,3,5,7,9}
	\draw[thick] (\j,0) to (\j,6);

\foreach \i in {0,2,4,6}
	\foreach \j in {1,3,5,7,9}
	  \draw[thick, fill=white] (\j,\i) circle (.25);
\foreach \i in {1,3,5}
	\foreach \j in {0,2,4,6,8,10}
	  \draw[thick, fill=white] (\j,\i) circle (.25);

      \draw[line width=.5mm,red] (1,6) to (1,5.5) to [out=-90,in=180] (1.5,5) to 
	(4.5,5) to [out=0,in=90] (5,4.5) to (5,3.5) to [out=-90,in=180] (5.5,3) to (10,3);

    \draw[line width=.5mm,blue] (5,6) to (5,5.5) to [out=-90,in=180] (5.5,5) to (6.5,5) to  [out=0,in=90] (7,4.5) to (7,1.5) to [out=-90,in=180] (7.5,1) to (10,1);

    \draw[line width=.5mm,green!80!black] (9,6) to (9,5.5) to [out=-90,in=180] (9.5,5) to (10,5);

	\foreach \i/\j in {1/6,2/5,4/5,5/4,6/3,8/3,10/3}
		 \draw[line width=.5mm,red,fill=white] (\i,\j) circle (.25);

	\foreach \i/\j in {5/6,6/5,8/1,7/4,7/2,10/1}
		 \draw[line width=.5mm,blue,fill=white] (\i,\j) circle (.25);

     \foreach \i/\j in {9/6,10/5}
		 \draw[line width=.5mm,green,fill=white] (\i,\j) circle (.25);

\foreach \i/\j in {1/1,3/1,5/1,7/1,9/1}
\node at (\i,\j) [circle,fill,inner sep=2.5pt]{};
\foreach \i/\j in {1/1,3/1,5/1,7/1,9/1}
\node at (\i,3) [circle,fill,inner sep=2.5pt]{};
\foreach \i/\j in {1/1,3/1,5/1,7/1,9/1}
\node at (\i,5) [circle,fill,inner sep=2.5pt]{};
\foreach \i/\j in {1/1,3/1,5/1,7/1,9/1}
\node at (\i,0) {{\bf +}};
\foreach \i/\j in {3/1,7/1,9/1}
\node at (\i,6) {{\bf +}};
\foreach \i/\j in {0/1,0/3,0/5}
\node at (0,\j) {{\bf +}};
\draw[very thick,->,>=Stealth] (7.5,4) to (8.5,4);
	\foreach \j/\c in {1/5,3/4,5/3,7/2,9/1}
	\node at (\j,6.5) {$\c$};
	\foreach \i/\c in {1/3,3/2,5/1}
	\node at (-.5,\i) {$\c$};

\end{tikzpicture}}}\right)=
\vcenter{\hbox{
\begin{tikzpicture}[scale=.55,every node/.style={scale=.7}]

\foreach \i in {1,3,5}
	\draw[thick] (0,\i) -- (10,\i);
\foreach \j in {1,3,5,7,9}
	\draw[thick] (\j,0) to (\j,6);

\foreach \i in {0,2,4,6}
	\foreach \j in {1,3,5,7,9}
	  \draw[thick, fill=white] (\j,\i) circle (.25);
\foreach \i in {1,3,5}
	\foreach \j in {0,2,4,6,8,10}
	  \draw[thick, fill=white] (\j,\i) circle (.25);

      \draw[line width=.5mm,red] (1,6) to (1,5.5) to [out=-90,in=180] (1.5,5) to (8.5,5) to  [out=0,in=90] (9,4.5) to (9,3.5) to [out=-90,in=180] (9.5,3) to (10,3);

    \draw[line width=.5mm,blue] (5,6) to (5,4.5) to (5,3.5) to [out=-90,in=180] (5.5,3) to (6.5,3) to [out=0,in=90] (7,2.5) to (7,1.5) to [out=-90,in=180] (7.5,1) to (10,1);

    \draw[line width=.5mm,green!80!black] (9,6) to (9,5.5) to [out=-90,in=180] (9.5,5) to (10,5);

	\foreach \i/\j in {1/6,2/5,4/5,6/5,8/5,9/4,10/3}
		 \draw[line width=.5mm,red,fill=white] (\i,\j) circle (.25);

	\foreach \i/\j in {5/6,5/4,6/3,7/2,8/1,10/1}
		 \draw[line width=.5mm,blue,fill=white] (\i,\j) circle (.25);

     \foreach \i/\j in {9/6,10/5}
		 \draw[line width=.5mm,green,fill=white] (\i,\j) circle (.25);

\foreach \i/\j in {1/1,3/1,5/1,7/1,9/1}
\node at (\i,\j) [circle,fill,inner sep=2.5pt]{};
\foreach \i/\j in {1/1,3/1,5/1,7/1,9/1}
\node at (\i,3) [circle,fill,inner sep=2.5pt]{};
\foreach \i/\j in {1/1,3/1,5/1,7/1,9/1}
\node at (\i,5) [circle,fill,inner sep=2.5pt]{};
\foreach \i/\j in {1/1,3/1,5/1,7/1,9/1}
\node at (\i,0) {{\bf +}};
\foreach \i/\j in {3/1,7/1,9/1}
\node at (\i,6) {{\bf +}};
\foreach \i/\j in {0/1,0/3,0/5}
\node at (0,\j) {{\bf +}};
	\foreach \j/\c in {1/5,3/4,5/3,7/2,9/1}
	\node at (\j,6.5) {$\c$};
	\foreach \i/\c in {1/3,3/2,5/1}
	\node at (-.5,\i) {$\c$};

\end{tikzpicture}}}
\]
\end{example}

\begin{example}
   {\em  Here is an example for Case 2.  Consider $\lambda=(2,1,0)$, $y=s_2\in S_3$, and the closed state $\mathfrak{s}_2$ given below. Then $s_1s_2>y$ and $e_1(\theta(\mathfrak{s}_2))=0$. In this case,  we could indeed ``recolor" the state for $\mathfrak{S}_{\lambda,s_1s_2}^\bullet$ to obtain a closed state for $\mathfrak{S}_{\lambda,s_2}^\bullet$.}

\[
\vcenter{\hbox{
\begin{tikzpicture}[scale=.55,every node/.style={scale=.7}]

\foreach \i in {1,3,5}
	\draw[thick] (0,\i) -- (10,\i);
\foreach \j in {1,3,5,7,9}
	\draw[thick] (\j,0) to (\j,6);

\foreach \i in {0,2,4,6}
	\foreach \j in {1,3,5,7,9}
	  \draw[thick, fill=white] (\j,\i) circle (.25);
\foreach \i in {1,3,5}
	\foreach \j in {0,2,4,6,8,10}
	  \draw[thick, fill=white] (\j,\i) circle (.25);

      \draw[line width=.5mm,red] (1,6) to (1,5.5) to [out=-90,in=180] (1.5,5) to 
	(4.5,5) to [out=0,in=90] (5,4.5) to (5,3.5) to [out=-90,in=180] (5.5,3) to (10,3);

    \draw[line width=.5mm,blue] (5,6) to (5,5.5) to [out=-90,in=180] (5.5,5) to (8.5,5) to [out=0,in=90] (9,4.5) to (9,1.5) to [out=-90,in=180] (9.5,1) to (10,1);

    \draw[line width=.5mm,green!80!black] (9,6) to (9,5.5) to [out=-90,in=180] (9.5,5) to (10,5);

	\foreach \i/\j in {1/6,2/5,4/5,5/4,6/3,8/3,10/3}
		 \draw[line width=.5mm,red,fill=white] (\i,\j) circle (.25);

	\foreach \i/\j in {5/6,6/5,8/5,9/4,9/2,10/1}
		 \draw[line width=.5mm,blue,fill=white] (\i,\j) circle (.25);

     \foreach \i/\j in {9/6,10/5}
		 \draw[line width=.5mm,green,fill=white] (\i,\j) circle (.25);

\foreach \i/\j in {1/1,3/1,5/1,7/1,9/1}
\node at (\i,\j) [circle,fill,inner sep=2.5pt]{};
\foreach \i/\j in {1/1,3/1,5/1,7/1,9/1}
\node at (\i,3) [circle,fill,inner sep=2.5pt]{};
\foreach \i/\j in {1/1,3/1,5/1,7/1,9/1}
\node at (\i,5) [circle,fill,inner sep=2.5pt]{};
\foreach \i/\j in {1/1,3/1,5/1,7/1,9/1}
\node at (\i,0) {{\bf +}};
\foreach \i/\j in {3/1,7/1,9/1}
\node at (\i,6) {{\bf +}};
\foreach \i/\j in {0/1,0/3,0/5}
\node at (0,\j) {{\bf +}};
	\foreach \j/\c in {1/5,3/4,5/3,7/2,9/1}
	\node at (\j,6.5) {$\c$};
	\foreach \i/\c in {1/3,3/2,5/1}
	\node at (-.5,\i) {$\c$};

\end{tikzpicture}}} \xlongrightarrow{recolor}
\vcenter{\hbox{
\begin{tikzpicture}[scale=.55,every node/.style={scale=.7}]

\foreach \i in {1,3,5}
	\draw[thick] (0,\i) -- (10,\i);
\foreach \j in {1,3,5,7,9}
	\draw[thick] (\j,0) to (\j,6);

\foreach \i in {0,2,4,6}
	\foreach \j in {1,3,5,7,9}
	  \draw[thick, fill=white] (\j,\i) circle (.25);
\foreach \i in {1,3,5}
	\foreach \j in {0,2,4,6,8,10}
	  \draw[thick, fill=white] (\j,\i) circle (.25);

      \draw[line width=.5mm,red] (1,6) to (1,5.5) to [out=-90,in=180] (1.5,5) to  (10,5);

    \draw[line width=.5mm,blue] (5,6) to (5,3.5) to [out=-90,in=180] (5.5,3) to (8.5,3) to [out=0,in=90] (9,2.5) to (9,1.5) to [out=-90,in=180] (9.5,1) to (10,1);

    \draw[line width=.5mm,green!80!black] (9,6) to (9,3.5) to [out=-90,in=180] (9.5,3) to (10,3);

	\foreach \i/\j in {1/6,2/5,4/5,6/5,8/5,10/5}
		 \draw[line width=.5mm,red,fill=white] (\i,\j) circle (.25);

	\foreach \i/\j in {5/6,5/4,6/3,8/3,9/2,10/1}
		 \draw[line width=.5mm,blue,fill=white] (\i,\j) circle (.25);

     \foreach \i/\j in {9/6,9/4,10/3}
		 \draw[line width=.5mm,green,fill=white] (\i,\j) circle (.25);

\foreach \i/\j in {1/1,3/1,5/1,7/1,9/1}
\node at (\i,\j) [circle,fill,inner sep=2.5pt]{};
\foreach \i/\j in {1/1,3/1,5/1,7/1,9/1}
\node at (\i,3) [circle,fill,inner sep=2.5pt]{};
\foreach \i/\j in {1/1,3/1,5/1,7/1,9/1}
\node at (\i,5) [circle,fill,inner sep=2.5pt]{};
\foreach \i/\j in {1/1,3/1,5/1,7/1,9/1}
\node at (\i,0) {{\bf +}};
\foreach \i/\j in {3/1,7/1,9/1}
\node at (\i,6) {{\bf +}};
\foreach \i/\j in {0/1,0/3,0/5}
\node at (0,\j) {{\bf +}};
	\foreach \j/\c in {1/5,3/4,5/3,7/2,9/1}
	\node at (\j,6.5) {$\c$};
	\foreach \i/\c in {1/3,3/2,5/1}
	\node at (-.5,\i) {$\c$};

\end{tikzpicture}}}
\]

\end{example}

We give a more precise description of how $e_i$ works on $\text{GTP}^\circ(\mathfrak{s})$ in the following proposition:

\begin{prop}\label{crystaloperator}
Let $\mathfrak{s}$ be a closed state in $\mathfrak{S}^\bullet_\lambda(w)$ for some $w\in S_r$. Suppose the $i$-th, the $(i+1)$-th, and the $(i+2)$-th rows in $\text{GTP}^\circ(\mathfrak{s})$ are given in the following way:
\[x_1,x_2,...,x_{r-i+1},\]
\[y_1,y_2,...,y_{r-i},\]
\[z_1,z_2,...,z_{r-i-1}.\]
Define $k=r-i+1$, $z_{k-1}=0$ and $y_k=0$. We will compute the following quantities:
\[E_1=(y_{k-1}-z_{k-1})-(x_k-y_k)=y_{k-1}-x_k, \]
\[E_2=(y_{k-2}-z_{k-2})-(x_{k-1}-y_{k-1})+E_1,\]
\[E_3=(y_{k-3}-z_{k-3})-(x_{k-2}-y_{k-2})+E_2,\]
\[...\]
\[E_{r-i}=(y_1-z_1)-(x_2-y_2)+E_{r-i-1}.\]
Then we have $e_i(\theta({\mathfrak{s}}))=0$ if and only if 
\[E_j\leq 0,\]
for all $1\leq j\leq r-i$. 

Furthermore, suppose that \[\max_{1\leq j\leq r-i}(E_j)>0\] and $t$ is the smallest index such that 
\[E_t=\max_{1\leq j\leq r-i}(E_j)>0.\]
Then the crystal operator $e_i$ will decrease $y_{k-t}$ by 1.
\end{prop} 

\begin{proof}

    To set up the proof, recall that $\mathfrak{s}$ corresponds to a Gelfand-Tsetlin pattern $\text{GTP}^\circ(\mathfrak{s})$ and a tableau $T'=\theta'(\mathfrak{s})$ (see the beginning of the section). Due to the effect of the involution, $e_i(\theta(\mathfrak{s}))=0$ is equivalent to $f_{r-i}(T')=0$. We use the row-reading of the crystal of tableaux in Section 3.1 of \cite{bump2017crystal}. Thus the crystal $\mathcal{B}_\lambda$ is viewed as a seminormal subcrystal of $R_r\otimes R_{r-1}\otimes ...\otimes R_1$ where each $R_i$ is the crystal of the $i$-th row of $T'$. 
    
    A key ingredient of the proof is Lemma 2.33 of \cite{bump2017crystal}. Since $f_{r-i}(T')=0$ is further equivalent to $\varphi_{r-i}(T')=0$, Lemma 2.33 of \cite{bump2017crystal} gives the following formula:
    \begin{equation}\label{keyeq}
    \varphi_{r-i}(R_r\otimes R_{r-1}\otimes ...\otimes R_1)=\max_{j=1}^r\left(\sum_{h=1}^j\varphi_{r-i}(R_{r-h+1})-\sum_{h=1}^{j-1}\varepsilon_{r-i}(R_{r-h+1}) \right). \tag{$\star$}\end{equation}

    We first give a description of $\varphi_{r-i}(R_{r-h+1})$. By definition, it measures the number of entries equal to $r-i$ in row $r-h+1$. Now we translate it to $\text{GTP}^\circ(\mathfrak{s})$: the $k$-th row of $\text{GTP}^\circ(\mathfrak{s})$ is the shape of the tableau with entries $1,...,r-k+1$. Hence the entry $r-i$ corresponds to comparing the $(i+1)$-th row and the $(i+2)$-th row of $\text{GTP}^\circ(\mathfrak{s})$ (treat the $r+1$-th row as empty). 

    The two rows are
    \[y_1,...,y_{r-i}\]
    and 
    \[z_1,...,z_{r-i-1}.\] 
    Hence we have 
    \[\varphi_{r-i}(R_{r-h+1})=y_{r-h+1}-z_{r-h+1}.\]
    The first term in \cref{keyeq} where $\sum_{h=1}^j\varphi_{r-i}(R_{r-h+1})$ could be nonzero is when $j=i+1$ and $h=j$. Indeed, if $j\leq i$ then $h\leq i$ and $R_{r-h+1}$ can only contain entries $\geq r-h+1\geq r-i+1$.

    In the case $h=j=i+1$, 
    \[\sum_{h=1}^j\varphi_{r-i}(R_{r-h+1})=\varphi_{r-i}(R_{r-i})=y_{r-i}.\]

    Similarly, $\varepsilon_{r-i}$ measures the number of entries equal to $r-i+1$. Hence we will compare the $i$-th row and the $(i+1)$-th row of $\text{GTP}^\circ(\mathfrak{s})$, and we have
    \[\varepsilon_{r-i}(R_{r-h+1})=x_{r-h+1}-y_{r-h+1}.\]
    The first term in \cref{keyeq} where $\sum_{h=1}^{j-1} \varepsilon_{r-i}(R_{r-h+1})$ could be nonzero is when $j=i+1$ and $h=j-1$. In this case  $\varepsilon_{r-i}(R_{r-i})=x_{r-i+1}.$ 
    
    Thus, the first nonzero term on the right hand side of \cref{keyeq} is given by
    \[\sum_{h=1}^{i+1}\varphi_{r-i}(R_{r-h+1})-\sum_{h=1}^{i}\varepsilon_{r-i}(R_{r-h+1})=y_{r-i}-x_{r-i+1},\]
    which is precisely the definition of $E_1$. 
    All the other sums could be identified by the descriptions for $\varphi_{r-i}$ and $\varepsilon_{r-i}$ above. Thus $E_1,...,E_{r-i}$ are precisely the possibly nonzero terms in the formula of Lemma 2.33. Hence we finish the proof by \cref{keyeq} and  Lemma 2.33 of \cite{bump2017crystal}.
    
\end{proof}

We are now ready to prove both cases. For case 1, we need the following generalization of the closed adjustment set up in \cref{lem:closedadj}. Recall the definition of reduced states in \cref{def:states}: A generalized state is a {\em reduced state} if any two colored paths cross at most once, and no $b_1$ patterns occur. 

\begin{prop}\label{prop:adj}
     Let $\mathfrak{s}$ be a reduced state. Let $c_a,c_b$ be two different colors. For paths $p_a$ and $p_b$ that cross each other, we can always move the crossing to any vertex where the two paths intersect. This adjustment changes the state to a new reduced state $\mathfrak{s}'.$ Furthermore, this process does not change whether any two paths cross or not.
\end{prop}

\begin{proof}
    By definition of a reduced state, the paths $p_a$ and $p_b$ cross exactly once. Let $v$ be a vertex at which the two paths intersect. We want to recolor the two paths, i.e., re-assign spins of all the edges on the two paths, so that the new paths cross at $v$. Without changing the starting and ending colors, we let the two paths intersect at all the other vertices at which they meet, and cross at $v$. This is the recoloring process, and this process will not change the boundary conditions.

    Now it remains to show we still have a reduced state $\mathfrak{s}'$. The rest of the proof is exactly the same as in \cref{lem:closedadj}. The adjustment does not affect how many times any two paths cross. Thus the state $\mathfrak{s}'$ is a reduced state, as desired.
\end{proof}

The following result will prove case 1 of \cref{mainthm2}:

\begin{prop}\label{prop:case1}
    Let $\mathfrak{s}\in \mathfrak{S}_\lambda^\bullet(s_iy)$ and $e_i(\theta(\mathfrak{s}))\neq 0$. Then there exists $\mathfrak{t}\in \mathfrak{S}_{\lambda}^\bullet(s_iy)$ such that $\theta(\mathfrak{t})=e_i(\theta(\mathfrak{s})).$
\end{prop}
\begin{proof}
    The idea of the proof can be summarized by the following diagram:
    \[\mathfrak{s} \xrightarrow{\text{recolor}} \mathfrak{s}' \xrightarrow{\text{crystal operator  } e_i}\mathfrak{t}'\xrightarrow{\text{recolor}}\mathfrak{t},\] 
    where $\mathfrak{s}$ and $\mathfrak{s}'$ have the same Gelfand-Tsetlin pattern, and $\mathfrak{t}$ and $\mathfrak{t}'$ have the same Gelfand-Tsetlin pattern. Moreover, $\mathfrak{s},\mathfrak{t}$ are closed states, while $\mathfrak{s}'$ and $\mathfrak{t}'$ are possibly only reduced states. 
    
    Since $e_i(\theta(\mathfrak{s}))\neq 0$, $e_i$ will move a vertical colored edge one unit to the right. We will denote the edge that needs to be moved by $g$. 

    Consider the edges around a square in the grid of the model, with the edge $g$ being the left side. Denote the bottom left vertex by $v$. Since $e_i\neq 0$, the top and right edges of the square must be uncolored, while the left and bottom edges of the square are colored. Denote the bottom edge by $h$. The effect of $e_i$ is to change $g$ and $h$ to uncolored edges, while giving colors to the other two edges of the square. 

    If $g,h$ are of different colors, say $c_a,c_b$, then we consider paths $p_a$ and $p_b$. By admissible local configurations, the two paths must cross at the vertex $v$ of intersection between edges $g,h$. We will carefully show that in $\mathfrak{s}$, the two paths $p_a,p_b$ must intersect at least twice:

    \begin{lem}
        {\em The two paths $p_a$ and $p_b$ must intersect at least twice. }
    \end{lem}

    \begin{proof}
        For the sake of clarity, let the vertical edge $g$ be of color blue and the horizontal edge $h$ be of color red. We then have a crossing and have red $>$ blue. 
        
        As in \cref{crystaloperator}, consider the three rows in $\text{GTP}^\circ(\mathfrak{s})$:
        \[x_1,x_2,...,x_{k},\]
        \[y_1,y_2,...,y_{k-1},\]
        \[z_1,z_2,...,z_{k-2}.\]
        The condition that $e_i\neq 0$ means that $e_i$ will decrease one element $y_q$ by 1 (here $y_q$ corresponds to $g$ in the state $\mathfrak{s}$), where $1\leq q\leq r-i$. By \cref{crystaloperator}, we have the following inequalities:
        
        \[(y_{q-1}-z_{q-1})\leq (x_q-y_q),\]
        \[(y_{q-2}-z_{q-2})+(y_{q-1}-z_{q-1})\leq (x_{q-1}-y_{q-1})+(x_q-y_q),\]
        \[...\]
        \[\sum_{i=1}^{q-1} (y_i-z_i)\leq \sum_{i=1}^{q-1} (x_{i+1}-y_{i+1}).\]
        The hypothesis that we have a crossing at $v$ forces $y_q=z_{q-1}$. Hence $i\neq r-1$ and $q\neq 1$.

\begin{figure}[h!]
\centering
\begin{tikzpicture}[scale=.55,every node/.style={scale=.7}]

\foreach \i in {1,3,5}
	\draw[thick] (0,\i) -- (6,\i);
\foreach \j in {1,3,5}
	\draw[thick] (\j,0) to (\j,6);

\draw[line width=.7mm,red] (3,3) to (6,3);
\draw[line width=.7mm,blue] (3,5) to  (5,5) to (5,1);

\foreach \i/\j in {1/1,3/1,5/1}
\node at (\i,\j) [circle,fill,inner sep=2.5pt]{};
\foreach \i/\j in {1/1,3/1,5/1}
\node at (\i,3) [circle,fill,inner sep=2.5pt]{};
\foreach \i/\j in {1/1,3/1,5/1}
\node at (\i,5) [circle,fill,inner sep=2.5pt]{};

\node at (5.4,4) {$y_{q}$};
\node at (5.5,2) {$z_{q-1}$};
\node at (4.7,4) {$g$};
\node at (5.85,3.3) {$h$};

\node at (5.3,2.7) {$v$};
\end{tikzpicture}
\caption{Local spins near the vertex $v$. The edge $g$ corresponds to the entry $y_q$ in the Gelfand-Tsetlin pattern.}
\end{figure}
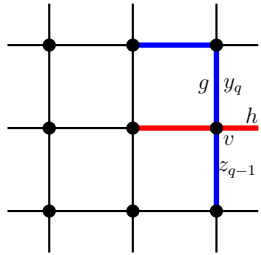

        Now by the description of $e_i$, we must have $y_{q-1}-z_{q-1}\leq x_q-y_q$. Since $y_{q}=z_{q-1}$, this implies $y_{q-1}\leq x_q$. But in a Gelfand-Tsetlin pattern we must have $y_{q-1}\geq x_q$, hence we have $y_{q-1}=x_q$.

        Furthermore, the fact that $e_i$ decreases $y_q$ by 1 also implies $y_q>x_{q+1}$. This means the horizontal edges from $x_q$ to $y_q$ are of color blue.

        Meanwhile, the horizontal edges from $y_{q-1}$ to $z_{q-1}$ are of color red. 

        Now since $y_{q-1}=x_q$, there is an intersection. If it is an intersection between $p_a$ and $p_b$ in $\mathfrak{s}$, then the proof is done (see \cref{fig:case1}). Otherwise, it is an intersection involving another color, say green (see \cref{fig:case2}). In this case, we must have $x_q=y_{q-1}=z_{q-2}$, and we iterate the process with $y_{q}=z_{q-1}$ replaced by $y_{q-1}=z_{q-2}$. This process will terminate when there are no new colors involved, and then we will find another intersection between $p_a$ and $p_b$.

\begin{figure}[h!]
\centering
\begin{tikzpicture}[scale=.55,every node/.style={scale=.7}]

\foreach \i in {1,3,5}
	\draw[thick] (0,\i) -- (6,\i);
\foreach \j in {1,3,5}
	\draw[thick] (\j,0) to (\j,6);

\draw[line width=.7mm,red] (1,5) to (3,5) to (3,3) to (6,3);
\draw[line width=.7mm,blue] (3,6) to (3,5) to  (5,5) to (5,1);

\foreach \i/\j in {1/1,3/1,5/1}
\node at (\i,\j) [circle,fill,inner sep=2.5pt]{};
\foreach \i/\j in {1/1,3/1,5/1}
\node at (\i,3) [circle,fill,inner sep=2.5pt]{};
\foreach \i/\j in {1/1,3/1,5/1}
\node at (\i,5) [circle,fill,inner sep=2.5pt]{};

\node at (3.5,6) {$x_q$};
\node at (3.5,4) {$y_{q-1}$};
\node at (5.4,4) {$y_{q}$};
\node at (5.5,2) {$z_{q-1}$};
\end{tikzpicture}
\caption{There is another intersection between red and blue at $x_q=y_{q-1}$. }
\label{fig:case1}
\end{figure}
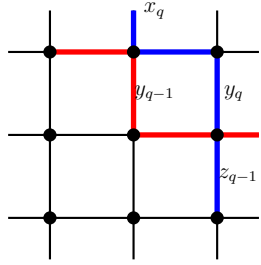

\begin{figure}[h!]
\centering
\begin{tikzpicture}[scale=.55,every node/.style={scale=.7}]

\foreach \i in {1,3,5}
	\draw[thick] (0,\i) -- (6,\i);
\foreach \j in {1,3,5}
	\draw[thick] (\j,0) to (\j,6);

\draw[line width=.7mm,red] (1,3) to (6,3);
\draw[line width=.7mm,blue] (1,5) to  (5,5) to (5,1);
\draw[line width=.7mm,green] (3,6) to (3,1);

\foreach \i/\j in {1/1,3/1,5/1}
\node at (\i,\j) [circle,fill,inner sep=2.5pt]{};
\foreach \i/\j in {1/1,3/1,5/1}
\node at (\i,3) [circle,fill,inner sep=2.5pt]{};
\foreach \i/\j in {1/1,3/1,5/1}
\node at (\i,5) [circle,fill,inner sep=2.5pt]{};

\node at (1.5,6) {$x_{q-1}$};
\node at (1.5,4) {$y_{q-2}$};
\node at (3.4,6) {$x_q$};
\node at (3.5,4) {$y_{q-1}$};
\node at (5.4,4) {$y_{q}$};
\node at (5.5,2) {$z_{q-1}$};
\end{tikzpicture}
\caption{At $x_q=y_{q-1}$, red and blue do not intersect. Then both paths have to go to the left and we look at $x_{q-1}=y_{q-2}$ next. }
\label{fig:case2}
\end{figure}
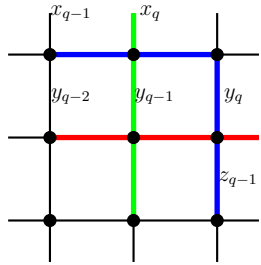

    \end{proof}

    By \cref{prop:adj}, since $p_a,p_b$ intersect at least twice, we could recolor $p_a,p_b$ so that they cross elsewhere, obtaining a new generalized state $\mathfrak{s}'$. Now in $\mathfrak{s}'$, the edges $g,h$ are of the same color.

    Now we could move the two colored edges $g,h$ to the other two edges in the square, obtaining $\mathfrak{t}'$. Clearly this process will not introduce more crossings between any two colored paths, as can be checked by inspecting all four vertices. Hence $\mathfrak{t}'$ is a reduced state. Finally, using \cref{prop:adj} again, we could recolor $\mathfrak{t}'$ into a closed state $\mathfrak{t}$, and the proof of \cref{prop:case1} is done. 

\end{proof}

Let $\mathfrak{s}$ be an open state of the model $\mathfrak{S}^\circ_{\lambda,w}$. Since open states are determined by the underlying Gelfand-Tsetlin pattern, one may ask how to read $w$ off from $\text{GTP}^\circ(\mathfrak{s})$. Here is a possible answer to the question.

\begin{defn}\label{def:y}
    {\em Let $A$ be a reduced Gelfand-Tsetlin pattern with top row $\lambda=(\lambda_1,\lambda_2,...,\lambda_r)$. We define a permutation $\tau\in S_r$ by the following inductive approach: moving from row 1 to row 2, $\tau(1)$ is defined to be the index such that $\lambda_{\tau(1)}$ is the rightmost element in row 1 that does not equal the element on the left bottom of it. Then we map the remaining elements in row 1 in order to elements in row 2, and repeat the definition inductively.  }
\end{defn}

\begin{example}
  {\em 
    Let $\lambda=(8,6,5,0)$. Consider the following Gelfand-Tsetlin pattern
    \[\begin{pmatrix}
        8& &{\color{red} 6}&&5&&0\\
        &8&&5&&{\color{red} 0}\\
        &&6 &&{\color{red} 2}\\
        &&&{\color{red} 4}
    \end{pmatrix}\] In each row, the red element is the rightmost element that does not equal the element on the left bottom of it. Then $\tau(1)=2, \tau(2)=4, \tau(3)=3, \tau(4)=1$.}
    \end{example}

\begin{prop}\label{prop:compare}
    Let $A$ be a reduced Gelfand-Tsetlin pattern with top row $\lambda$. By \cref{thm:BBBG1}, there exists $w\in S_r$ and an open state $\mathfrak{s}\in \mathfrak{S}^\circ_\lambda(w)$ such that $\text{GTP}^\circ(\mathfrak{s})=A$. Let $\tau$ be the permutation associated to $A$ in \cref{def:y}. Then \[\tau=w^{-1}.\]
\end{prop}

\begin{proof}
    Suppose that there is an intersection at vertex $v$ in an open state $\mathfrak{s}$. Then the vertical edges above $v$ and below $v$ are both colored, leading to two entries recording the column in $\text{GTP}^\circ(\mathfrak{s})$. By how the Gelfand-Tsetlin pattern is formed from $\mathfrak{s}$, the two entries are located as follows: the bottom entry is on the left bottom of the top entry. This is a one-one correspondence between intersections in $\mathfrak{s}$ and all pairs of equal entries in the form described above.

    We prove by induction. Label the colors on the top boundary by $c_1,c_2,...,c_r$. By the correspondence described above, $p_r$ is not intersecting any other paths if and only if the top right element in $\text{GTP}^\circ(\mathfrak{s})$ does not equal the element to the left bottom of it. 

    Meanwhile, $p_r$ is not intersecting any other paths if and only if the color $c_r$ goes out on the first row, which is equivalent to $w(r)=1$. 

    On the other hand, by the definition of $\tau$, the statement that the top right element in $\text{GTP}^\circ(\mathfrak{s})$ does not equal the element to the left bottom of it, is equivalent to $\tau(1)=r$. Hence we have proved 
    \[\tau(1)=r \Leftrightarrow w(r)=1.\]

    Now suppose in general that $\tau(1)=k$ where $1\leq k \leq r-1$. Then there are intersections recorded by the $k+1,k+2,...,r$-th entry on the first row of $\text{GTP}^\circ(\mathfrak{s})$. In the first row of an open state, every intersection is a crossing (since the left color is larger than the right color), therefore it is the color $c_k$ that goes out on the right boundary of row 1, which means, by the definition of $w$, $w(k)=1$. Thus $\tau(1)=k$ implies $w(k)=1$. The argument can be reversed, so it is an equivalence.

    To move to the next row, it suffices to forget the color $c_k$. Then the case reduces to the colors 
    \[c_1,c_2,...,c_{k-1},c_{k+1},...,c_r\] labeled from left to right above row 2, and the argument will be the same as the one for row 1. We could repeat this process to finish the proof.
\end{proof}

\begin{prop}\label{prop:rho}
 Let $\mathfrak{s} \in \mathfrak{S}_\lambda^\bullet(s_iy)$ with $s_iy>y$. Let $\tau$ be the permutation associated to $\text{GTP}^\circ({\mathfrak{s}})$ in \cref{def:y}.  If $e_i(\theta({\mathfrak{s}}))=0$ then $\tau(i)<\tau(i+1)$.
 \end{prop}
 
 \begin{proof}
     We use proof by contradiction. Suppose that $\tau(i)=k$ and $\tau(i+1)=\ell$ and we have $k\geq \ell$. By definition, $\tau$ is a bijection, hence we in fact have $k>\ell$. The values $k,\ell$ only depend on the $i,(i+1),(i+2)$-th rows of the Gelfand-Tsetlin pattern $A$. Write the entries of $A$ as the following:
       \[x_1,x_2,...,x_{r-i+1},\]
        \[y_1,y_2,...,y_{r-i},\]
        \[z_1,z_2,...,z_{r-i-1}.\]
     Since $\tau(i)=k$, by definition of $\tau$ we have
     
     \[x_k<y_{k-1}\]
     \[x_{k+1}=y_{k},\]
     \[x_{k+2}=y_{k+1},\]
     \[...\]
     \[x_{r-i+1}=y_{r-i}\]
     Similarly, since $\tau(i+1)=\ell$ and $k> \ell$ we have
     \[y_{k}=z_{k-1}, \]
     \[y_{k+1}=z_{k},\]
     \[y_{k+2}=z_{k+1},\]
     \[...\]
     \[y_{r-i}=z_{r-i-1}.\]
     Now, all these will contradict the condition $e_i(\theta(\mathfrak{s}))=0$. We have 
     \[x_{r-i+1}=y_{r-i},\]
     \[(y_{r-i-1}-z_{r-i-1})=(x_{r-i}-y_{r-i}),\]
     \[...\]
     \[(y_{k}-z_{k})=(x_{k+1}-y_{k+1}),\]
     \[(y_{k-1}-z_{k-1})>(x_k-y_k).\]
     The last inequality is a contradiction to \cref{crystaloperator}.
 \end{proof}

We can now prove case 2 of \cref{mainthm2}:

\begin{prop}\label{prop:case2}
    Suppose that $s_iy>y$. If $e_i(\theta(\mathfrak{s})) = 0$ then there exists a recoloring of the state $\mathfrak{s}$ into $\mathfrak{S}_\lambda^\bullet(y)$, i.e., there exists $\mathfrak{t}\in \mathfrak{S}_{\lambda}^\bullet(y)$ such that $\theta(\mathfrak{s})=\theta(\mathfrak{t}).$  
\end{prop}

\begin{proof}
    Let $\tau$ be the permutation associated to $\text{GTP}^\circ(\mathfrak{s})$. By \cref{prop:rho}, $e_i(\theta({\mathfrak{s}}))=0$ implies that $\tau(i)<\tau(i+1)$, which is equivalent to  $\tau s_i>\tau$ in the Bruhat order of $S_r$. Furthermore, it is equivalent to  \[s_i\tau^{-1}>\tau^{-1}.\] Meanwhile, by the uniqueness of open states, $\text{GTP}^\circ(\mathfrak{s})$ corresponds to an open state in $\mathfrak{S}^\circ_\lambda(w)$. Since $\tau=w^{-1}$ (\cref{prop:compare}), we actually have derived that
    \[s_iw>w.\]

    On the other hand, by \cref{thm:state}, the existence of $\mathfrak{s}\in \mathfrak{S}_\lambda^\bullet(s_iy)$ implies that 
    \[s_iy\geq w.\] Now we have $s_iy>y$, $s_iw>w$, and $s_iy\geq w$. By the lifting property of the Bruhat order (see Proposition 2.2.7 of \cite{BB}), we have $y\geq w$. Hence by \cref{thm:state} again, there exists $\mathfrak{t}\in\mathfrak{S}^\bullet_\lambda(y)$ such that $\theta(\mathfrak{t})=\theta(\mathfrak{s}).$
\end{proof}

Finally, the main theorem \cref{mainthm2} reproves the main theorem of \cite{BBBG}.

\section{{Applications to colored bosonic models}{}}

The purpose of this section is to apply the results above to colored bosonic models. Uncolored bosonic models were studied by Korff \cite{korff2013cylindric}. Bump and Naprienko study a colored version in \cite{bump2022colored}. In these bosonic models, a horizontal edge still carries either one color or the uncolored spin $+$, but a vertical edge may carry several copies of different colors. There is a way to combine different colors together, called fusion, and assign Boltzmann weights to vertices. At first this appears to give a different state space. However, we will show that, at $t=0$, after a vertical shift, fusion from small colors to large colors relates to the open colored five-vertex model, while fusion from large colors to small colors relates to the closed colored five-vertex model. 

We first recall the construction of colored bosonic models. Note that \cite{bump2022colored} uses fusion from small colors to large colors. It turns out to be interesting to consider other orders of fusion as well.  The order of fusion is essential: we shall see that one gives an analogue of the open model, while the other gives an analogue of the closed model.

\vspace{2mm}

Let $\lambda=(\lambda_1,\lambda_2,\ldots,\lambda_r)$ be a partition of length $r$, and choose $N\geq \lambda_1$.  We will define two models, to be denoted by $\mathfrak{S}_{\lambda,w}^{b,\bullet}$ and $\mathfrak{S}_{\lambda,w}^{b,\circ}$.  Both models will be based on an $r\times (N+1)$ grid.  Rows are labelled $1,2,\ldots,r$ from top to bottom, and columns are labelled $0,1,\ldots,N$ from right to left.  The horizontal spinset is
\[
\Sigma_h=\{+,c_1,\ldots,c_r\},
\]
where $+$ is the uncolored, or vacuum, spin, and the colors are ordered
\[
  c_1>c_2>\cdots>c_r.
\]
The vertical spinset $\Sigma_v$ is the set of all finite multisets in the colors.  Thus a vertical spin may be written as
\[
  c_1^{m_1}c_2^{m_2}\cdots c_r^{m_r},\qquad m_i\in\mathbb{N}_{\geq 0}.
\]
A vertical edge carries color $c_i$ if $m_i>0$; if all $m_i=0$, the vertical edge is uncolored, and we again denote this spin by $+$ in pictures.

Let $\cvec=(\gamma_1,\ldots,\gamma_r)$, where each $\gamma_i\in \{c_1,...,c_r\},$ and for $w\in S_r$ write
\[
  w\cvec=(\gamma_{w^{-1}(1)},\ldots,\gamma_{w^{-1}(r)}).
\]
The boundary conditions are as follows: On the top boundary, put the color $\gamma_i$ in column $\lambda_i$; if several parts of $\lambda$ are equal, the corresponding top edge carries the multiset of those colors.  On the right boundary, put the colors in the order $w\cvec$ from top to bottom, so the color $c_i$ exits through row $w(i)$.  All remaining boundary edges are assigned the vacuum spin $+$.

To define the Boltzmann weights, one first uses monochrome vertices.  A {\em monochrome vertex} is labelled by a row parameter $z_i$ and a color $c_j$.  Its vertical edges may carry only bosons of color $c_j$, so their spins are identified with nonnegative integers $n$.  Its horizontal edges may carry $+$ or any color. The admissible monochrome vertices and their weights are shown in Table 4:

\begin{table}[H]
\begin{center}
\begin{tabular}{ |c|c|c|c|c| } 
\hline
 & $A(n)$ & $B(n)$ & $C(n)$ & $D(n)$  \\ 
 \hline &
 \begin{tikzpicture}[scale=.7,every node/.style={scale=.5}]
\foreach \i in {1}
	\draw (0,\i) -- (2,\i);
\foreach \j in {1}
	\draw (\j,0) to (\j,2);
    \foreach \i/\j in {1/0,0/1,1/2,2/1}
  \draw[fill=white] (\j,\i) circle (.35);
  \foreach \i/\j in {0/1,2/1}
\node at (\i,1) {{\bf +}};
  \foreach \i/\j in {1/2,1/0}
\node at (1,\j) {$n$};
 \end{tikzpicture}  & \begin{tikzpicture}[scale=.7,every node/.style={scale=.5}]
\foreach \i in {1}
	\draw (0,\i) -- (2,\i);
\foreach \j in {1}
	\draw (\j,0) to (\j,2);
    \foreach \i/\j in {1/0,0/1,1/2,2/1}
  \draw[fill=white] (\j,\i) circle (.35);
  \foreach \i/\j in {0/1,2/1}
\node at (\i,1) {$d$};
  \foreach \i/\j in {1/2,1/0}
\node at (1,\j) {$n$};
 \end{tikzpicture}   &  \begin{tikzpicture}[scale=.7,every node/.style={scale=.5}]
\foreach \i in {1}
	\draw (0,\i) -- (2,\i);
\foreach \j in {1}
	\draw (\j,0) to (\j,2);
    \foreach \i/\j in {1/0,0/1,1/2,2/1}
  \draw[fill=white] (\j,\i) circle (.35);
  \foreach \i/\j in {0/1}
\node at (\i,1) {$c$};
  \foreach \i/\j in {2/1}
\node at (\i,1) {$+$};
  \foreach \i/\j in {1/0}
\node at (1,\j) {$n\!+\!1$};
  \foreach \i/\j in {1/2}
\node at (1,\j) {$n$};
 \end{tikzpicture} & \begin{tikzpicture}[scale=.7,every node/.style={scale=.5}]
\foreach \i in {1}
	\draw (0,\i) -- (2,\i);
\foreach \j in {1}
	\draw (\j,0) to (\j,2);
    \foreach \i/\j in {1/0,0/1,1/2,2/1}
  \draw[fill=white] (\j,\i) circle (.35);
  \foreach \i/\j in {0/1}
\node at (\i,1) {$+$};
  \foreach \i/\j in {2/1}
\node at (\i,1) {$c$};
  \foreach \i/\j in {1/0}
\node at (1,\j) {$n$};
  \foreach \i/\j in {1/2}
\node at (1,\j) {$n\!+\!1$};
 \end{tikzpicture} \\ 
 \hline
 $P$-weights & 1 & $\begin{array}{c}
1 \quad c_j<d \\
z_i \quad c_j=d \\
t^n \quad c_j>d
\end{array}$& $z_i(1-t^{n+1})$ & $1$  \\ 
 \hline 
 $R$-weights & 1 & $\begin{array}{c}
1 \quad c_j<d \\
z_i \quad c_j=d \\
t^n \quad c_j>d
\end{array}$& $z_i(1-t)$ & $\frac{1-t^{n+1}}{1-t}$\\
 \hline
\end{tabular}
\vspace{2mm}
\caption{Monochrome vertices with parameters $z_i,c_j$.  The vertical edge can carry only the color $c_j$.  The vertical labels $n$ count copies of the boson of color $c_j$.}
\label{tab:monochrome}
\end{center}
\end{table}

Note that the $P$-weights and $R$-weights agree when we specialize $t=0$. 

A colored vertex and its Boltzmann weights are obtained by fusion. To demonstrate, we explain the fusion process for the closed bosonic model $\mathfrak{S}_{\lambda,w}^{b,\bullet}.$ To obtain the open bosonic model, one just needs to reverse the order of the fusion. 

A single fused vertex with parameter $z_i$ is formed by concatenating monochrome vertices
\[
  (z_i,c_1), (z_i,c_2),\ldots,(z_i,c_r)
\]
from left to right, in decreasing order with respect to the order of colors $c_1>c_2>\cdots>c_r$.

\begin{rem}
    In \cite{bump2022colored}, the order of colors is $c_r>c_{r-1}>\cdots>c_1.$
\end{rem}

\begin{table}[H]
\begin{center}
\begin{tabular}{ |c|c| } 
\hline
\begin{tikzpicture}[scale=.9,every node/.style={scale=.5}]
\foreach \i in {1}
	\draw (0,\i) -- (2,\i);
\foreach \j in {1}
	\draw (\j,0) to (\j,2);
    \foreach \i/\j in {1/0,0/1,1/2,2/1}
  \draw[fill=white] (\j,\i) circle (.35);
  \foreach \i/\j in {0/1}
\node at (\i,1) {$a$};
  \foreach \i/\j in {2/1}
\node at (\i,1) {$c$};
  \foreach \i/\j in {1/0}
\node at (1,\j) {$d$};
  \foreach \i/\j in {1/2}
\node at (1,\j) {$b$};
\node[fill=white,inner sep=1pt] at (1,1) {$z_i$};
 \end{tikzpicture} & 
 \begin{tikzpicture}[scale=.9,every node/.style={scale=.5}]
\draw (0,1) -- (2,1);
\draw (4,1) -- (8,1);
\draw (1,0) to (1,2);
\draw (5,0) to (5,2);
\draw (7,0) to (7,2);
\draw[fill=white] (0,1) circle (.35);
\node at (0,1) {$a$};
\draw[fill=white] (1,2) circle (.35);
\node at (1,2) {$b_1$};
\draw[fill=white] (1,0) circle (.35);
\node at (1,0) {$d_1$};
\draw[fill=white] (5,2) circle (.35);
\node at (5,2) {$b_{r-1}$};
\draw[fill=white] (5,0) circle (.35);
\node at (5,0) {$d_{r-1}$};
\draw[fill=white] (7,2) circle (.35);
\node at (7,2) {$b_r$};
\draw[fill=white] (7,0) circle (.35);
\node at (7,0) {$d_r$};
\draw[fill=white] (8,1) circle (.35);
\node at (8,1) {$c$};
\node at (3,2) {$\cdots$};
\node at (3,1) {$\cdots$};
\node at (3,0) {$\cdots$};
\node[fill=white,inner sep=1pt] at (1,1) {$z_i,c_1$};
\node[fill=white,inner sep=1pt] at (5,1) {$z_i,c_{r-1}$};
\node[fill=white,inner sep=1pt] at (7,1) {$z_i,c_r$};
\end{tikzpicture} \\
 \hline
\end{tabular}
\vspace{2mm}
\caption{Fusion procedure in decreasing order.  The single fused vertex $z_i$ is formed by concatenating monochrome vertices $(z_i,c_1),\ldots,(z_i,c_r)$ from left to right.  The vertical spins $b,d$ are the sums of the monochrome vertical spins $b_k,d_k$.}
\label{tab:fusion}
\end{center}
\end{table}

For fixed exterior spins $a,b,c,d$, the Boltzmann weight of the fused vertex is the product of the monochrome weights in Table~\ref{tab:monochrome}, summed over compatible internal horizontal spins.  In fact, there is at most one compatible assignment of the internal horizontal spins. Thus we have defined the model $\mathfrak{S}_{\lambda,w}^{b,\bullet}$.  See \cite{bump2022colored} for more details on the open bosonic model.

\begin{table}[H]
\begin{center}
\scriptsize
\[
  \begin{array}{|c|c|c|c|}
    \hline
    \begin{tikzpicture}\gamgamunco{+}{+}{+}{+}{}\spokepa\spokepb\spokepc\spokepd\rcent\end{tikzpicture} &
    \begin{tikzpicture}\gamgamunco{c}{c}{c}{c}{}\spokea{red}{c}\spokeb{red}{c}\spokec{red}{c}\spoked{red}{c}\rcent\end{tikzpicture} &
    \begin{tikzpicture}\gamgamunco{c}{d}{d}{c}{}\spokea{blue}{c}\spokeb{red}{d}\spokec{red}{d}\spoked{blue}{c}\rcent\end{tikzpicture} &
    \begin{tikzpicture}\gamgamunco{c}{d}{c}{d}{}\spokea{blue}{c}\spokeb{red}{d}\spokec{blue}{c}\spoked{red}{d}\rcent\end{tikzpicture} \\
    \hline
    z_i-tz_j & z_i-tz_j &
    \begin{array}{ll}(1-t)z_j&\text{if $c<d$}\\(1-t)z_i&\text{if $c>d$}\end{array} &
    \begin{array}{ll}z_i-z_j&\text{if $c>d$}\\t(z_i-z_j)&\text{if $c<d$}\end{array} \\
    \hline
    \begin{tikzpicture}\gamgamunco{c}{+}{+}{c}{}\spokea{red}{c}\spokepb\spokepc\spoked{red}{c}\rcent\end{tikzpicture} &
    \begin{tikzpicture}\gamgamunco{+}{c}{c}{+}{}\spokepa\spokeb{red}{c}\spokec{red}{c}\spokepd\rcent\end{tikzpicture} &
    \begin{tikzpicture}\gamgamunco{+}{c}{+}{c}{}\spokepa\spokeb{red}{c}\spokepc\spoked{red}{c}\rcent\end{tikzpicture} &
    \begin{tikzpicture}\gamgamunco{c}{+}{c}{+}{}\spokea{red}{c}\spokepb\spokec{red}{c}\spokepd\rcent\end{tikzpicture} \\
    \hline
    (1-t)z_i & (1-t)z_j & t(z_i-z_j) & z_i-z_j\\
    \hline
  \end{array}
\]
\vspace{2mm}
\caption{The fused $R$-matrix for the decreasing fusion order. Note that only the weight of the third configuration in the first row is different (opposite, in fact) from the corresponding weight in the $R$-matrix for the model in \cite{bump2022colored}. }
\label{tab:fused-r}
\end{center}
\end{table}

The colored $R$-matrix compatible with this decreasing fusion order is given in \cref{tab:fused-r}, while the auxiliary $R$-vertices used in the local Yang--Baxter equations are in \cref{tab:aux-r}.

\begin{prop}
With the fused weights defined by the decreasing fusion order, and with the fused $R$-matrix in Table~\ref{tab:fused-r}, the colored weights satisfy the Yang--Baxter equation.  The same fused $R$-matrix works for both the $P$- and $R$-weights.
\end{prop}

\begin{proof}
Replace each fused vertex by its chain of monochrome vertices in decreasing order, as in Table~\ref{tab:fusion}.  The fused Yang--Baxter equation is then proved by pushing auxiliary $R$-vertices through these monochrome vertices. At each monochrome vertex, we will check a local auxiliary Yang-Baxter equation as depicted in Figure~\ref{fig:auxiliary-ybe}. The two vertical boundary labels $c,f$ are nonnegative multiplicities of the monochrome color $c_k$; the four remaining boundary labels $a,b,d,e$ are horizontal spins. All these external spins $a,b,c,d,e,f$ are fixed, and we check the two partition functions obtained after summing over the states are equal.

Once the unfused Yang-Baxter equations are verified, applying them successively for colors $c_1,c_2,...,c_r$
moves the auxiliary vertex all the way through the unfused vertices. After the last monochrome factor the auxiliary vertex has returned to the same color $c_1$.  Multiplying the $r$ unfused auxiliary Yang--Baxter equations up gives the desired fused Yang--Baxter identity. 

Now we prove the unfused auxiliary Yang-Baxter equations. Since each monochrome vertex can carry at most two colors in adjacent spins,  Figure~\ref{fig:auxiliary-ybe} involves at most three colors.  Without loss of generality, we may assume the three colors are $c_i,c_j$ and $c_k$ with distinct $i,j,k$ and an unspecified orders among them. Then the boundary conditions $a,b,c,d,e,f$ could take values in $+$ or a nonnegative integer in one of the three colors.

Since we still have infinitely many possible combinations of boundary conditions, we would like to reduce to a finite number of cases that could be checked by a computer. First, along the vertical monochrome line there are exactly two monochrome vertices, and each of them changes the vertical multiplicity by $0$, $+1$, or $-1$.  Hence every nonzero state has
\[
  |c-f|\leq 2.
\]
If this inequality fails, both sides of the auxiliary Yang--Baxter equation have no admissible states and both partition functions are zero.

Assume now that $|c-f|\leq2$.  Write
\[
  c=n+c_0,\qquad f=n+f_0,
\]
where $n\geq0$, one of $c_0,f_0$ is zero, and the other is at most $2$.  The reduced pair $(c_0,f_0)$ is exactly the finite boundary data that controls possible failures of nonnegativity for the internal vertical spins.  Once this nonnegativity issue is fixed, increasing $n$ shifts every vertical occupation number by the same amount.  The only remaining $n$-dependence in the local weights is through $t^n$, $1-t^{n+1}$, and, for the $R$-weights, $(1-t^{n+1})/(1-t)$.  Since there are only two monochrome vertices, after clearing the denominator $(1-t)$ in the $R$-case, the difference between the two sides is a polynomial of degree at most two in $T=t^n$.  Thus it suffices to check $n=0,1,2$ for each reduced pair.  This is a finite number of cases to verify, and is done by a SageMath computer program, and hence the fused Yang-Baxter equation follows.
\end{proof}

\begin{table}[H]
\begin{center}
\scriptsize
\[
  \begin{array}{|c|c|c|c|}
    \hline
    \begin{tikzpicture}\gamgamunco{+}{+}{+}{+}{c}\spokepa\spokepb\spokepc\spokepd\cv{c}\end{tikzpicture} &
    \begin{tikzpicture}\gamgamunco{d}{d}{d}{d}{c}\spokea{blue}{d}\spokeb{blue}{d}\spokec{blue}{d}\spoked{blue}{d}\cv{c}\end{tikzpicture} &
    \begin{tikzpicture}\gamgamunco{d}{e}{d}{e}{c}\cv{c}
      \spokea{blue}{d}\spokeb{darkgreen}{e}\spokec{blue}{d}\spoked{darkgreen}{e}\cv{c}\end{tikzpicture} &
    \begin{tikzpicture}\gamgamunco{d}{e}{e}{d}{c}\cv{c}
      \spokea{blue}{d}\spokeb{darkgreen}{e}\spokec{darkgreen}{e}\spoked{blue}{d}\cv{c}\end{tikzpicture} \\
    \hline
    z_i-tz_j & \begin{array}{c}z_i-tz_j\\\text{$c=d$ allowed}
      \end{array} &
      \begin{array}{c}\begin{array}{ll} t(z_i-z_j)&\text{if $e>d$}\\z_i-z_j&\text{if $d>e$}\end{array}\\\text{$c=d$ or $e$ allowed}\end{array}&
      \begin{array}{ll}
        (1-t)z_j&\text{if $d>c>e$}\\
                &\text{or $e>d>c$}\\
                &\text{or $c>e>d$}\\
        (1-t)z_i&\text{if $e>c>d$}\\
                &\text{or $d>e>c$}\\
                &\text{or $c>d>e$}
       \end{array}\\
    \hline
    \begin{tikzpicture}\gamgamunco{d}{c}{c}{d}{c}\cv{c}\spokea{blue}{d}\spokeb{red}{c}\spokec{red}{c}\spoked{blue}{d}\cv{c}\end{tikzpicture} &
    \begin{tikzpicture}\gamgamunco{c}{d}{d}{c}{c}\cv{c}\spokea{red}{c}\spokeb{blue}{d}\spokec{blue}{d}\spoked{red}{c}\cv{c}\end{tikzpicture} &
    \begin{tikzpicture}\gamgamunco{+}{d}{+}{d}{c}\spokepa\spokeb{blue}{d}\spokepc\spoked{blue}{d}\cv{c}\end{tikzpicture}&
    \begin{tikzpicture}\gamgamunco{d}{+}{d}{+}{c}\spokea{blue}{d}\spokepb\spokec{blue}{d}\spokepd\cv{c}\end{tikzpicture}\\
    \hline
    (1-t)z_j&(1-t)z_i&\begin{array}{c}t(z_i-z_j)\\\text{$c=d$ allowed}\end{array}&\begin{array}{c}z_i-z_j\\\text{$c=d$ allowed}\end{array}\\
    \hline
    &\begin{tikzpicture}\gamgamunco{d}{+}{+}{d}{c}\spokea{blue}{d}\spokepb\spokepc\spoked{blue}{d}\cv{c}\end{tikzpicture}&
    \begin{tikzpicture}\gamgamunco{+}{d}{d}{+}{c}\spokepa\spokeb{blue}{d}\spokec{blue}{d}\spokepd\cv{c}\end{tikzpicture}&\\
    \hline
    &\begin{array}{c}(1-t)z_i\\\text{$c=d$ allowed}\end{array}&
    \begin{array}{c}(1-t)z_j\\\text{$c=d$ allowed}\end{array}&\\
    \hline
  \end{array}\]
\caption{Auxiliary $R$-vertices for the decreasing fusion order. Note that only the weight of the fourth configuration in the first row is different (opposite, in fact) from the corresponding weight in the $R$-matrix for the model in \cite{bump2022colored}.}
\label{tab:aux-r}
\end{center}

\end{table}

\begin{figure}[H]
\begin{equation}
\label{eq:auxiliary-ybe}
\begin{array}{ccc}
\auxybelhs{z_i,z_j,c_k}{z_i,c_k}{z_j,c_k}{.52}
&\qquad\qquad&
\auxyberhs{z_i,z_j,c_{k+1}}{z_i,c_k}{z_j,c_k}{.52}
\end{array}
\end{equation}
\caption{Auxiliary Yang--Baxter equation for the decreasing fusion order.  The two vertical boundary labels $c,f$ are nonnegative multiplicities of the monochrome color $c_k$; the four slanted boundary labels $a,b,d,e$ are horizontal spins.  We will prove that the two systems have the same partition function.}
\label{fig:auxiliary-ybe}
\end{figure}

\medskip
\noindent\textit{Example.}
Take the monochrome color to be denoted simply by $1$, and impose the following boundary data in Figure~\ref{fig:auxiliary-ybe}
\[
  a=+,\qquad b=+,\qquad c=2,\qquad d=+,\qquad e=1,\qquad f=1.
\]
Here the vertical labels $c=2$ and $f=1$ are occupation numbers of the same monochrome color $1$, while the horizontal label $e=1$ means that the lower right horizontal edge carries that color.  The only auxiliary $R$-entries that will be used, besides the all-vacuum entry, are the entries with the auxiliary color equal to this horizontal color: the pattern $(1,+,+,1)$ has weight $(1-t)z_i$, and the pattern $(+,1,+,1)$ has weight $t(z_i-z_j)$.  Thus this example does not use any mixed-color inequality in Table~\ref{tab:aux-r}.  With these boundary spins, the left-hand side of Figure~\ref{fig:auxiliary-ybe} has one admissible state:
\[
\begin{array}{c}
\ExampleLHSState
\end{array}
\]
At the auxiliary vertex both incoming horizontal spins are $+$, so the auxiliary weight is $z_i-tz_j$.  The upper monochrome vertex is $A(2)$, of weight $1$, and the lower monochrome vertex is $D(1)$.  Thus the left-hand partition function is
\[
  Z_{\mathrm{left}}^{P}=z_i-tz_j,
  \qquad
  Z_{\mathrm{left}}^{R}=(z_i-tz_j)\frac{1-t^2}{1-t}=(1+t)(z_i-tz_j).
\]

On the right-hand side of Figure~\ref{fig:auxiliary-ybe}, there are exactly two admissible states.  In the first one, the upper monochrome vertex is $A(2)$ and the lower monochrome vertex is $D(1)$:
\[
\ExampleRHSStateOne
\]
The auxiliary vertex has boundary pattern $(1,+,+,1)$, hence its weight is $(1-t)z_i$.  This state contributes
\[
  (1-t)z_i
  \quad\text{for the $P$-weights,}
  \qquad
  (1+t)(1-t)z_i
  \quad\text{for the $R$-weights.}
\]
In the second one, the upper monochrome vertex is $D(1)$ and the lower monochrome vertex is $A(1)$:
\[
\ExampleRHSStateTwo
\]
The auxiliary vertex has boundary pattern $(+,1,+,1)$, hence its weight is $t(z_i-z_j)$.  This state contributes
\[
  t(z_i-z_j)
  \quad\text{for the $P$-weights,}
  \qquad
  (1+t)t(z_i-z_j)
  \quad\text{for the $R$-weights.}
\]
Therefore
\[
  Z_{\mathrm{right}}^{P}=(1-t)z_i+t(z_i-z_j)=z_i-tz_j=Z_{\mathrm{left}}^{P},
\]
and
\[
  Z_{\mathrm{right}}^{R}=(1+t)\bigl((1-t)z_i+t(z_i-z_j)\bigr)
  =(1+t)(z_i-tz_j)=Z_{\mathrm{left}}^{R}.
\]
This is one concrete local auxiliary Yang--Baxter check.  The same computation is unchanged in substance if the boundary multiplicities are replaced by
\[
  c=n+1,\qquad f=n.
\]
The admissible states are the same, with $A(2)$ replaced by $A(n+1)$, $D(1)$ replaced by $D(n)$, and $A(1)$ replaced by $A(n)$.  For the $P$-weights the partition functions are still
\[
  Z_{\mathrm{left}}^{P}=z_i-tz_j,
  \qquad
  Z_{\mathrm{right}}^{P}=(1-t)z_i+t(z_i-z_j)=z_i-tz_j.
\]
For the $R$-weights, every term is multiplied by the same factor
\[
  \frac{1-t^{n+1}}{1-t},
\]
coming from the unique $D(n)$-vertex.  Thus the identity for $c=n+1$, $f=n$ is exactly the same identity, with a common $n$-dependent factor in the $R$-case.  This illustrates the general reduction used above: the difference $c-f$ is the essential finite datum, while the common shift contributes only a controlled polynomial dependence in $t^n$.

\medskip 

Now we turn to calculate some partition functions. From now on, we will assume that \[(\gamma_1,...,\gamma_r)=(c_1,...,c_r).\] Under this assumption, we do not need to distinguish $P$-weights and $R$-weights since the Boltzmann weights of types $C(n)$ and $D(n)$ can only appear with $n=0$, where they are the same in both models. We first calculate the partition function for the model $\mathfrak{S}_{\lambda,1}^{b,\bullet}.$ When $w=1$, the boundary flag on the right is the standard flag $(c_1,\ldots,c_r)$.

\begin{prop}
For $w=1$, the closed colored model has a unique state, and
\[
  Z(\mathfrak S^{b,\bullet}_{\lambda,1})={\bf z}^\lambda=z_1^{\lambda_1}\cdots z_r^{\lambda_r}.
\]
\end{prop}

\begin{proof}
The color $c_1$ is the largest color and must exit through the first row.  Since the top boundary places $c_1$ in column $\lambda_1$, the path is determined.  After this path is fixed, the same argument applies to $c_2$ in the remaining rows, and then successively to all colors.  Thus there is only one state.  In that state, color $c_i$ travels horizontally through exactly $\lambda_i$ vertices in row $i$ contributing a factor $z_i$, while all other vertices have weight $1$ at $t=0$.  The total weight is therefore $z_1^{\lambda_1}\cdots z_r^{\lambda_r}$.
\end{proof}

Let $s_i=(i,i+1)$ act on functions of $z_1,\ldots,z_r$ by interchanging $z_i$ and $z_{i+1}$.  Recall that the  Demazure-Lusztig operators $\mathcal{L}_{i,t}$ are defined by 
\[
\mathcal{L}_{i,t}f(z)
=
\frac{f(z)-f(s_i z)}{z^{\alpha_i}-1}
-
t\frac{f(z)-z^{\alpha_i}f(s_i z)}{z^{\alpha_i}-1}
\]
where $\alpha_i=z_{i}/z_{i+1}.$

\begin{prop}
Let $w\in W$ and $s_i$ be the simple reflection in $W=S_r$. If $s_iw>w$ then 
\[
Z\bigl(\mathfrak{S}^{b,\bullet}_{\lambda,s_iw}\bigr)
=
t\mathcal{L}_{i,t}^{-1}
Z\bigl(\mathfrak{S}^{b,\bullet}_{\lambda,w}\bigr).
\]
Setting $t=0$, we have 
\[
  Z(\mathfrak S^{b,\bullet}_{\lambda,s_iw})=\partial_i Z(\mathfrak S^{b,\bullet}_{\lambda,w}).
\]
\end{prop}

\begin{proof}
Write
\[
f(\bZ)=Z\bigl(\mathfrak S^{b,\bullet}_{\lambda,w}\bigr),
\qquad
g(\bZ)=Z\bigl(\mathfrak S^{b,\bullet}_{\lambda,s_iw}\bigr).
\]
Since $s_iw>w$, the right boundary colors in rows $i$ and $i+1$ satisfy
$c_{w^{-1}(i)}>c_{w^{-1}(i+1)}$.
Attach the $R$-matrix between these two rows and run a standard train argument.  On the left there is a unique completion, of weight $z_i-tz_{i+1}$; after moving the $R$-matrix to the right there are two completions, of weights $(1-t)z_{i+1}$ and $z_i-z_{i+1}$.  Hence
\[
(z_i-tz_{i+1})f(\bZ)
=(1-t)z_{i+1}f(s_i\bZ)
 +(z_i-z_{i+1})g(s_i\bZ).
\]
Replacing $\bZ$ by $s_i\bZ$ gives
\[
(z_{i+1}-tz_i)f(s_i\bZ)
=(1-t)z_i f(\bZ)+(z_{i+1}-z_i)g(\bZ),
\]
and therefore
\begin{align*}
g(\bZ)
&=\frac{(1-t)z_i f(\bZ)-(z_{i+1}-tz_i)f(s_i\bZ)}{z_i-z_{i+1}}\\
&=(1-t)\partial_i f(\bZ)+t f(s_i\bZ).
\end{align*}
Solving for $g$ we obtain $g=t\mathcal L_{i,t}^{-1}f$ for $t\neq0$.  Finally, setting $t=0$ in the regular formula above gives $g=\partial_i f$.
\end{proof}

\begin{thm}

If $t=0$,  then for every $w\in S_r$,
\[
  Z(\mathfrak S^{b,\bullet}_{\lambda,w})=\partial_w {\bf z^\lambda}.
\]
Here, if $w=s_{i_1}\cdots s_{i_\ell}$ is a reduced expression, then $\partial_w=\partial_{i_1}\cdots\partial_{i_\ell}$.
\end{thm}

\begin{proof}
The result for $w=1$ is the unique-state calculation above. In general, write $w=s_{i_1}\cdots s_{i_\ell}$ as a reduced word and set $w_k=s_{i_k}s_{i_{k+1}}...s_{i_\ell}$ for $1\leq k \leq \ell$. Then $w_{k-1}=s_{i_k-1}w_{k}$ and $s_{i_{k-1}}w_{k}>w_{k}$.  Applying the Demazure recursion repeatedly gives
\[
  Z(\mathfrak S^{b,\bullet}_{\lambda,w_{k-1}})=\partial_{i_k}Z(\mathfrak S^{b,\bullet}_{\lambda,w_{k}}),
\]
so
\[
  Z(\mathfrak S^{b,\bullet}_{\lambda,w})=
  \partial_{i_1}\partial_{i_2}\cdots\partial_{i_\ell}{\bf z}^\lambda.
\]
The usual braid and idempotence relations for the isobaric Demazure operators imply that this expression is independent of the chosen reduced word.
\end{proof}

Now we combine the main theorems with bosonic models to study sets of states of both the closed and the open bosonic models. Let $\mathfrak{s}$ be a state of $\mathfrak{S}^{b,\bullet}_{\lambda,w}$. As in the colored closed five-vertex models,  for $1\leq i\leq r$, define
\[
  A_{i,j}(\mathfrak{s}),\qquad 1\leq j\leq r-i+1,
\]
to be the column label of the $j$-th occupied vertical boson, counted with multiplicity, above the $i$-th row of vertices and ordered from left to right.  Thus
\[
  A_{1,j}(\mathfrak{s})=\lambda_j,
  \qquad 1\leq j\leq r.
\]
We call
\[
  \operatorname{GTP}(\mathfrak{s})=(A_{i,j}(\mathfrak{s}))
\]
the underlying Gelfand--Tsetlin pattern of the bosonic state $\mathfrak{s}$.

\begin{prop}
For every admissible state $\mathfrak{s}$ of the closed bosonic model, the array $\operatorname{GTP}(\mathfrak{s})=(A_{i,j}(\mathfrak{s}))$ is a Gelfand--Tsetlin pattern with top row $\lambda$. 
\end{prop}
\begin{proof}
The proof is similar to the closed colored five-vertex model case. The difference is that we do not obtain a left-strict Gelfand-Tsetlin pattern any more.    
\end{proof}

\begin{prop}
Assume $t=0$. Given a Gelfand-Tsetlin pattern $A$ with top row $\lambda$, there is a unique flag $w\in S_r$ and a state $\mathfrak{s}$ of $\mathfrak{S}_{\lambda,w}^{b,\circ}$  such that $\text{GTP}(\mathfrak{s})=A$. Hence there is a bijection between the set of all Gelfand-Tsetlin patterns with top row $\lambda$ and the set of all states of $\mathfrak{S}_{\lambda,w}^{b,\circ}$ for all $w\in S_r$.
\end{prop}

\begin{proof}
    It is essential that $t=0$ in the proof. In the open bosonic model, if we want the local Boltzmann weight to be nonzero, then the spins of the left edge and the top edge determine a unique assignment of the spins of the bottom and right edges. 

    Combining this observation with Proposition 6.6, we obtain the second part of the statement.
\end{proof}

Recall that we define $\rho=(r-1,r-2,\ldots,1,0)$ and we proved that there is a canonical bijection $S$ between the set of left-strict Gelfand-Tsetlin patterns with top row $\lambda+\rho$ and the set of Gelfand-Tsetlin patterns with top row $\lambda$. Let $\theta^{b,\bullet}:\mathfrak{S}_{\lambda}^{b,\bullet}(W) \rightarrow \mathcal{B}_\lambda$ be the composition of $\theta$ with this shift. Define $\theta_{\lambda,w}^{b,\circ}$ similarly for open bosonic models.  We obtain the following theorems for the states of bosonic models:

\begin{thm}
    Assume $t=0$. Then
    \begin{enumerate}
        \item The map $\theta^{b,\circ}$ is a weight-preserving embedding of the states of the bosonic model $\mathfrak{S}^{b,\circ}_{\lambda,w}$ (fusion from small colors to large colors) into $\mathcal{B}_\lambda$, and the image of $\theta^{b,\circ}$ is the Demazure atom $\mathcal{B}^\circ_{\lambda}(w)$.

        \item  The map $\theta^{b,\bullet}$ is a weight-preserving embedding of the states of the bosonic model $\mathfrak{S}^{b,\bullet}_{\lambda,w}$ (fusion from large colors to small colors) into $\mathcal{B}_\lambda$, and the image of $\theta^{b,\bullet}$ is the Demazure crystal $\mathcal{B}^\bullet_{\lambda}(w)$.
    \end{enumerate}

\end{thm}

\begin{proof}
    For all $w\in S_r$, the states of the bosonic model $\mathfrak{S}^{b,\circ}_{\lambda,w}$ are in bijective correspondence with the set of all Gelfand-Tsetlin patterns with top row $\lambda$, and further bijective to the set of all left-strict Gelfand-Tsetlin patterns with top row $\lambda+\rho$, and finally bijective to the states of $\mathfrak{S}_{\lambda,w}^\circ$ for $w\in S_{r}$. The composite bijection preserves boundary flags. Hence the bijections to Demazure atoms are also preserved. And the same argument works for closed models.
\end{proof}

\begin{rem}
    With this bijection, we can give the set of states of either $\mathfrak{S}^{b,\circ}_{\lambda,w}$ or $\mathfrak{S}^{b,\bullet}_{\lambda,w}$ a crystal structure with the assumption that $t=0$. The recoloring arguments in \cref{prop:case1} and \cref{prop:case2} work the same way for these models. As in the fermionic case (see \cref{strongrmk}), it follows from the Yang-Baxter equation that the image of $\theta$ has the same character as the Demazure crystal $\mathcal{B}_\lambda$ but this result is stronger since it identifies the image as the Demazure crystal.
\end{rem}

We illustrate how the crystal operators work for bosonic models as well in the following examples:

\begin{example}
   {\em Consider $\lambda=(2,1,0)$, $y=s_2\in S_3$, and the bosonic state given below. The operator $e_1$ will move colored edges on and above row $2$ one unit to the right. In this case we need to ``recolor" to obtain a legal bosonic state after moving the edges.

   }

\[e_1\left(
\vcenter{\hbox{
\begin{tikzpicture}[scale=.55,every node/.style={scale=.7}]

\foreach \y in {1,3,5}
    \draw[thick] (0,\y) -- (6,\y);
\foreach \x in {1,3,5}
    \draw[thick] (\x,0) -- (\x,6);

\foreach \y in {0,2,6}
    \foreach \x in {1,3,5}
        \draw[thick,fill=white] (\x,\y) circle (.25);
\foreach \x in {1,5}
    \draw[thick,fill=white] (\x,4) circle (.25);
\foreach \y in {1,3,5}
    \foreach \x in {0,2,4,6}
        \draw[thick,fill=white] (\x,\y) circle (.25);

\draw[line width=.5mm,red]
    (1,6) -- (1,5.5)
    to[out=-90,in=180] (1.5,5)
    -- (2.32,5)
    to[out=0,in=90] (2.82,4.5)
    -- (2.82,3.5)
    to[out=-90,in=180] (3.32,3)
    -- (6,3);

\draw[line width=.5mm,blue]
    (3,6) -- (3,1.5)
    to[out=-90,in=180] (3.5,1)
    -- (6,1);

\draw[line width=.5mm,green!80!black]
    (5,6) -- (5,5.5)
    to[out=-90,in=180] (5.5,5)
    -- (6,5);

\foreach \x/\y in {1/6,2/5,4/3,6/3}
    \draw[line width=.5mm,red,fill=white]
        (\x,\y) circle (.25);
\foreach \x/\y in {3/6,3/2,4/1,6/1}
    \draw[line width=.5mm,blue,fill=white]
        (\x,\y) circle (.25);
\foreach \x/\y in {5/6,6/5}
    \draw[line width=.5mm,green!80!black,fill=white]
        (\x,\y) circle (.25);

\foreach \y in {1,3,5}
    \foreach \x in {1,3,5}
        \node[circle,fill,inner sep=2.5pt] at (\x,\y) {};
\foreach \x in {1,3,5}
    \node at (\x,0) {\bf +};
\foreach \y in {1,3,5}
    \node at (0,\y) {\bf +};

\draw[very thick,->,>=Stealth] (3.4,4) -- (4.5,4);
\foreach \x/\c in {1/2,3/1,5/0}
    \node at (\x,6.5) {$\c$};
\foreach \y/\r in {1/3,3/2,5/1}
    \node at (-.5,\y) {$\r$};

\end{tikzpicture}}}
\right)=
\vcenter{\hbox{
\begin{tikzpicture}[scale=.55,every node/.style={scale=.7}]

\foreach \y in {1,3,5}
    \draw[thick] (0,\y) -- (6,\y);
\foreach \x in {1,3,5}
    \draw[thick] (\x,0) -- (\x,6);

\foreach \y in {0,2,4,6}
    \foreach \x in {1,3,5}
        \draw[thick,fill=white] (\x,\y) circle (.25);
\foreach \y in {1,3,5}
    \foreach \x in {0,2,4,6}
        \draw[thick,fill=white] (\x,\y) circle (.25);

\draw[line width=.5mm,red]
    (1,6) -- (1,5.5)
    to[out=-90,in=180] (1.5,5)
    -- (4.5,5)
    to[out=0,in=90] (5,4.5)
    -- (5,3.5)
    to[out=-90,in=180] (5.5,3)
    -- (6,3);

\draw[line width=.5mm,blue]
    (3,6) -- (3,1.5)
    to[out=-90,in=180] (3.5,1)
    -- (6,1);

\draw[line width=.5mm,green!80!black]
    (5,6) -- (5,5.5)
    to[out=-90,in=180] (5.5,5)
    -- (6,5);

\foreach \x/\y in {1/6,2/5,4/5,5/4,6/3}
    \draw[line width=.5mm,red,fill=white]
        (\x,\y) circle (.25);
\foreach \x/\y in {3/6,3/4,3/2,4/1,6/1}
    \draw[line width=.5mm,blue,fill=white]
        (\x,\y) circle (.25);
\foreach \x/\y in {5/6,6/5}
    \draw[line width=.5mm,green!80!black,fill=white]
        (\x,\y) circle (.25);

\foreach \y in {1,3,5}
    \foreach \x in {1,3,5}
        \node[circle,fill,inner sep=2.5pt] at (\x,\y) {};
\foreach \x in {1,3,5}
    \node at (\x,0) {\bf +};
\foreach \y in {1,3,5}
    \node at (0,\y) {\bf +};

\foreach \x/\c in {1/2,3/1,5/0}
    \node at (\x,6.5) {$\c$};
\foreach \y/\r in {1/3,3/2,5/1}
    \node at (-.5,\y) {$\r$};

\end{tikzpicture}}}
\]
\end{example}

\begin{example}
   {\em  Here is an example for Case 2.  Consider $\lambda=(2,1,0)$, $y=s_2\in S_3$, and the closed state $\mathfrak{s}_2$ given below. Then $s_1s_2>y$ and $e_1(\theta(\mathfrak{s}_2))=0$. In this case,  we could indeed ``recolor" the state for $\mathfrak{S}_{\lambda,s_2}^\bullet$ to obtain a closed state for $\mathfrak{S}_{\lambda,s_1s_2}^\bullet$.}

\[
\vcenter{\hbox{
\begin{tikzpicture}[scale=.55,every node/.style={scale=.7}]

\foreach \y in {1,3,5}
    \draw[thick] (0,\y) -- (6,\y);
\foreach \x in {1,3,5}
    \draw[thick] (\x,0) -- (\x,6);

\foreach \y in {0,2,4,6}
    \foreach \x in {1,3,5}
        \draw[thick,fill=white] (\x,\y) circle (.25);
\foreach \y in {1,3,5}
    \foreach \x in {0,2,4,6}
        \draw[thick,fill=white] (\x,\y) circle (.25);

\draw[line width=.5mm,red]
    (1,6) -- (1,5.5)
    to[out=-90,in=180] (1.5,5)
    -- (2.5,5)
    to[out=0,in=90] (3,4.5)
    -- (3,3.5)
    to[out=-90,in=180] (3.5,3)
    -- (6,3);

\draw[line width=.5mm,blue]
    (3,6) -- (3,5.5)
    to[out=-90,in=180] (3.5,5)
    -- (4.5,5)
    to[out=0,in=90] (5,4.5)
    -- (5,1.5)
    to[out=-90,in=180] (5.5,1)
    -- (6,1);

\draw[line width=.5mm,green!80!black]
    (5,6) -- (5,5.5)
    to[out=-90,in=180] (5.5,5)
    -- (6,5);

\foreach \x/\y in {1/6,2/5,3/4,4/3,6/3}
    \draw[line width=.5mm,red,fill=white]
        (\x,\y) circle (.25);

\foreach \x/\y in {3/6,4/5,5/4,5/2,6/1}
    \draw[line width=.5mm,blue,fill=white]
        (\x,\y) circle (.25);

\foreach \x/\y in {5/6,6/5}
    \draw[line width=.5mm,green!80!black,fill=white]
        (\x,\y) circle (.25);

\foreach \y in {1,3,5}
    \foreach \x in {1,3,5}
        \node[circle,fill,inner sep=2.5pt] at (\x,\y) {};

\foreach \x in {1,3,5}
    \node at (\x,0) {\normalfont\bfseries +};
\foreach \y in {1,3,5}
    \node at (0,\y) {\normalfont\bfseries +};

\foreach \x/\c in {1/2,3/1,5/0}
    \node at (\x,6.5) {$\c$};
\foreach \y/\r in {1/3,3/2,5/1}
    \node at (-.5,\y) {$\r$};

\end{tikzpicture}}}
\xlongrightarrow{\mathrm{recolor}}
\vcenter{\hbox{
\begin{tikzpicture}[scale=.55,every node/.style={scale=.7}]

\foreach \y in {1,3,5}
    \draw[thick] (0,\y) -- (6,\y);
\foreach \x in {1,3,5}
    \draw[thick] (\x,0) -- (\x,6);

\foreach \y in {0,2,4,6}
    \foreach \x in {1,3,5}
        \draw[thick,fill=white] (\x,\y) circle (.25);
\foreach \y in {1,3,5}
    \foreach \x in {0,2,4,6}
        \draw[thick,fill=white] (\x,\y) circle (.25);

\draw[line width=.5mm,red]
    (1,6) -- (1,5.5)
    to[out=-90,in=180] (1.5,5)
    -- (6,5);

\draw[line width=.5mm,blue]
    (3,6) -- (3,3.5)
    to[out=-90,in=180] (3.5,3)
    -- (4.5,3)
    to[out=0,in=90] (5,2.5)
    -- (5,1.5)
    to[out=-90,in=180] (5.5,1)
    -- (6,1);

\draw[line width=.5mm,green!80!black]
    (5,6) -- (5,3.5)
    to[out=-90,in=180] (5.5,3)
    -- (6,3);

\foreach \x/\y in {1/6,2/5,4/5,6/5}
    \draw[line width=.5mm,red,fill=white]
        (\x,\y) circle (.25);

\foreach \x/\y in {3/6,3/4,4/3,5/2,6/1}
    \draw[line width=.5mm,blue,fill=white]
        (\x,\y) circle (.25);

\foreach \x/\y in {5/6,5/4,6/3}
    \draw[line width=.5mm,green!80!black,fill=white]
        (\x,\y) circle (.25);

\foreach \y in {1,3,5}
    \foreach \x in {1,3,5}
        \node[circle,fill,inner sep=2.5pt] at (\x,\y) {};

\foreach \x in {1,3,5}
    \node at (\x,0) {\normalfont\bfseries +};
\foreach \y in {1,3,5}
    \node at (0,\y) {\normalfont\bfseries +};

\foreach \x/\c in {1/2,3/1,5/0}
    \node at (\x,6.5) {$\c$};
\foreach \y/\r in {1/3,3/2,5/1}
    \node at (-.5,\y) {$\r$};

\end{tikzpicture}}}
\]

\end{example}

\begin{rem}
    There is a concern on admissibility of states once overlapping is allowed in bosonic models. For example, the first configuration below is not legal in a closed bosonic model. However, after shifting it to the fermionic case, it becomes a legal configuration:
\[
\vcenter{\hbox{
  \begin{tikzpicture}
    \draw (6,5)--(8,5);
    \draw (7,4)--(7,6);
    \draw[line width=3pt,red] (6,5)--(7.1,5)--(7.1,4);
    \draw[line width=3pt,blue] (7,6)--(7,4);
  \end{tikzpicture}
}}
\qquad \longleftrightarrow \qquad
\vcenter{\hbox{
  \begin{tikzpicture}
    \draw (5,5)--(8,5);
    \foreach \x in {6,7}
      \draw (\x,4)--(\x,6);
    \draw[line width=3pt,red] (5,5)--(6,5)--(6,4);
    \draw[line width=3pt,blue] (6,6)--(6,5)--(7,5)--(7,4);
  \end{tikzpicture}
}}
\]
In fact, the point is that we would never see the first configuration in our work. Given an underlying Gelfand-Tsetlin pattern, there is a unique associated bosonic open state, which can be translated to a unique fermionic open state. Then the recoloring process in the fermionic case generates a legal closed state, and corresponds uniquely to a bosonic closed states if the boundary flag is compatible. In such a process, the local fermionic configuration above corresponds to the figure below, which is admissible in closed bosonic models.
\[
\begin{tikzpicture}
  \draw (6,5)--(8,5);
  \draw (7,4)--(7,6);
  \draw[line width=3pt,red] (6,5)--(6.9,5)--(6.9,4);
  \draw[line width=3pt,blue] (7,6)--(7,4);
\end{tikzpicture}
\]

When $t$ is set to zero, it is exactly those states with small colors crossing large colors horizontally that will become illegal. Those states correspond to fermionic states with $b_1$ patterns, which are also illegal in fermionic models.
\end{rem}

\bibliographystyle{alpha}
\bibliography{ref}

\end{document}